\documentclass[a4paper,reqno]{amsart}
\usepackage{amssymb, amsmath, amsthm, mathrsfs, ascmac,color}
\usepackage[pdftex]{graphicx} 
\usepackage[subrefformat=parens]{subcaption}
\usepackage[top=2.5cm, bottom=2.5cm, left=3cm,right=3cm]{geometry}
\usepackage{comment}
\usepackage{caption}
\usepackage{here}

\includecomment{jp-text}
\excludecomment{en-text}
\usepackage[foot]{amsaddr}
\newtheoremstyle{mystyle}
  {}
  {}
  {\normalfont}
  { }
  {\bfseries}
  {}
  {10pt}
  { }
\theoremstyle{mystyle}
\newtheorem{thm}{Theorem}[section]
\newtheorem{prop}[thm]{Proposition}
\newtheorem{lem}{Lemma}[section]

\newtheorem{rmk}{Remark}
\makeatletter
\renewcommand{\theequation}{%
\thesection.\arabic{equation}}
\@addtoreset{equation}{section}
\makeatother

\newcommand{\dd}{\mathrm d}
\newcommand{\ee}{\mathrm e}
\newcommand{\EE}{\mathbb E}
\newcommand{\VV}{\mathrm{Var}}

\newcommand{\GG}{\mathscr F_{t_{i-1}}}

\newcommand{\pto}{\stackrel{p}{\longrightarrow}}
\newcommand{\dto}{\stackrel{d}{\longrightarrow}}
\newcommand{\TT}{\mathsf T}
\newcommand{\bs}{\boldsymbol}

\allowdisplaybreaks
\begin{document}
\bibliographystyle{plain}
\title
[Parameter estimation for
a linear parabolic SPDE with a small noise ] 
{
Parameter estimation for a linear parabolic SPDE model
\\
in two space dimensions with a small noise
}
\date{}
\author{Yozo Tonaki$^1$}
\author{Yusuke Kaino$^2$}
\author{Masayuki Uchida$^{1,3}$}
\address{
$^1$Graduate School of Engineering Science, Osaka University
}
\address{
$^2$Graduate School of Maritime Sciences, Kobe University
}
\address{
$^3$Center for Mathematical Modeling and Date Science (MMDS), Osaka University
and JST CREST
}


\keywords{
Adaptive estimation,
high frequency data,
small noise,
stochastic partial differential equations in two space dimensions, 
$Q$-Wiener process
}

\maketitle

\begin{abstract}
We study parameter estimation for a linear parabolic second-order 
stochastic partial differential equation (SPDE) 
in two space dimensions with a small dispersion parameter
using high frequency data with respect to time and space.
We set two types of $Q$-Wiener processes as a driving noise.
We provide minimum contrast estimators of the coefficient parameters of the SPDE 
appearing in the coordinate process of the SPDE based on the thinned data in space, 
and approximate the coordinate process based on the thinned data in time.
Moreover, we propose an estimator of the drift parameter 
using the fact that the coordinate process is 
the Ornstein-Uhlenbeck process and statistical inference for diffusion processes
with a small noise.
\end{abstract}

\section{Introduction}\label{sec1}
We deal with the following linear parabolic 
stochastic partial differential equation (SPDE)
in two space dimensions
\begin{align}
\dd X_t^{Q}(y,z)
&=\biggl\{
\theta_2
\biggl(
\frac{\partial^2}{\partial y^2}+\frac{\partial^2}{\partial z^2}
\biggr)
+\theta_1\frac{\partial}{\partial y}
+\eta_1\frac{\partial}{\partial z}
+\theta_0
\biggr\}
X_t^{Q}(y,z)
\dd t
\nonumber
\\
&\qquad+\epsilon\dd W_t^Q(y,z),
\quad
(t,y,z)\in[0,1]\times D,
\label{small_2d_spde}
\\
X_0^{Q}(y,z)
&=\xi(y,z),
\quad (y,z)\in D,
\nonumber
\\
X_t^{Q}(y,z)&=0,
\quad (t,y,z)\in [0,1]\times \partial D,
\nonumber
\end{align}
where $D=[0,1]^2$, 
$\epsilon\in(0,1]$ is a known small dispersion parameter,
$W_t^Q$ is a $Q$-Wiener process in a Sobolev space on $D$,
an initial value $\xi$ is independent of $W_t^Q$,
$\theta=(\theta_0,\theta_1,\eta_1,\theta_2)$ is an unknown parameter
and $(\theta_0,\theta_1,\eta_1,\theta_2)\in\mathbb R^3\times(0,\infty)$.
Moreover, the parameter space $\Theta$ is a compact convex subset of 
$\mathbb R^3\times(0,\infty)$, 
$\theta^*=(\theta_0^*,\theta_1^*,\eta_1^*,\theta_2^*)$ is the true value of $\theta$ 
and we assume that $\theta^*\in \mathrm{Int}\,\Theta$.
The data are discrete observations 
$\{X_{t_i}^{Q}(y_{j_1},z_{j_2})\}$, 
$i=0,\ldots,N$, $j_1=0,\ldots,M_1$, $j_2=0,\ldots,M_2$, $M=M_1M_2$
with $t_i=i/N$, $y_{j_1}=j_1/M_1$ and $z_{j_2}=j_2/M_2$. 

SPDEs have been applied in various fields such as physics, engineering, and economics.
For instance, 
a stochastic heat equation is a family of our model, a linear parabolic SPDE,
and is a basic and important model that appears in many situations.
For application of linear parabolic SPDEs, 
see Piterbarg and Ostrovskii \cite{Piterbarg_Ostrovskii1997}, 
which dealt with sea surface temperature variability.

Statistical inference for SPDE models has been developed by many researchers.
For an overview of existing theories, see 
Lototsky \cite{Lototsky2009} and Cialenco \cite{Cialenco2018}.
As for discrete observations, 
see Markussen \cite{Markussen2003},
Bibinger and Trabs \cite{Bibinger_Trabs2020},
Chong  \cite{Chong2020}, \cite{Chong2019arXiv},
Cialenco et al. \cite{Cialenco_etal2020},
Cialenco and Huang \cite{Cialenco_Huang2020},
Hildebrandt \cite{Hildebrandt2020}, 
Kaino and Uchida \cite{Kaino_Uchida2020}, \cite{Kaino_Uchida2021},
Hildebrandt and Trabs 
\cite{Hildebrandt_Trabs2021}, \cite{Hildebrandt_Trabs2021arXiv},
Tonaki et al. \cite{TKU2022arXiv} 
and references therein.
Recently, Kaino and Uchida \cite{Kaino_Uchida2021} 
considered the following linear parabolic SPDE model in one space dimension 
\begin{align}
\dd X_t(y)
&=\biggl(
\theta_2\frac{\partial^2}{\partial y^2}
+\theta_1\frac{\partial}{\partial y}
+\theta_0
\biggr)
X_t(y)
\dd t
+\epsilon\dd B_t(y),
\quad
(t,y)\in[0,T]\times[0,1],
\label{small_1d_spde}
\\
X_0(y)
&=\xi(y),
\quad y\in [0,1],
\qquad
X_t(0)=X_t(1)=0,
\quad t\in [0,T],
\nonumber
\end{align}
where $\epsilon\in(0,1]$ is a known small dispersion parameter,
$T>0$, $B_t$ is a cylindrical Brownian motion in a Sobolev space on $[0,1]$,
$\xi$ is an initial value 
and $\theta_0,\theta_1,\theta_2$ are unknown parameters.
They proposed the adaptive maximum likelihood type estimation for 
the coefficient parameters $\theta_0$, $\theta_1$ and $\theta_2$,
and then showed that 
the estimators of $\theta_0$, $\theta_1$ and $\theta_2$ are asymptotically normal.
Tonaki et al. \cite{TKU2022arXiv} studied the following linear parabolic SPDE
in two space dimensions
\begin{align}
\dd \bar{X}_t^Q(y,z)
&=\biggl\{
\theta_2
\biggl(
\frac{\partial^2}{\partial y^2}+\frac{\partial^2}{\partial z^2}
\biggr)
+\theta_1\frac{\partial}{\partial y}
+\eta_1\frac{\partial}{\partial z}
+\theta_0
\biggr\}
\bar{X}_t^Q(y,z)
\dd t
\nonumber
\\
&\qquad+\sigma\dd W_t^Q(y,z),
\quad
(t,y,z)\in[0,1]\times D,
\label{2d_spde}
\\
\bar{X}_0^Q(y,z)
&=\xi(y,z),
\quad (y,z)\in D,
\nonumber
\\
\bar{X}_t^Q(y,z)&=0,
\quad (t,y,z)\in [0,1]\times \partial D,
\nonumber
\end{align}
where $D=[0,1]^2$, 
$W_t^Q$ is a $Q$-Wiener process in a Sobolev space on $D$,
$\xi$ is an initial value, 
$(\theta_0,\theta_1,\eta_1,\theta_2)$ and $\sigma$ are unknown parameters
and $(\theta_0,\theta_1,\eta_1,\theta_2,\sigma)\in\mathbb R^3\times(0,\infty)^2$.
Since a mild solution $\bar{X}_t^I$ of the SPDE \eqref{2d_spde} 
driven by a cylindrical Brownian motion $B_t=W_t^I$ 
($I$ is the identity operator) 
is not square integrable for a.e. $(t,y,z)\in[0,1]\times D$ 
(see Remark 1 in \cite{TKU2022arXiv}), 
they considered two types of $Q$-Wiener processes given 
by \eqref{QWp_ver1} and \eqref{QWp_ver2} below.
They showed consistency and asymptotic normality 
for the estimators of $(\theta_0,\theta_1,\eta_1,\theta_2,\sigma^2)$ 
when the driving noise is a $Q_1$-Wiener process given by \eqref{QWp_ver1}, 
and for the estimators of $(\theta_1,\eta_1,\theta_2,\sigma^2)$ 
when the driving noise is a $Q_2$-Wiener process defined by \eqref{QWp_ver2},
respectively.
For parameter estimation of SPDEs driven by a $Q$-Wiener process,
see H\"{u}bner et al. \cite{Hubner_etal1993}
and Cialenco and Glatt-Holtz \cite{Cialenco_Glatt-Holtz2011}.
Refer to 
Lord et al. \cite{Lord_etal2014},
Da Prato and Zabczyk \cite{DaPrato_Zabczyk2014}
and
Lototsky and Rozovsky \cite{Lototsky_Rozovsky2017}
for the $Q$-Wiener process and the mild solution of SPDEs.

In this paper, 
we apply the estimation method for the coefficient parameters in
the SPDE \eqref{small_1d_spde} proposed by Kaino and Uchida \cite{Kaino_Uchida2021}
to the SPDE \eqref{small_2d_spde} driven by a $Q$-Wiener process based on 
Tonaki et al. \cite{TKU2022arXiv}.
In other words, we consider
adaptive estimation of
the SPDE \eqref{small_2d_spde} with a small noise 
driven by two types of $Q$-Wiener processes 
defined by \eqref{QWp_ver1} and \eqref{QWp_ver2}. 
For adaptive estimation of stochastic differential equations, 
see Yoshida  \cite{Yoshida1992} and Uchuda and Yoshida \cite{Uchida_Yoshida2012, Uchida_Yoshida2014}.
Since the coordinate process of the SPDE \eqref{small_2d_spde} 
is a diffusion process, 
we derive an estimator of  $\theta_0$
based on statistical inference for diffusion processes with a small noise
in an analogous manner to Kaino and Uchida \cite{Kaino_Uchida2021}.
For statistical inference for diffusion processes 
with a small noise based on discrete observations,
see 
Genon-Catalot \cite{Genon-Catalot1990},
Laredo \cite{Laredo1990},
S{\o}rensen and Uchida \cite{Sorensen_Uchida2003},
Uchida \cite{Uchida2003}, \cite{Uchida2004},
Gloter and S{\o}rensen \cite{Gloter_Sorensen2009},
Guy et al. \cite{Guy_etal2014},
Nomura and Uchida \cite{Nomura_Uchida2016},
Kaino and Uchida \cite{Kaino_Uchida2018}
and Kawai and Uchida \cite{Kawai_Uchida2022}.

This paper is organized as follows.
In Section \ref{sec2}, we give the setting of our model. 
In Section \ref{sec3}, 
we propose estimators of the coefficient parameters 
$\theta_1$, $\eta_1$, $\theta_2$ and $\theta_0$ in the SPDE \eqref{small_2d_spde}
driven by two types of $Q$-Wiener processes 
and show the asymptotic properties of these estimators.
Section \ref{sec5} is devoted to the proofs of the results 
in Section \ref{sec3}.
Finally, we treat parameter estimation based on the exact likelihood
of the one dimensional Ornstein-Uhlenbeck process with a small noise 
appearing as the coordinate process of the SPDE \eqref{small_2d_spde} in Appendix I.
In order to illustrate the properties  of the parameters 
in the SPDE \eqref{small_2d_spde},
we show the sample paths with different values of the 
parameters in Appendix II.

\section{Preliminaries}\label{sec2}
Let $(\Omega,\mathscr F, \{{\mathscr F}_t\}_{t\ge0}, P)$ 
be a stochastic basis with usual conditions,
and let $\{w_{k,\ell}\}_{k,\ell\in\mathbb N}$
be independent real valued standard Brownian motions on this basis.

By setting the differential operator $A_\theta$ by 
\begin{equation*}
-A_\theta
=\theta_2
\biggl(
\frac{\partial^2}{\partial y^2}+\frac{\partial^2}{\partial z^2}
\biggr)
+\theta_1\frac{\partial}{\partial y}
+\eta_1\frac{\partial}{\partial z}
+\theta_0,
\end{equation*}
the SPDE \eqref{small_2d_spde} is expressed as
\begin{equation*}
\dd X_t^{Q}(y,z)=-A_\theta X_t^{Q}(y,z)\dd t+\epsilon\dd W_t^Q(y,z),
\end{equation*}
and it follows that 
$A_\theta e_{k,\ell}=\lambda_{k,\ell}e_{k,\ell}$ for $k,\ell\in\mathbb N$,
where the eigenfunctions $e_{k,\ell}$ of $A_\theta$ and 
the corresponding eigenvalues $\lambda_{k,\ell}$ are given by 
\begin{align*}
e_{k,\ell}(y,z)
&=2\sin(\pi k y)\sin(\pi \ell z)
\ee^{-\frac{\theta_1}{2\theta_2}y}\ee^{-\frac{\eta_1}{2\theta_2}z},
\quad (y,z)\in D,
\\
\lambda_{k,\ell}
&=-\theta_0+\frac{\theta_1^2+\eta_1^2}{4\theta_2}+\pi^2(k^2+\ell^2)\theta_2.
\end{align*}
We set
$H_\theta
=\{f:D\to\mathbb R|\, \|f\|_\theta<\infty \text{ and } 
f(y,z)=0\text{ for } (y,z)\in \partial D  \}$
with 
\begin{equation*}
\langle f,g\rangle_\theta
=\int_0^1\int_0^1
f(y,z)g(y,z)
\ee^{\frac{\theta_1}{\theta_2}y}\ee^{\frac{\eta_1}{\theta_2}z} 
\dd y\dd z,
\quad
\|f\|_\theta=\langle f,f\rangle_\theta^{1/2}.
\end{equation*}

We introduce two types of $Q$-Wiener processes defined as follows.
\begin{align}
\langle W_t^{Q_1},f\rangle_\theta
&=\sum_{k,\ell\ge1}\lambda_{k,\ell}^{-\alpha/2}
\langle f,e_{k,\ell}\rangle_\theta w_{k,\ell}(t), 
\label{QWp_ver1}
\\
\langle W_t^{Q_2},f\rangle_\theta
&=\sum_{k,\ell\ge1}\mu_{k,\ell}^{-\alpha/2}
\langle f,e_{k,\ell}\rangle_\theta w_{k,\ell}(t)
\label{QWp_ver2}
\end{align}
for $f\in H_\theta$ and $t\ge0$, 
where 
$\mu_{k,\ell}=\pi^2(k^2+\ell^2)+\mu_0$, $\mu_0\in(-2\pi^2,\infty)$
and $\alpha\in(0,1)$.
$\mu_0$ is an unknown parameter (may be known),
the parameter space of $\mu_0$ is a compact convex subset of 
$(-2\pi^2,\infty)$ and the true value $\mu_0^*$ belongs to its interior.
$\alpha$ is known and its restriction guarantees 
that $Q_1$ and $Q_2$-Wiener processes are well-defined in a Hilbert space
and that the parameters are estimable,
see Remarks 1 and 4 in Tonaki et al. \cite{TKU2022arXiv}. 
Note that the $Q_1$-Wiener process defined by \eqref{QWp_ver1} is introduced 
as a driving noise with the damping factor 
based on the eigenvalue $\lambda_{k,\ell}$ of $A_\theta$ 
corresponding to $e_{k,\ell}$, and the $Q_2$-Wiener process is constructed 
as a driving noise with the damping factor 
which does not include the parameter $\theta$ based on the $Q_1$-Wiener process.
Specifically, by choosing $Q_1$ as the covariance operator defined on
the domain $\mathscr D(A_{\theta}^{-1/2}) \supset H_\theta$ 
with inner product 
\begin{equation*}
\langle u,v \rangle_{\theta,-1/2} 
= \langle A_\theta^{-1/2} u, A_\theta^{-1/2} v \rangle_{\theta}
\end{equation*}
and its corresponding induced norm $\|u\|_{\theta,-1/2}=\|A_\theta^{-1/2}u\|_\theta$ 
such that
\begin{equation*}
Q_1 u = 
\sum_{k,\ell \ge 1} \lambda_{k,\ell}^{-(1+\alpha)} 
\langle u, v_{k,\ell} \rangle_{\theta,-1/2} v_{k,\ell}
\end{equation*}
for $u = \sum_{k,\ell\ge1} \langle u, v_{k,\ell} \rangle_{\theta,-1/2} v_{k,\ell} 
\in \mathscr D(A_{\theta}^{-1/2})$,
$v_{k,\ell}=e_{k,\ell}/\|e_{k,\ell}\|_{\theta,-1/2}$ and $\alpha>0$,
\eqref{QWp_ver1} is obtained.
The same is true for \eqref{QWp_ver2}.
See Remarks 1 and 2 in Tonaki et al. \cite{TKU2022arXiv}.

We assume that $\xi\in H_\theta$ and 
$\lambda_{1,1}^*=-\theta_0^*
+\frac{(\theta_1^*)^2+(\eta_1^*)^2}{4\theta_2^*}+2\pi^2\theta_2^*>0$.
$X_t^{Q}$ is called a mild solution of \eqref{small_2d_spde} on $D$ 
if it satisfies that for any $t\in[0,1]$, 
\begin{equation*}
X_t^{Q}=
\ee^{-tA_\theta}\xi+\epsilon\int_0^t\ee^{-(t-s)A_\theta}\dd W_s^Q
\quad \mathrm{a.s.},
\end{equation*}
where
$\ee^{-tA_\theta}u=\sum_{k,\ell\ge1}\ee^{-\lambda_{k,\ell}t}
\langle u,e_{k,\ell}\rangle_\theta e_{k,\ell}$ for $u\in H_\theta$.
By defining the $Q_1$-Wiener process $W_t^{Q_1}$ in \eqref{QWp_ver1}, 
the random field $X_t^{Q_1}(y,z)$ is spectrally decomposed as
\begin{equation*}
X_t^{Q_1}(y,z)
=\sum_{k,\ell\ge1}x_{k,\ell}^{Q_1}(t)e_{k,\ell}(y,z),
\end{equation*}
where the coordinate process 
\begin{equation}\label{small_cp1_ver1}
x_{k,\ell}^{Q_1}(t)
=\langle X_t^{Q_1},e_{k,\ell}\rangle_\theta
=\ee^{-\lambda_{k,\ell}t}\langle \xi,e_{k,\ell}\rangle_\theta
+\epsilon\int_0^t\lambda_{k,\ell}^{-\alpha/2}
\ee^{-\lambda_{k,\ell}(t-s)}\dd w_{k,\ell}(s)
\end{equation}
is the Ornstein-Uhlenbeck process 
which satisfies the stochastic differential equation
with small dispersion parameter
\begin{equation}\label{small_dp_ver1}
\dd x_{k,\ell}^{Q_1}(t)
=-\lambda_{k,\ell} x_{k,\ell}^{Q_1}(t)\dd t
+\epsilon\lambda_{k,\ell}^{-\alpha/2}\dd w_{k,\ell}(t),
\quad
x_{k,\ell}^{Q_1}(0)=\langle \xi,e_{k,\ell}\rangle_\theta
\end{equation}
and can be expressed by using the random field $X_t^{Q_1}(y,z)$ as
\begin{equation}\label{small_cp2_ver1}
x_{k,\ell}^{Q_1}(t)
=
2\int_0^1\int_0^1
X_t^{Q_1}(y,z)\sin(\pi k y)\sin(\pi \ell z)
\ee^{\frac{\theta_1}{2\theta_2}y}\ee^{\frac{\eta_1}{2\theta_2}z}
\dd y\dd z.
\end{equation}
Similarly, by setting the $Q_2$-Wiener process $W_t^{Q_2}$ in \eqref{QWp_ver2}, 
the random field $X_t^{Q_2}(y,z)$ is represented as
\begin{equation*}
X_t^{Q_2}(y,z)
=\sum_{k,\ell\ge1}x_{k,\ell}^{Q_2}(t)e_{k,\ell}(y,z),
\end{equation*}
where the coordinate process 
\begin{equation}\label{small_cp1_ver2}
x_{k,\ell}^{Q_2}(t)
=\langle X_t^{Q_2},e_{k,\ell}\rangle_\theta
=\ee^{-\lambda_{k,\ell}t}\langle \xi,e_{k,\ell}\rangle_\theta
+\epsilon\int_0^t\mu_{k,\ell}^{-\alpha/2}
\ee^{-\lambda_{k,\ell}(t-s)}\dd w_{k,\ell}(s)
\end{equation}
is a diffusion process defined by the stochastic differential equation 
\begin{equation}\label{small_dp_ver2}
\dd x_{k,\ell}^{Q_2}(t)
=-\lambda_{k,\ell} x_{k,\ell}^{Q_2}(t)\dd t
+\epsilon\mu_{k,\ell}^{-\alpha/2}\dd w_{k,\ell}(t),
\quad
x_{k,\ell}^{Q_2}(0)=\langle \xi,e_{k,\ell}\rangle_\theta,
\end{equation}
and is also given by
\begin{equation}\label{small_cp2_ver2}
x_{k,\ell}^{Q_2}(t)
=2\int_0^1\int_0^1
X_t^{Q_2}(y,z)\sin(\pi k y)\sin(\pi \ell z)
\ee^{\frac{\theta_1}{2\theta_2}y}\ee^{\frac{\eta_1}{2\theta_2}z}
\dd y\dd z.
\end{equation}

We assume the following condition of the initial value $\xi\in H_\theta$. 
\begin{enumerate}
\item[\textbf{[A1]}]
The initial value $\xi$ is non-random, 
$\langle \xi,e_{1,1}\rangle_\theta\neq0$ and
$\|A_\theta\xi\|_\theta^2<\infty$. 
\end{enumerate}
Note that if [A1] is satisfied, 
then Assumption 1 in Tonaki et al. \cite{TKU2022arXiv} holds. 
By setting the $Q$-Wiener process to \eqref{QWp_ver1} or \eqref{QWp_ver2}, 
there exists a unique mild solution $X_t^Q$ of the SPDE \eqref{small_2d_spde} such that
$\sup_{t\in[0,1]}\EE[\|X_t^Q\|_\theta^2]<\infty$
under [A1] and $\lambda_{1,1}^*>0$. See Remark 3 in \cite{TKU2022arXiv}.

We treat thinned data with respect to space or time 
to estimate the coefficient parameters.
Set $\overline m_1\le M_1$ and $\overline m_2\le M_2$
such that $\overline m:=\overline m_1\overline m_2=O(N^\rho)$
for some $0<\rho<1\land2(1-\alpha)$, and let
\begin{equation*}
\overline y_{j_1}
=\biggl\lfloor\frac{M_1}{\overline m_1}\biggr\rfloor\frac{j_1}{M_1},
\quad
\overline z_{j_2}
=\biggl\lfloor\frac{M_2}{\overline m_2}\biggr\rfloor\frac{j_2}{M_2}
\end{equation*}
for $j_1=0,\ldots,\overline m_1$ and $j_2=0,\ldots,\overline m_2$.
For $\delta\in(0,1/2)$, 
there exist $J_1,J_2\ge1$, $m_1,m_2\ge1$ such that 
\begin{align*}
&\overline y_{J_1}<\delta \le \overline y_{J_1+1} < \cdots 
< \overline y_{J_1+m_1}\le 1-\delta < \overline y_{J_1+m_1+1},
\\
&\overline z_{J_2}<\delta \le \overline z_{J_2+1} < \cdots 
< \overline z_{J_2+m_2}\le 1-\delta < \overline z_{J_2+m_2+1},
\end{align*}
and let 
\begin{equation*}
\widetilde y_{j_1}
=\overline y_{J_1+j_1}
=\biggl\lfloor\frac{M_1}{\overline m_1}\biggr\rfloor\frac{J_1+j_1}{M_1},
\quad
\widetilde z_{j_2}
=\overline z_{J_2+j_2}
=\biggl\lfloor\frac{M_2}{\overline m_2}\biggr\rfloor\frac{J_2+j_2}{M_2},
\end{equation*}
$j_1=1,\ldots, m_1$ and $j_2=1,\ldots, m_2$,
and $D_\delta=[\delta,1-\delta]^2\subset D$.
Furthermore, let $n\le N$, 
$\widetilde t_i=
\lfloor\frac{N}{n}\rfloor
\frac{i}{N}$, $i=0,\ldots, n$ and $\Delta_n=\widetilde t_i-\widetilde t_{i-1}$. 
Note that $m:=m_1m_2=O(N^\rho)$,
$(\widetilde y_{j_1},\widetilde z_{j_2})\in D_\delta$ 
for any $j_1=1,\ldots, m_1$ and $j_2=1,\ldots, m_2$, 
and $\Delta_N=t_i-t_{i-1}$.

\section{Main results}\label{sec3}

\subsection{SPDE driven by $Q_1$-Wiener process}\label{sec3.1}
In this subsection, we consider estimation for  
the coefficient parameter $\theta=(\theta_0,\theta_1,\eta_1,\theta_2)$
in the SPDE \eqref{small_2d_spde}
driven by the $Q_1$-Wiener process defined as \eqref{QWp_ver1}. 
The flow of the parameter estimation method is as follows.
We first introduce the minimum contrast estimators 
of the coefficient parameters $\theta_1$, $\eta_1$ and $\theta_2$ 
using the thinned data with respect to space.
Next, we construct the approximate coordinate process 
utilizing these minimum contrast estimators, 
and provide the adaptive ML type estimators of $\theta_0$ 
based on the thinned data in time.

Let $\Delta_i X^{Q}(y,z)
=X_{t_i}^{Q}(y,z)-X_{t_{i-1}}^{Q}(y,z)$ and
$\Gamma(s)=\int_0^\infty x^{s-1}\ee^{-x}\dd x$ $(s>0)$.
The following proposition holds as in Tonaki et al. \cite{TKU2022arXiv}.
\begin{prop}\label{prop1}
Under [A1], it holds that uniformly in $(y,z)\in D_\delta$,  
\begin{equation*}
\EE\bigl[(\Delta_i X^{Q_1})^2(y,z)\bigr]
=\epsilon^2
\biggl\{
\Delta_N^\alpha\frac{\Gamma(1-\alpha)}{4\pi\alpha\theta_2}
\ee^{-\frac{\theta_1}{\theta_2}y}\ee^{-\frac{\eta_1}{\theta_2}z}
+O(\Delta_N)
\biggr\}
+r_{N,i},
\end{equation*}
where $\sum_{i=1}^N |r_{N,i}|=O(\Delta_N^{\beta})$
for any $\beta\in(0,1)$, and thus
\begin{equation}\label{prop1-eq2}
\EE\biggl[
\frac{\epsilon^{-2}}{N\Delta_N^\alpha}\sum_{i=1}^N(\Delta_i X^{Q_1})^2(y,z)
\biggr]
=\frac{\Gamma(1-\alpha)}{4\pi\alpha\theta_2}
\ee^{-\frac{\theta_1}{\theta_2}y}\ee^{-\frac{\eta_1}{\theta_2}z}
+O(\Delta_N^{1-\alpha} \lor \epsilon^{-2}\Delta_N^{1-\alpha+\beta}).
\end{equation}
\end{prop}

We make the following condition. 
\begin{enumerate}
\item[\textbf{[A2]}]
There exists $\beta\in[\alpha,1)$ such that 
$\epsilon^2 N^{(1-\alpha+\beta)\land(5/2-2\alpha)}\to\infty$
as $N\to\infty$ and $\epsilon\to0$.
\end{enumerate}
Note that the remainder term in \eqref{prop1-eq2} 
can be asymptotically ignored under [A2]. 
Let
\begin{equation*}
Z_N^{Q}(y,z)
=\frac{1}{N\Delta_N^\alpha}\sum_{i=1}^N(\Delta_i X^{Q})^2(y,z),
\end{equation*}
and define the contrast function as follows. 
\begin{equation*}
U_{N,m,\epsilon}^{(1)}(\theta_1,\eta_1,\theta_2)
=\sum_{j_1=1}^{m_1}\sum_{j_2=1}^{m_2}
\biggl\{\epsilon^{-2}Z_N^{Q_1}(\widetilde y_{j_1},\widetilde z_{j_2})
-\frac{\Gamma(1-\alpha)}{4\pi\alpha\theta_2}
\exp\biggl(
-\frac{\theta_1}{\theta_2} \widetilde y_{j_1}\biggr)
\exp\biggl(
-\frac{\eta_1}{\theta_2} \widetilde z_{j_2}
\biggr)
\biggr\}^2.
\end{equation*}
Let $\hat \theta_1$, $\hat \eta_1$ and $\hat \theta_2$ be minimum contrast estimators 
defined as 
\begin{equation*}
(\hat \theta_1,\hat \eta_1, \hat \theta_2)
=\underset{\theta_1,\eta_1,\theta_2}
{\mathrm{arginf}}\, U_{N,m,\epsilon}^{(1)}(\theta_1,\eta_1,\theta_2).
\end{equation*}
We set
$R_{N,\epsilon}
=N^{1\land2(1-\alpha)} \land \epsilon^2 N^{(1-\alpha+\beta)\land(5/2-2\alpha)}
\land \epsilon^4 N^{2(1-\alpha+\beta)}$. 

\begin{thm}\label{th1}
Under [A1] and [A2], it holds that 
for any $\gamma>0$ with $mN^{2\gamma}R_{N,\epsilon}^{-1}\to0$, 
as $N\to\infty$, $m\to\infty$ and $\epsilon\to0$, 
\begin{equation}\label{th1-eq1}
m^{1/2}N^\gamma
\begin{pmatrix}
\hat\theta_1-\theta_1^*
\\
\hat\eta_1-\eta_1^*
\\
\hat\theta_2-\theta_2^*
\end{pmatrix}
\pto0.
\end{equation}
\end{thm}

\begin{rmk}
Theorem \ref{th1} is the same result 
as Theorem 3.2 in Tonaki et al. \cite{TKU2022arXiv}. 
Indeed, if $\sigma$ is known in the SPDE \eqref{2d_spde}, 
then the assertion of Theorem 3.2 in \cite{TKU2022arXiv} holds 
for the coefficient parameters $\theta_1$, $\eta_1$ and $\theta_2$
instead of $s=\sigma^2/\theta_2$, $\kappa=\theta_1/\theta_2$ and $\eta=\eta_1/\theta_2$.
Theorem \ref{th1} shows that the estimators 
$\hat\theta_1$, $\hat\eta_1$ and $\hat\theta_2$ have $m^{1/2}N^\gamma$-consistency
and $m^{1/2}N^\gamma=o(R_{N,\epsilon}^{1/2})$.
Moreover, \eqref{th1-eq1} can be regarded as that 
for $\gamma'>0$ with $N^{\gamma'}R_{N,\epsilon}^{-1/2}\to0$,
\begin{equation*}
N^{\gamma'}
\begin{pmatrix}
\hat\theta_1-\theta_1^*
\\
\hat\eta_1-\eta_1^*
\\
\hat\theta_2-\theta_2^*
\end{pmatrix}
\pto0.
\end{equation*}
\end{rmk}
We make the following assumption. 
\begin{enumerate}
\item[\textbf{[A3]}]
There exists $\gamma>0$ such that $mN^{2\gamma}R_{N,\epsilon}^{-1}\to0$
as $N\to\infty$, $m\to\infty$ and $\epsilon\to0$.
\end{enumerate}
Note that under [A1]-[A3], \eqref{th1-eq1} holds with $\gamma>0$ in [A3].

By using the estimators $\hat\theta_1$, $\hat\eta_1$ and $\hat\theta_2$,
the approximate coordinate process of the coordinate process in \eqref{small_cp1_ver1}
is constructed as follows.
\begin{equation*}
\hat x_{k,\ell}^{Q_1}(\widetilde t_i)
=\frac{2}{M}\sum_{j_1=1}^{M_1}\sum_{j_2=1}^{M_2}
X_{\widetilde t_i}^{Q_1}(y_{j_1},z_{j_2})
\sin(\pi k y_{j_1})\sin(\pi \ell z_{j_2})
\exp\biggl(\frac{\hat\theta_1}{2\hat\theta_2}y_{j_1}\biggr)
\exp\biggl(\frac{\hat\eta_1}{2\hat\theta_2}z_{j_2}\biggr)
\end{equation*}
for $i=1,\ldots,n$. 
The estimator of $\theta_0$ is obtained  
by utilizing this approximate coordinate process 
$\{\hat x_{k,\ell}^{Q_1}(\widetilde t_i)\}_{i=1}^n$ 
and statistical inference for diffusion processes with a small dispersion parameter. 

\begin{en-text}
We construct the approximate coordinate process  
\begin{equation*}
\hat x_{k,\ell}^{Q_1}(\widetilde t_i)
=\frac{2}{M}\sum_{j_1=1}^{M_1}\sum_{j_2=1}^{M_2}
X_{\widetilde t_i}^{Q_1}(y_{j_1},z_{j_2})
\sin(\pi k y_{j_1})\sin(\pi \ell z_{j_2})
\exp\biggl(\frac{\hat\theta_1}{2\hat\theta_2}y_{j_1}\biggr)
\exp\biggl(\frac{\hat\eta_1}{2\hat\theta_2}z_{j_2}\biggr),
\end{equation*}
$i=1,\ldots,n$ of the coordinate process in \eqref{small_cp1_ver1}
by using the estimators $\hat\theta_1$, $\hat\eta_1$ and $\hat\theta_2$,
and then we consider estimation for $\theta_0$ 
by utilizing this approximate coordinate process 
$\{\hat x_{k,\ell}^{Q_1}(\widetilde t_i)\}_{i=1}^n$ 
and statistical inference for diffusion processes with small dispersion parameter. 
\end{en-text}
 
We consider the following asymptotics for $n$ and $\epsilon$. 
\begin{enumerate}
\item[\textbf{[B1]}]
$\lim_{n\to\infty,\epsilon\to0}n\epsilon^2=0$.

\item[\textbf{[B2]}]
$\varlimsup_{n\to\infty,\epsilon\to0}(n\epsilon^2)^{-1}<\infty$, that is,

\noindent
(I) $\lim_{n\to\infty,\epsilon\to0}(n\epsilon^2)^{-1}=0$, or
(II) $0<\lim_{n\to\infty,\epsilon\to0}(n\epsilon^2)^{-1}<\infty$.

\end{enumerate}

\begin{rmk}\label{rmk2}
[B1] and [B2] are the conditions corresponding to 
[B1]-[B3] in Uchida \cite{Uchida2003}.
\cite{Uchida2003} considered 
the contrast function based on the Euler-Maruyama approximation, 
and hence imposed the condition that  
$\lim_{n\to\infty,\epsilon\to0}(n\epsilon)^{-1}=0$
in order to asymptotically ignore the approximation error. 
In this paper, however, we deal with the Ornstein-Uhlenbeck process 
such as \eqref{small_dp_ver1}, and consider the contrast function based on 
the explicit likelihood of the Ornstein-Uhlenbeck process, 
so that this condition can be removed. 
See Appendix for details.
\end{rmk}

The contrast function is as follows. 
\begin{equation*}
V_{n,\epsilon}^{(1)}(\lambda|x)
=\sum_{i=1}^n
\frac{(x(\widetilde t_i)-\ee^{-\lambda \Delta_n}x(\widetilde t_{i-1}))^2}
{\frac{\epsilon^2(1-\ee^{-2\lambda\Delta_n})}{2\lambda^{1+\alpha}}}
+n\log \frac{1-\ee^{-2\lambda\Delta_n}}{2\lambda^{1+\alpha}\Delta_n}.
\end{equation*}
Set
\begin{equation*}
\hat \lambda_{1,1}
=\underset{\lambda}
{\mathrm{arginf}}\, V_{n,\epsilon}^{(1)}(\lambda|\hat x_{1,1}^{Q_1})
\end{equation*}
as the adaptive ML type estimator of $\lambda_{1,1}$. 
Moreover, let
\begin{equation*}
\hat\theta_0
=-\hat\lambda_{1,1}+\frac{\hat\theta_1^2+\hat\eta_1^2}{4\hat\theta_2}
+2\pi^2\hat\theta_2
\end{equation*}
be the estimator of $\theta_0$. Define 
\begin{equation*}
G_1(\lambda)=\frac{1-\ee^{-2\lambda}}{2\lambda^{1-\alpha}}x_{1,1}^{Q_1}(0)^2,
\quad
H_1(\lambda)=\frac{\alpha^2}{2\lambda^2}.
\end{equation*}
Set $I_1(\lambda)=H_1(\lambda)+c\, G_1(\lambda)$ under [B2], 
where $c=\lim_{n\to\infty,\epsilon\to0}(n\epsilon^2)^{-1}$.

We consider the following conditions
to control the error of the approximate coordinate process.
\begin{enumerate}
\item[\textbf{[C1]}]
$\frac{n^{2-\alpha}\epsilon^2 \lor n^{1-\beta}}{mN^{2\gamma}}\to0$,
$\frac{n^{1+\tau_1}}{(M_1 \land M_2)^{2\tau_1}}\to0$
and $\frac{n^{2-\alpha+\tau_2}\epsilon^2}{(M_1 \land M_2)^{2\tau_2}}\to0$
for some $\tau_1\in[0,1)$ and $\tau_2\in[0,\alpha)$.

\item[\textbf{[C2]}]
$\frac{n^{1-\alpha}\lor (n^\beta\epsilon^2)^{-1}}{mN^{2\gamma}}\to0$,
$\frac{n^{\tau_1}\epsilon^{-2}}{(M_1 \land M_2)^{2\tau_1}}\to0$
and $\frac{n^{1-\alpha+\tau_2}}{(M_1 \land M_2)^{2\tau_2}}\to0$
for some $\tau_1\in[0,1)$ and $\tau_2\in[0,\alpha)$.

\item[\textbf{[C3]}]
$\frac{n^{2-\alpha} \lor n^{1-\beta}\epsilon^{-2}}{mN^{2\gamma}}\to0$,
$\frac{n^{1+\tau_1}\epsilon^{-2}}{(M_1 \land M_2)^{2\tau_1}}\to0$
and $\frac{n^{2-\alpha+\tau_2}}{(M_1 \land M_2)^{2\tau_2}}\to0$
for some $\tau_1\in[0,1)$ and $\tau_2\in[0,\alpha)$.

\end{enumerate}
[C1] and [C2] are the conditions for consistency of an estimator of $\theta_0$ 
under certain situations, 
and [C3] is the condition for asymptotic normality of an estimator of $\theta_0$.

\begin{thm}\label{th2}
Assume [A1]-[A3].
\begin{enumerate}
\item[(a)]
If [B1] and [C1] hold, or if [B2] and [C2] hold, 
then as $n\to\infty$ and $\epsilon\to0$,
\begin{equation*}
\hat\theta_0 \pto \theta_0^*.
\end{equation*}

\item[(b)]
\begin{enumerate}
\item[(i)]
If [B1] and [C3] hold, then as $n\to\infty$ and $\epsilon\to0$,
\begin{equation*}
\epsilon^{-1}(\hat\theta_0-\theta_0^*)
\dto N(0,G_1(\lambda_{1,1}^*)^{-1}).
\end{equation*}

\item[(ii)]
If [B2] and [C3] hold, then as $n\to\infty$ and $\epsilon\to0$,
\begin{equation*}
\sqrt n(\hat\theta_0-\theta_0^*)
\dto N(0,I_1(\lambda_{1,1}^*)^{-1}).
\end{equation*}
\end{enumerate}
\end{enumerate}
\end{thm}

\begin{rmk}
Since $I_1(\lambda)=H_1(\lambda)+c\,G_1(\lambda)$, 
$c=0$ under [B2](I) and $(n\epsilon^2)^{-1}\to c\neq0$ under [B2](II),
the assertion in Theorem \ref{th2} (b)-(ii) can be rewritten as follows.
\begin{enumerate}
\item[(1)]
If [B2](I) and [C3] hold, then as $n\to\infty$ and $\epsilon\to0$,
\begin{equation*}
\sqrt n(\hat\theta_0-\theta_0^*)
\dto N(0,H_1(\lambda_{1,1}^*)^{-1}).
\end{equation*}

\item[(2)]
If [B2](II) and [C3] hold, then as $n\to\infty$ and $\epsilon\to0$,
\begin{equation*}
\epsilon^{-1}(\hat\theta_0-\theta_0^*)
\dto N(0,c I_1(\lambda_{1,1}^*)^{-1}).
\end{equation*}
\end{enumerate}
\end{rmk}

\begin{rmk}
Theorem 3.3 in Tonaki et al. \cite{TKU2022arXiv} showed that 
when the driving noise is a $Q_1$-Wiener process, 
the estimator of $\theta_0$ has asymptotic normality with convergence rate $\sqrt n$. 
However, according to Theorem \ref{th2} (b)-(i), when [B1] $n\epsilon^2\to0$ holds, 
the estimator of $\theta_0$ has asymptotic normality with convergence rate $\epsilon^{-1}$, 
which is faster than the convergence rate 
$\sqrt n$ of Theorem 3.3 in \cite{TKU2022arXiv}. 
This means that our estimator $\hat\theta_0$ is better 
than the estimator of $\theta_0$ proposed by \cite{TKU2022arXiv}.
\end{rmk}

\subsection{SPDE driven by $Q_2$-Wiener process}\label{sec3.2}
In this subsection, we study estimation for 
the coefficient parameter $\theta=(\theta_0,\theta_1,\eta_1,\theta_2)$
in the SPDE \eqref{small_2d_spde}
driven by the $Q_2$-Wiener process defined as \eqref{QWp_ver2}. 

In a similar way to Proposition 3.4 in Tonaki et al. \cite{TKU2022arXiv}, 
the following proposition holds.
\begin{prop}\label{prop2}
Under [A1], it holds that uniformly in $(y,z)\in D_\delta$,
\begin{equation*}
\EE[(\Delta_i X^{Q_2})^2(y,z)]
=
\epsilon^2
\biggl\{
\frac{\Delta_N^\alpha\Gamma(1-\alpha)}{4\pi\alpha\theta_2^{1-\alpha}}
\ee^{-\frac{\theta_1}{\theta_2}y}\ee^{-\frac{\eta_1}{\theta_2}z}
+O(\Delta_N)
\biggr\}+r_{N,i},
\end{equation*}
where $\sum_{i=1}^N |r_{N,i}|=O(\Delta_N^\beta)$ for any $\beta\in(0,1)$, and hence
\begin{equation*}
\EE\biggl[
\frac{\epsilon^{-2}}{N\Delta_N^\alpha}\sum_{i=1}^N(\Delta_i X^{Q_2})^2(y,z)
\biggr]
=\frac{\Gamma(1-\alpha)}{4\pi\alpha\theta_2^{1-\alpha}}
\ee^{-\frac{\theta_1}{\theta_2}y}\ee^{-\frac{\eta_1}{\theta_2}z}
+O(\Delta_N^{1-\alpha} \lor \epsilon^{-2}\Delta_N^{1-\alpha+\beta}).
\end{equation*}
\end{prop}

Therefore, setting the contrast function as 
\begin{equation*}
U_{N,m,\epsilon}^{(2)}(\theta_1,\eta_1,\theta_2)
=\sum_{j_1=1}^{m_1}\sum_{j_2=1}^{m_2}
\biggl\{\epsilon^{-2}Z_N^{Q_2}(\widetilde y_{j_1},\widetilde z_{j_2})
-\frac{\Gamma(1-\alpha)}{4\pi\alpha\theta_2^{1-\alpha}}
\exp\biggl(
-\frac{\theta_1}{\theta_2} \widetilde y_{j_1}
\biggr)
\exp\biggl(
-\frac{\eta_1}{\theta_2} \widetilde z_{j_2}
\biggr)
\biggr\}^2,
\end{equation*}
and letting 
$\tilde \theta_1$, $\tilde \eta_1$ and $\tilde \theta_2$ 
be minimum contrast estimators defined as 
\begin{equation*}
(\tilde \theta_1,\tilde \eta_1,\tilde \theta_2)
=\underset{\theta_1,\eta_1,\theta_2}
{\mathrm{arginf}}\, U_{N,m,\epsilon}^{(2)}(\theta_1,\eta_1,\theta_2),
\end{equation*}
we obtain the following theorem as in Theorem \ref{th1}.
\begin{thm}\label{th3}
Under [A1] and [A2], it holds that 
for any $\gamma>0$ with $mN^{2\gamma}R_{N,\epsilon}^{-1}\to0$, 
as $N\to\infty$, $m\to\infty$ and $\epsilon\to0$, 
\begin{equation*}
m^{1/2}N^\gamma
\begin{pmatrix}
\tilde\theta_1-\theta_1^*
\\
\tilde\eta_1-\eta_1^*
\\
\tilde\theta_2-\theta_2^*
\end{pmatrix}
\pto0.
\end{equation*}
\end{thm}
The approximation of \eqref{small_cp2_ver2} is defined as follows.
\begin{equation*}
\tilde x_{k,\ell}^{Q_2}(\widetilde t_i)
=\frac{2}{M}\sum_{j_1=1}^{M_1}\sum_{j_2=1}^{M_2}
X_{\widetilde t_i}^{Q_2}(y_{j_1},z_{j_2})
\sin(\pi k y_{j_1})\sin(\pi \ell z_{j_2})
\exp\biggl(\frac{\tilde\theta_1}{2\tilde\theta_2}y_{j_1}\biggr)
\exp\biggl(\frac{\tilde\eta_1}{2\tilde\theta_2}z_{j_2}\biggr)
\end{equation*}
for $i=1,\ldots,n$.
We set the contrast function by 
\begin{equation*}
V_{n,\epsilon}^{(2)}(\lambda,\mu|x)
=\sum_{i=1}^n
\frac{(x(\widetilde t_i)-\ee^{-\lambda \Delta_n}x(\widetilde t_{i-1}))^2}
{\frac{\epsilon^2(1-\ee^{-2\lambda\Delta_n})}{2\lambda\mu^\alpha}}
+n\log \frac{1-\ee^{-2\lambda\Delta_n}}{2\lambda\mu^\alpha\Delta_n},
\end{equation*}
and if $\mu_0$ is known, then 
$\mu_{1,1}=2 \pi^2 + \mu_0$ is known and
let
\begin{equation*}
\tilde \lambda_{1,1}
=\underset{\lambda}
{\mathrm{arginf}}\, V_{n,\epsilon}^{(2)}(\lambda,\mu_{1,1}|
\tilde x_{1,1}^{Q_2})
\end{equation*}
as the adaptive ML type estimator of $\lambda_{1,1}$, 
or if $\mu_0$ is unknown, then $\mu_{1,1}=2 \pi^2 + \mu_0$ is unknown and
let
\begin{equation*}
(\tilde \lambda_{1,1}, \tilde \mu_{1,1})
=\underset{\lambda,\mu}
{\mathrm{arginf}}\, V_{n,\epsilon}^{(2)}(\lambda,\mu|
\tilde x_{1,1}^{Q_2})
\end{equation*}
as the adaptive ML type estimator of $(\lambda_{1,1},\mu_{1,1})$. 
Moreover, let
\begin{equation*}
\tilde\theta_0
=-\tilde\lambda_{1,1}+\frac{\tilde\theta_1^2+\tilde\eta_1^2}{4\tilde\theta_2}
+2\pi^2\tilde\theta_2,
\quad
\tilde\mu_0=\tilde\mu_{1,1}-2\pi^2,
\end{equation*}
\begin{equation*}
G_2(\lambda,\mu)=\frac{1-\ee^{-2\lambda}}{2\lambda}\mu^\alpha x_{1,1}^{Q_2}(0)^2,
\quad
H_2(\mu)=\frac{\alpha^2}{2\mu^2},
\quad
I_2(\lambda,\mu)=\mathrm{diag} \{G_2(\lambda,\mu),H_2(\mu)\}.
\end{equation*}
We additionally consider the following conditions to control the approximation error.
\begin{enumerate}
\item[\textbf{[C4]}]
$\frac{n^{3-\alpha}\epsilon^4 \lor n^{2-\beta}\epsilon^2}{mN^{2\gamma}}\to0$,
$\frac{n^{2+\tau_1}\epsilon^2}{(M_1 \land M_2)^{2\tau_1}}\to0$
and $\frac{n^{3-\alpha+\tau_2}\epsilon^4}{(M_1 \land M_2)^{2\tau_2}}\to0$
for some $\tau_1\in[0,1)$ and $\tau_2\in[0,\alpha)$.

\item[\textbf{[C5]}]
$\frac{(n^\alpha\epsilon^2)^{-1} \lor (n^{1+\beta}\epsilon^4)^{-1}}{mN^{2\gamma}}\to0$,
$\frac{n^{-1+\tau_1}\epsilon^{-4}}{(M_1 \land M_2)^{2\tau_1}}\to0$
and $\frac{n^{-\alpha+\tau_2}\epsilon^{-2}}{(M_1 \land M_2)^{2\tau_2}}\to0$
for some $\tau_1\in[0,1)$ and $\tau_2\in[0,\alpha)$.

\end{enumerate}

\begin{thm}\label{th4}
Assume [A1]-[A3]. 
\begin{enumerate}
\item[(1)]
Suppose that $\mu_0$ is known.
\begin{enumerate}
\item[(a)]
If [C1] and [C4] hold, then as $n\to\infty$ and $\epsilon\to0$,
\begin{equation*}
\tilde \theta_0\pto \theta_0^*.
\end{equation*}

\item[(b)]
If [C3] and [C4] hold, then as $n\to\infty$ and $\epsilon\to0$,
\begin{equation*}
\epsilon^{-1}(\tilde\theta_0-\theta_0^*)
\dto N(0,G_2(\lambda_{1,1}^*,\mu_{1,1})^{-1}).
\end{equation*}
\end{enumerate}

\item[(2)]
Suppose that $\mu_0$ is unknown.
\begin{enumerate}
\item[(a)] 
If [C4] and [C5] hold, then as $n\to\infty$ and $\epsilon\to0$,
\begin{equation*}
(\tilde \theta_0,\tilde \mu_0)\pto (\theta_0^*,\mu_0^*).
\end{equation*}

\item[(b)]
If [C3]-[C5] hold, then as $n\to\infty$ and $\epsilon\to0$,
\begin{equation*}
\begin{pmatrix}
\epsilon^{-1}(\tilde\theta_0-\theta_0^*)
\\
\sqrt n(\tilde\mu_0-\mu_0^*)
\end{pmatrix}
\dto N(0,I_2(\lambda_{1,1}^*,\mu_{1,1}^*)^{-1}).
\end{equation*}

\end{enumerate}

\end{enumerate}

\end{thm}

\begin{rmk}{}
Let $\epsilon=1/n^\upsilon$, $\upsilon>0$.
The regular conditions required in Theorem \ref{th4} 
to evaluate the approximation error of the coordinate process can be simplified as follows.
\begin{enumerate}
\item[(1)] 
[C1] and [C4] are satisfied under any one of the following conditions. 
\begin{enumerate}
\item[(i)]
$\upsilon>1/2$ and [C1],

\item[(ii)]
$\upsilon\le1/2$ and [C4].
\end{enumerate}

\item[(2)]
[C3] and [C4] are satisfied under any one of the following conditions. 
\begin{enumerate}
\item[(i)]
$\upsilon>1/4$ and [C3],

\item[(ii)]
$\upsilon\le1/4$ and [C4].
\end{enumerate}

\item[(3)]
[C4] and [C5] are satisfied under any one of the following conditions. 
\begin{enumerate}
\item[(i)]
$\upsilon>1/2$ and [C5],

\item[(ii)]
$\upsilon\le1/2$ and [C4].

\end{enumerate}

\item[(4)]
[C3]-[C5] are satisfied under any one of the following conditions. 
\begin{enumerate}
\item[(i)]
$\upsilon>1$ and [C5],

\item[(ii)]
$1/4<\upsilon\le1$ and [C3],

\item[(iii)]
$\upsilon\le1/4$ and [C4].

\end{enumerate}
\end{enumerate}
The assertion of (1) can be obtained from the fact 
that [C1] and [C4] are the conditions where 
$r_{n,\epsilon}$ in Lemma \ref{lem3} is 
$r_{n,\epsilon}=n$ and $r_{n,\epsilon}=(n\epsilon)^2=n^{2-2\upsilon}$, respectively, 
and that when $\upsilon>1/2$ (resp. $\upsilon\le1/2$), 
it follows from $n \lor (n\epsilon)^2=n$ (resp. $(n\epsilon)^2$)
that if [C1] (resp. [C4]) is satisfied, then [C4] (resp. [C1]) holds.
Similarly, (2)-(4) can be obtained.
\end{rmk}

\begin{rmk}[]
(i) Theorem \ref{th4} (1) is an extension 
of the result of Kaino and Uchida \cite{Kaino_Uchida2021} with respect to $\theta_0$
to SPDEs in two space dimensions. 
Indeed, when $\mu_0$ is known, the $Q_2$-Wiener process is a driving noise 
with a known damping factor, which corresponds to a cylindrical Brownian motion 
in the sense that the driving noise is known. 
In other words, by using the $Q_2$-Wiener process with $\mu_0=0$, 
for example, the same result as \cite{Kaino_Uchida2021} can be obtained 
in SPDEs in two space dimensions.

(ii) Theorem \ref{th4} (2) does not require the balance condition 
between $n$ and $\epsilon$ as [B] in Gloter and S{\o}rensen \cite{Gloter_Sorensen2009}. 
This is because, as mentioned in Remark \ref{rmk2}, 
the solution of our diffusion process model 
can be explicitly expressed 
and it does not require a condition to evaluate the approximation error.
See Appendix for details.

\end{rmk}


\section{Proofs}\label{sec5}
We set the following notation.
\begin{enumerate}
\item[1.]
Let $\bs k=(k_1,k_2)\in \mathbb N^2$.

\item[2.]
For $A,B\ge0$, we write $A \lesssim B$ if $A\le CB$ for some constant $C>0$.

\item[3.]
For $x=(x_1,\ldots,x_d)\in\mathbb R^d$ and $f:\mathbb R^d\to\mathbb R$, 
we write $\partial_{x_1} f(x)=\frac{\partial}{\partial x_1}f(x)$,
$\partial f(x)=(\partial_{x_1}f(x),\ldots,\partial_{x_d}f(x))$
and $\partial^2 f(x)=(\partial_{x_j}\partial_{x_i} f(x))_{i,j=1}^d$.

\end{enumerate}

\subsection{Proofs of Proposition \ref{prop1}, Theorems \ref{th1} and \ref{th2}}
\label{sec5.1}
In this subsection, we provide proofs of our assertions 
in Subsection \ref{sec3.1}.
Let
\begin{align*}
A_{i,\bs k}
&=
-\langle \xi,e_{\bs k}\rangle_\theta
(1-\ee^{-\lambda_{\bs k}\Delta_N})\ee^{-\lambda_{\bs k}(i-1)\Delta_N},
\\
B_{1,i,\bs k}^{Q_1}
&=
-\frac{\epsilon(1-\ee^{-\lambda_{\bs k}\Delta_N})}{\lambda_{\bs k}^{\alpha/2}}
\int_0^{(i-1)\Delta_N}\ee^{-\lambda_{\bs k}((i-1)\Delta_N-s)}\dd w_{\bs k}(s),
\\
B_{2,i,\bs k}^{Q_1}
&=
\frac{\epsilon}{\lambda_{\bs k}^{\alpha/2}}
\int_{(i-1)\Delta_N}^{i\Delta_N}
\ee^{-\lambda_{\bs k}(i\Delta_N-s)}\dd w_{\bs k}(s),
\end{align*}
and $B_{i,\bs k}^{Q_1} =B_{1,i,\bs k}^{Q_1}+B_{2,i,\bs k}^{Q_1}$.
The increment $\Delta_i x_{\bs k}^{Q_1}$ can be expressed as 
$\Delta_i x_{\bs k}^{Q_1}=A_{i,\bs k}+B_{i,\bs k}^{Q_1} $.

By strengthening the assumption from Assumption 1 
in Tonaki et al. \cite{TKU2022arXiv} to [A1], 
the following lemma on $A_{i,\bs k}$ holds instead of the Lemma 5.3 
in \cite{TKU2022arXiv}.
\begin{lem}\label{lem1}
Under [A1], it holds that uniformly in $(y,z)\in D$,
\begin{align*}
&\sum_{i=1}^N\sum_{\bs k_1,\bs k_2\in\mathbb N^2}
\bigl|A_{i,\bs k_1}A_{i,\bs k_2}e_{\bs k_1}(y,z)e_{\bs k_2}(y,z)\bigr|
=O(\Delta_N^{\beta}), \quad \beta\in(0,1),
\\
&\sup_{j\ge1}\sum_{i=1}^N\sum_{\bs k_1,\bs k_2\in\mathbb N^2}
\bigl|A_{i,\bs k_1}A_{j,\bs k_2}e_{\bs k_1}(y,z)e_{\bs k_2}(y,z)\bigr|
=O(\Delta_N^{1/2}).
\end{align*}
\end{lem}
\begin{proof}
Note that the boundedness of $(1-\ee^{-x})/x^\tau$, $\tau\in(0,1)$ 
on $x\in(0,\infty)$
and 
\begin{equation*}
\Delta_N\sum_{\bs k\in\mathbb N^2}
\frac{1-\ee^{-\lambda_{\bs k}\Delta_N}}{(\lambda_{\bs k}\Delta_N)^{1+\tau}}
=O(1),\quad \tau\in(0,1),
\end{equation*}
which is obtained by (5.31) in Tonaki et al. \cite{TKU2022arXiv}.
Let $\beta\in(0,1)$.
It holds from $2-\beta=1+(1-\beta)$ and $1-\beta\in(0,1)$ that under [A1],
\begin{align*}
\sum_{\bs k\in\mathbb N^2}
\frac{1-\ee^{-\lambda_{\bs k}\Delta_N}}{\lambda_{\bs k}^2}
&=
\Delta_N^\beta\sum_{\bs k\in\mathbb N^2}
\frac{1}{\lambda_{\bs k}^{2-\beta}}\cdot
\frac{1-\ee^{-\lambda_{\bs k}\Delta_N}}{(\lambda_{\bs k}\Delta_N)^\beta}
=O(\Delta_N^\beta),
\\
\sum_{\bs k\in\mathbb N^2}
\frac{(1-\ee^{-\lambda_{\bs k}\Delta_N})^2}{\lambda_{\bs k}^2}
&=\Delta_N^2\sum_{\bs k\in\mathbb N^2}
\frac{1-\ee^{-\lambda_{\bs k}\Delta_N}}{(\lambda_{\bs k}\Delta_N)^{2-\beta}}
\cdot
\frac{1-\ee^{-\lambda_{\bs k}\Delta_N}}{(\lambda_{\bs k}\Delta_N)^\beta}
=O(\Delta_N).
\end{align*}
It also follows that 
$\sum_{\bs k\in\mathbb N^2}\lambda_{\bs k}^2
\langle \xi,e_{\bs k} \rangle_\theta^2<\infty$ 
under [A1]. Hence, we obtain by using the Schwarz inequality that 
\begin{align*}
&\sum_{i=1}^N\sum_{\bs k_1,\bs k_2\in\mathbb N^2}
\bigl|A_{i,\bs k_1}A_{i,\bs k_2}e_{\bs k_1}(y,z)e_{\bs k_2}(y,z)\bigr|
\\
&\lesssim 
\sum_{i=1}^N\sum_{\bs k_1,\bs k_2\in\mathbb N^2}
(1-\ee^{-\lambda_{\bs k_1}\Delta_N})(1-\ee^{-\lambda_{\bs k_2}\Delta_N})
\\
&\qquad\qquad\times
\ee^{-\lambda_{\bs k_1}(i-1)\Delta_N }
\ee^{-\lambda_{\bs k_2}(i-1)\Delta_N }
|\langle\xi,e_{\bs k_1}\rangle_\theta \langle\xi,e_{\bs k_2}\rangle_\theta|
\\
&=
\sum_{i=1}^N
\Biggl(
\sum_{\bs k\in\mathbb N^2}
(1-\ee^{-\lambda_{\bs k}\Delta_N})
\ee^{-\lambda_{\bs k}(i-1)\Delta_N }
|\langle\xi,e_{\bs k}\rangle_\theta|
\Biggr)^2
\\
&\le
\Biggl(
\sum_{\bs k\in\mathbb N^2}
\lambda_{\bs k}^2
\langle\xi,e_{\bs k}\rangle_\theta^2
\Biggr)
\Biggl(
\sum_{\bs k\in\mathbb N^2}
\frac{(1-\ee^{-\lambda_{\bs k}\Delta_N})^2}{\lambda_{\bs k}^2}
\sum_{i=1}^N
\ee^{-2\lambda_{\bs k}(i-1)\Delta_N }
\Biggr)
\\
&\lesssim
\sum_{\bs k\in\mathbb N^2}
\frac{1-\ee^{-\lambda_{\bs k}\Delta_N}}{\lambda_{\bs k}^2}
\\
&=O(\Delta_N^{\beta})
\end{align*}
and
\begin{align*}
&\sup_{j\ge1}\sum_{i=1}^N\sum_{\bs k_1,\bs k_2\in\mathbb N^2}
\bigl|A_{i,\bs k_1}A_{j,\bs k_2}e_{\bs k_1}(y,z)e_{\bs k_2}(y,z)\bigr|
\\
&\lesssim 
\sup_{j\ge1}\sum_{i=1}^N\sum_{\bs k_1,\bs k_2\in\mathbb N^2}
(1-\ee^{-\lambda_{\bs k_1}\Delta_N})(1-\ee^{-\lambda_{\bs k_2}\Delta_N})
\\
&\qquad\qquad\times
\ee^{-\lambda_{\bs k_1}(i-1)\Delta_N }
\ee^{-\lambda_{\bs k_2}(j-1)\Delta_N }
|\langle\xi,e_{\bs k_1}\rangle_\theta \langle\xi,e_{\bs k_2}\rangle_\theta|
\\
&\le 
\sum_{\bs k_1,\bs k_2\in\mathbb N^2}
(1-\ee^{-\lambda_{\bs k_2}\Delta_N})
|\langle\xi,e_{\bs k_1}\rangle_\theta \langle\xi,e_{\bs k_2}\rangle_\theta|
\\
&=
\Biggl(
\sum_{\bs k\in\mathbb N^2}
|\langle\xi,e_{\bs k}\rangle_\theta|
\Biggr)
\Biggl(
\sum_{\bs k\in\mathbb N^2}
(1-\ee^{-\lambda_{\bs k}\Delta_N})
|\langle\xi,e_{\bs k}\rangle_\theta|
\Biggr)
\\
&\le
\Biggl(
\sum_{\bs k\in\mathbb N^2}
\frac{1}{\lambda_{\bs k}^2}
\Biggr)^{1/2}
\Biggl(
\sum_{\bs k\in\mathbb N^2}
\frac{(1-\ee^{-\lambda_{\bs k}\Delta_N})^2}{\lambda_{\bs k}^2}
\Biggr)^{1/2}
\Biggl(
\sum_{\bs k\in\mathbb N^2}
\lambda_{\bs k}^2
\langle\xi,e_{\bs k}\rangle_\theta^2
\Biggr)
\\
&=O(\Delta_N^{1/2}).
\end{align*}
\end{proof}

For $B_{i,\bs k}^{Q_1}$,
the following lemma holds in the same way as Lemma 5.4 
in Tonaki et al. \cite{TKU2022arXiv}. 
\begin{lem}\label{lem2}
It holds that uniformly in $(y,z)\in D_\delta$,
\begin{equation*}
\sum_{\bs k,\bs \ell\in\mathbb N^2}
\EE[B_{i,\bs k}^{Q_1} B_{i,\bs \ell}^{Q_1}]e_{\bs k}(y,z)e_{\bs \ell}(y,z)
=\epsilon^2\biggl\{
\Delta_N^{\alpha}\frac{\Gamma(1-\alpha)}{4\pi\alpha\theta_2}
\exp\biggl(
-\frac{\theta_1}{\theta_2}y
\biggr)
\exp\biggl(
-\frac{\eta_1}{\theta_2}z
\biggr)
+r_{N,i}+O(\Delta_N)\biggr\},
\end{equation*}
where $\sum_{i=1}^N |r_{N,i}|=O(\Delta_N^{\beta})$, $\beta\in(0,1)$.
Moreover, it holds that uniformly in $(y,z)\in D_\delta$,
\begin{equation*}
\sup_{i\neq j}\sum_{j=1}^N
\sum_{\bs k,\bs \ell\in\mathbb N^2}
\Bigl|
\EE[B_{i,\bs k}^{Q_1} B_{j,\bs \ell}^{Q_1}]e_{\bs k}(y,z)e_{\bs \ell}(y,z)
\Bigr|
=O(\epsilon^2).
\end{equation*}
\end{lem}

In view of the proof of Lemma 5.4 in Tonaki et al. \cite{TKU2022arXiv},
it holds by choosing $\beta\in[\alpha,1)$ that
\begin{equation}\label{eq-0001}
\sup_{i=1,\ldots,N}\sum_{\bs k,\bs \ell\in\mathbb N^2}
\Bigl|
\EE[B_{i,\bs k}^{Q_1} B_{i,\bs \ell}^{Q_1}]e_{\bs k}(y,z)e_{\bs \ell}(y,z)
\Bigr|
=O(\epsilon^2\Delta_N^{\alpha})
\end{equation}
uniformly in $(y,z)\in D$.

\begin{proof}[\bf{Proof of Proposition \ref{prop1}}]
Noting that
\begin{equation*}
\EE[(\Delta_i X^{Q_1})^2(y,z)]
=\sum_{\bs k,\bs \ell\in\mathbb N^2} 
A_{i,\bs k}A_{i,\bs \ell}e_{\bs k}(y,z)e_{\bs \ell}(y,z)
+\sum_{\bs k,\bs \ell\in\mathbb N^2} 
\EE[B_{i,\bs k}^{Q_1}B_{i,\bs \ell}^{Q_1}]e_{\bs k}(y,z)e_{\bs \ell}(y,z),
\end{equation*}
we obtain the desired result
from Lemmas \ref{lem1} and \ref{lem2}. 
\end{proof}

\begin{proof}[\bf Proof of Theorem \ref{th1}]
Let $\nu=(\theta_1,\eta_1,\theta_2)$, 
$g_\nu(y,z)=\frac{\Gamma(1-\alpha)}{4\pi\alpha \theta_2}
\exp(-\frac{\theta_1}{\theta_2}y)\exp(-\frac{\eta_1}{\theta_2}z)$, 
\begin{equation*}
\xi_{N,\epsilon}(y,z,\nu)=\epsilon^{-2}Z_N^{Q_1}(y,z)-g_\nu(y,z)
\end{equation*}
and $\xi_{N,\epsilon}^*(y,z)=\xi_{N,\epsilon}(y,z,\nu^*)$.
Note that under [A1], the followings hold instead of (5.46) and (5.47) in 
Tonaki et al. \cite{TKU2022arXiv}.
\begin{equation}\label{eq-0002}
\sup_{(y,z)\in D_\delta}\EE[\epsilon^{-2}Z_N^{Q_1}(y,z)]
\lesssim 1,
\end{equation}
\begin{equation}\label{eq-0003}
\sup_{(y,z)\in D_\delta}\EE[\xi_{N,\epsilon}^*(y,z)^2]
=O(R_{N,\epsilon}^{-1}).
\end{equation}
\eqref{eq-0002} is shown by Proposition \ref{prop1}.

\textit{Proof of \eqref{eq-0003}. }
It follows that
\begin{align*}
\sup_{(y,z)\in D_\delta}\EE[\xi_{N,\epsilon}^*(y,z)^2]
\le
\sup_{(y,z)\in D_\delta}\VV[\epsilon^{-2}Z_N^{Q_1}(y,z)]
+\sup_{(y,z)\in D_\delta}\EE[\xi_{N,\epsilon}^*(y,z)]^2
\end{align*}
and from \eqref{prop1-eq2} that 
$\sup_{(y,z)\in D_\delta}\EE[\xi_{N,\epsilon}^*(y,z)]^2
=O(\Delta_N^{2(1-\alpha)} \lor \epsilon^{-4}\Delta_N^{2(1-\alpha+\beta)})$.
It also follows from the Isserlis' theorem that
\begin{align*}
&\VV[\epsilon^{-2}Z_N^{Q_1}(y,z)]
\\
&=\epsilon^{-4}\bigl\{\EE[Z_N^{Q_1}(y,z)^2]-\EE[Z_N^{Q_1}(y,z)]^2\bigr\}
\\
&=\biggl(\frac{\epsilon^{-2}}{N\Delta_N^{\alpha}}\biggr)^2
\sum_{i,j=1}^N
\Bigl\{
\EE[(\Delta_i X^{Q_1})^2(y,z)(\Delta_j X^{Q_1})^2(y,z)]
\\
&\qquad\qquad
-\EE[(\Delta_i X^{Q_1})^2(y,z)]\EE[(\Delta_j X^{Q_1})^2(y,z)]
\Bigr\}
\\
&=\biggl(\frac{\epsilon^{-2}}{N\Delta_N^{\alpha}}\biggr)^2
\sum_{i,j=1}^N
\sum_{\bs k_1,\ldots,\bs k_4\in\mathbb N^2}
\Bigl\{
A_{i,\bs k_1}A_{i,\bs k_2}A_{j,\bs k_3}A_{j,\bs k_4}
+2A_{i,\bs k_1}A_{i,\bs k_2}\EE[B_{j,\bs k_3}^{Q_1} B_{j,\bs k_4}^{Q_1} ]
\nonumber
\\
&\qquad\qquad\qquad
+4A_{i,\bs k_1}A_{j,\bs k_2}\EE[B_{i,\bs k_3}^{Q_1} B_{j,\bs k_4}^{Q_1} ]
+\EE[B_{i,\bs k_1}^{Q_1} B_{i,\bs k_2}^{Q_1} B_{j,\bs k_3}^{Q_1} B_{j,\bs k_4}^{Q_1} ]
\nonumber
\\
&\qquad\qquad\qquad
-A_{i,\bs k_1}A_{i,\bs k_2}A_{j,\bs k_3}A_{j,\bs k_4}
-2A_{i,\bs k_1}A_{j,\bs k_2}\EE[B_{i,\bs k_3}^{Q_1} B_{j,\bs k_4}^{Q_1} ]
\nonumber
\\
&\qquad\qquad\qquad-
\EE[B_{i,\bs k_1}^{Q_1} B_{i,\bs k_2}^{Q_1}]
\EE[B_{j,\bs k_3}^{Q_1} B_{j,\bs k_4}^{Q_1}]
\Bigr\}
e_{\bs k_1}(y,z)\cdots e_{\bs k_4}(y,z)
\\
&\lesssim
\biggl(\frac{\epsilon^{-2}}{N\Delta_N^{\alpha}}\biggr)^2
\sum_{i,j=1}^N
\sum_{\bs k_1,\ldots,\bs k_4\in\mathbb N^2}
A_{i,\bs k_1}A_{i,\bs k_2}\EE[B_{j,\bs k_3}^{Q_1} B_{j,\bs k_4}^{Q_1} ]
e_{\bs k_1}(y,z)\cdots e_{\bs k_4}(y,z)
\nonumber
\\
&\qquad+
\biggl(\frac{\epsilon^{-2}}{N\Delta_N^{\alpha}}\biggr)^2
\sum_{i=1}^N
\sum_{\bs k_1,\ldots,\bs k_4\in\mathbb N^2}
\EE[B_{i,\bs k_1}^{Q_1} B_{i,\bs k_2}^{Q_1}]
\EE[B_{i,\bs k_3}^{Q_1} B_{i,\bs k_4}^{Q_1}]
e_{\bs k_1}(y,z)\cdots e_{\bs k_4}(y,z)
\nonumber
\\
&\qquad+
\biggl(\frac{\epsilon^{-2}}{N\Delta_N^{\alpha}}\biggr)^2
\sum_{i\neq j}
\sum_{\bs k_1,\ldots,\bs k_4\in\mathbb N^2}
\Bigl\{
A_{i,\bs k_1}A_{j,\bs k_2}\EE[B_{i,\bs k_3}^{Q_1} B_{j,\bs k_4}^{Q_1} ]
\nonumber
\\
&\qquad\qquad\qquad+
\EE[B_{i,\bs k_1}^{Q_1} B_{j,\bs k_2}^{Q_1}]
\EE[B_{i,\bs k_3}^{Q_1} B_{j,\bs k_4}^{Q_1}]
\Bigr\}
e_{\bs k_1}(y,z)\cdots e_{\bs k_4}(y,z)
\\
&=:S_1+S_2+S_3.
\end{align*}
Noting from Lemmas \ref{lem1}, \ref{lem2} and $\beta\in[\alpha,1)$ that 
\begin{align*}
S_1
&\le
\biggl(\frac{\epsilon^{-2}}{N\Delta_N^{\alpha}}\biggr)^2
\Biggl\{
\Biggl(
\sum_{i=1}^N
\sum_{\bs k_1,\bs k_2\in\mathbb N^2}
\bigl|
A_{i,\bs k_1}A_{i,\bs k_2}
e_{\bs k_1}(y,z)e_{\bs k_2}(y,z)
\bigr|
\Biggr)
\\
&\qquad\qquad\times
\Biggl(
\sum_{i=1}^N
\sum_{\bs k_1,\bs k_2\in\mathbb N^2}
\Bigl|
\EE[B_{i,\bs k_1}^{Q_1} B_{i,\bs k_2}^{Q_1} ]
e_{\bs k_1}(y,z)e_{\bs k_2}(y,z)
\Bigr|
\Biggr)
\Biggr\}
\\
&=O\Biggl(
\biggl(\frac{\epsilon^{-2}}{N\Delta_N^\alpha}\biggr)^2
\epsilon^2\Delta_N^{-1+\alpha+\beta}
\Biggr)
=O(\epsilon^{-2}\Delta_N^{1-\alpha+\beta}),
\\
S_2
&=
\biggl(\frac{\epsilon^{-2}}{N\Delta_N^{\alpha}}\biggr)^2
\sum_{i=1}^N
\Biggl(
\sum_{\bs k_1,\bs k_2\in\mathbb N^2}
\EE[B_{i,\bs k_1}^{Q_1} B_{i,\bs k_2}^{Q_1} ]
e_{\bs k_1}(y,z)e_{\bs k_2}(y,z)
\Biggr)^2
\\
&=O\Biggl(
\biggl(\frac{\epsilon^{-2}}{N\Delta_N^\alpha}\biggr)^2
\epsilon^4\Delta_N^{2\alpha-1}
\Biggr)
=O(\Delta_N),
\\
S_3
&\le 
\biggl(\frac{\epsilon^{-2}}{N\Delta_N^{\alpha}}\biggr)^2
\Biggl\{
\Biggl(
\sup_{j\ge1}\sum_{i=1}^N
\sum_{\bs k_1,\bs k_2\in\mathbb N^2}
\bigl|
A_{i,\bs k_1}A_{j,\bs k_2}
e_{\bs k_1}(y,z)e_{\bs k_2}(y,z)
\bigr|
\Biggr)
\\
&\qquad\qquad\qquad\times
\Biggl(
\sup_{i\neq j}
\sum_{j=1}^N
\sum_{\bs k_1,\bs k_2\in\mathbb N^2}
\Bigl|
\EE[B_{i,\bs k_1}^{Q_1} B_{j,\bs k_2}^{Q_1} ]
e_{\bs k_1}(y,z)e_{\bs k_2}(y,z)
\Bigr|
\Biggr)
\\
&\qquad\qquad+ 
\Biggl(
\sup_{j\neq i}\sum_{i=1}^N
\sum_{\bs k_1,\bs k_2\in\mathbb N^2}
\Bigl|
\EE[B_{i,\bs k_1}^{Q_1} B_{j,\bs k_2}^{Q_1} ]e_{\bs k_1}(y,z)e_{\bs k_2}(y,z)
\Bigr|
\Biggr)^2
\Biggr\}
\\
&=O\Biggl(
\biggl(\frac{\epsilon^{-2}}{N\Delta_N^\alpha}\biggr)^2
(\epsilon^2\Delta_N^{1/2} \lor \epsilon^4)
\Biggr)
=O(\epsilon^{-2}\Delta_N^{2(1-\alpha)+1/2} \lor \Delta_N^{2(1-\alpha)}),
\end{align*}
one has $\sup_{(y,z)\in D_\delta}
\VV[\epsilon^{-2}Z_N^{Q_1}(y,z)]=O(\Delta_N^{1\land2(1-\alpha)} 
\lor \epsilon^{-2}\Delta_N^{(1-\alpha+\beta)\land(2(1-\alpha)+1/2)})$.
Hence, \eqref{eq-0003} holds. 

By applying the Taylor expansion,
\begin{align*}
-\partial{U}_{N,m,\epsilon}^{(1)}(\nu^*)^\TT
=\partial{U}_{N,m,\epsilon}^{(1)}(\hat\nu)^\TT
-\partial{U}_{N,m,\epsilon}^{(1)}(\nu^*)^\TT
=\int_0^1\partial^2{U}_{N,m,\epsilon}^{(1)}(\nu^*+u(\hat\nu-\nu^*))\dd u(\hat\nu-\nu^*),
\end{align*}
that is, for $\gamma>0$ with $mN^{2\gamma}R_{N,\epsilon}^{-1}\to0$, 
\begin{equation*}
-\frac{N^{\gamma}}{m^{1/2}}\partial{U}_{N,m,\epsilon}^{(1)}(\nu^*)^\TT
=\frac{1}{m}\int_0^1\partial^2{U}_{N,m,\epsilon}^{(1)}(\nu^*+u(\hat\nu-\nu^*))\dd u\, 
m^{1/2}N^{\gamma}(\hat\nu-\nu^*).
\end{equation*}

In the same way as the proof of Theorem 3.2 in Tonaki et al. \cite{TKU2022arXiv}, 
it holds from \eqref{eq-0002} and \eqref{eq-0003} that
\begin{enumerate}
\item[(i)]
$\hat\nu\pto\nu^*$,

\item[(ii)]
There exists $U^*\in\mathrm{GL}_3(\mathbb R)$ such that
$\frac{1}{m}\partial^2{U}_{N,m,\epsilon}^{(1)}(\nu^*)\pto U^*$,

\item[(iii)]
For $\delta_{N,\epsilon}\downarrow0$,
\begin{equation*}
\frac{1}{m}\sup_{|\nu-\nu^*|\le\delta_{N,\epsilon}}
|\partial^2{U}_{N,m,\epsilon}^{(1)}(\nu)
-\partial^2{U}_{N,m,\epsilon}^{(1)}(\nu^*)|\pto0.
\end{equation*}
\end{enumerate}
Thus all that remains is to prove
\begin{equation}\label{eq-0004}
-\frac{N^{\gamma}}{m^{1/2}}\partial{U}_{N,m,\epsilon}^{(1)}(\nu^*)^\TT \pto 0. 
\end{equation}

\textit{Proof of \eqref{eq-0004}. }
Let $\varphi_\nu(y,z)
=\exp(-\frac{\theta_1}{\theta_2}y)\exp(-\frac{\eta_1}{\theta_2}z)(1,-s y,-s z)^\TT$.
Since
\begin{equation*}
\EE[\xi_{N,\epsilon}^*(y,z)]=
O(\Delta_N^{1-\alpha} \lor \epsilon^{-2}\Delta_N^{1-\alpha+\beta}),
\end{equation*}
\begin{equation*}
\EE\bigl[(\xi_{N,\epsilon}^*(y,z)-\EE[\xi_{N,\epsilon}^*(y,z)])^2\bigr]
=O(R_{N,\epsilon}^{-1})
\end{equation*}
uniformly in $(y,z)\in D_\delta$,
it follows from  
$R_{N,\epsilon}^{1/2}\le N^{1-\alpha} \land \epsilon^2 N^{1-\alpha+\beta}$ that
\begin{align*}
&-\frac{N^{\gamma}}{m^{1/2}}
\partial{U}_{N,m,\epsilon}^{(1)}(\nu^*)^\TT
\\
&=
\frac{N^\gamma}{m^{1/2}}
\frac{\Gamma(1-\alpha)}{2\pi\alpha}
\sum_{j_1=1}^{m_1}\sum_{j_2=1}^{m_2}
\bigl\{
\xi_{N,\epsilon}^*(\widetilde y_{j_1},\widetilde z_{j_2})
-\EE[\xi_{N,\epsilon}^*(\widetilde y_{j_1},\widetilde z_{j_2})]
\bigr\}
\varphi_{\nu^*}(\widetilde y_{j_1}, \widetilde z_{j_2})
\\
&\qquad+
\frac{N^\gamma}{m^{1/2}}
\frac{\Gamma(1-\alpha)}{2\pi\alpha}
\sum_{j_1=1}^{m_1}\sum_{j_2=1}^{m_2}
\EE[\xi_{N,\epsilon}^*(\widetilde y_{j_1},\widetilde z_{j_2})]
\varphi_{\nu^*}(\widetilde y_{j_1}, \widetilde z_{j_2})
\\
&\lesssim
m^{1/2}N^\gamma 
\Biggl(
\frac{1}{m}
\sum_{j_1=1}^{m_1}\sum_{j_2=1}^{m_2}
\bigl\{
\xi_{N,\epsilon}^*(\widetilde y_{j_1},\widetilde z_{j_2})
-\EE[\xi_{N,\epsilon}^*(\widetilde y_{j_1},\widetilde z_{j_2})]
\bigr\}^2
\Biggr)^{1/2}
\\
&\qquad\qquad\times
\Biggl(
\frac{1}{m}
\sum_{j_1=1}^{m_1}\sum_{j_2=1}^{m_2}
|\varphi_{\nu^*}(\widetilde y_{j_1}, \widetilde z_{j_2})|^2
\Biggr)^{1/2}
\\
&\qquad+
\frac{N^\gamma}{m^{1/2}}
\sum_{j_1=1}^{m_1}\sum_{j_2=1}^{m_2}
\EE[\xi_{N,\epsilon}^*(\widetilde y_{j_1},\widetilde z_{j_2})]
\varphi_{\nu^*}(\widetilde y_{j_1}, \widetilde z_{j_2})
\\
&=
O_p\bigl(m^{1/2}N^\gamma R_{N,\epsilon}^{-1/2}\bigr)
+O\bigl(m^{1/2}N^\gamma
(\Delta_N^{1-\alpha} \lor \epsilon^{-2}\Delta_N^{1-\alpha+\beta})
\bigr)
\\
&=o_p(1).
\end{align*}
\end{proof}

Before proving Theorem \ref{th2}, 
we prepare a lemma on the error evaluation between 
the approximate coordinate process $\hat x_{\bs k}^{Q_1}(t)$
and the coordinate process $x_{\bs k}^{Q_1}(t)$.
Hereafter, let 
$X_t=X_t^{Q_1}$, $x_{\bs k}=x_{\bs k}^{Q_1}$ and 
$\hat x_{\bs k}=\hat x_{\bs k}^{Q_1}$.
We set 
\begin{equation}\label{eq-0101}
\mathcal M_{i}(\lambda)
=x_{1,1}(\widetilde t_i)-\ee^{-\lambda\Delta_n} x_{1,1}(\widetilde t_{i-1}),
\quad
\hat{\mathcal M}_{i}(\lambda)
=\hat x_{1,1}(\widetilde t_i)-\ee^{-\lambda\Delta_n} \hat x_{1,1}(\widetilde t_{i-1}),
\end{equation}
\begin{equation*}
\hat{\mathcal A}_n
=\sup_{\lambda}\sum_{i=1}^n
\{\hat{\mathcal M}_i(\lambda)-\mathcal M_i(\lambda)\}^2,
\quad
\hat{\mathcal B}_n
=\sum_{i=0}^n
\{\hat x_{1,1}(\widetilde t_{i})-x_{1,1}(\widetilde t_{i})\}^2.
\end{equation*}
\begin{lem}\label{lem3}
Assume [A1]-[A3].
\begin{enumerate}
\item[(1)]
Let $\{r_{n,\epsilon}\}$ be a sequence such that
$\frac{(\epsilon^2n^{1-\alpha} \lor n^{-\beta})r_{n,\epsilon}}{m N^{2\gamma}}\to0$,
$\frac{n^{\tau_1}r_{n,\epsilon}}{(M_1\land M_2)^{2\tau_1}}\to0$ 
and $\frac{n^{1-\alpha+\tau_2}\epsilon^2r_{n,\epsilon}}{(M_1\land M_2)^{2\tau_2}}\to0$
for some $\tau_1\in[0,1)$ and $\tau_2\in[0,\alpha)$. 
As $n,M\to\infty$ and $\epsilon\to0$, 
\begin{equation*}
r_{n,\epsilon} \hat{\mathcal A}_n=o_p(1).
\end{equation*}

\item[(2)]
Let $\{s_{n,\epsilon}\}$ be a sequence such that
$\frac{n s_{n,\epsilon}}{m N^{2\gamma}}\to0$,
$\frac{n s_{n,\epsilon}}{(M_1 \land M_2)^{1+\tau_3}}\to0$
and $\frac{n\epsilon^2 s_{n,\epsilon}}{(M_1 \land M_2)^{2\tau_4}}\to0$
for some $\tau_3\in[0,1)$ and $\tau_4\in[0,\alpha)$. 
As $n,M\to\infty$ and $\epsilon\to0$, 
\begin{equation*}
s_{n,\epsilon} \hat{\mathcal B}_n=o_p(1).
\end{equation*}

\end{enumerate}
\end{lem}

By setting the following condition
\begin{enumerate}
\item[\textbf{[C3$\boldsymbol'$]}]
$\frac{\epsilon^{-2}}{mN^{2\gamma}}\to0$
and $\frac{\epsilon^{-2}}{(M_1 \land M_2)^{1+\tau_4}}\to0$
for some $\tau_4\in[0,\alpha)$,
\end{enumerate}
it holds from Lemma \ref{lem3} that 
\begin{align}
n\hat{\mathcal A}_n &=o_p(1) \quad \text{under [C1]}, 
\label{eq-0102}
\\
\epsilon^{-2}\hat{\mathcal A}_n &=o_p(1) \quad \text{under [C2]}, 
\label{eq-0103}
\\
n\epsilon^{-2}\hat{\mathcal A}_n &=o_p(1) \quad \text{under [C3]},
\label{eq-0104}
\\
(n\epsilon^2)^{-1}\hat{\mathcal B}_n &=o_p(1) \quad \text{under [C3$'$]}.
\nonumber
\end{align}
Note that if [C3] holds, then [C3$'$] is satisfied. It also holds that 
\begin{equation}
(n\epsilon^2)^{-1}\hat{\mathcal B}_n=o_p(1) \quad \text{under [C3]}.
\label{eq-0105}
\end{equation}

\begin{proof}[\bf{Proof of Lemma \ref{lem3}}]
(1) Let $D_{j_1,j_2}=(y_{j_1-1},y_{j_1}]\times(z_{j_2-1},z_{j_2}]$, 
$\Delta_i X(y,z)=X_{\widetilde t_i}(y,z)-X_{\widetilde t_{i-1}}(y,z)$
and $\mathcal N_i(y,z;\lambda)=X_{\widetilde t_{i}}(y,z)
-\ee^{-\lambda\Delta_n}X_{\widetilde t_{i-1}}(y,z)$.
Since
\begin{align*}
\hat x_{1,1}(t)
&=
\frac{2}{M}\sum_{j_1=1}^{M_1}\sum_{j_2=1}^{M_2}
X_t(y_{j_1},z_{j_2})\sin(\pi y_{j_1})\sin(\pi z_{j_2})
\exp\biggl(
\frac{\hat\theta_1}{2\hat\theta_2}y_{j_1}
+\frac{\hat\eta_1}{2\hat\theta_2}z_{j_2}
\biggr),
\\
x_{1,1}(t)
&=
2\sum_{j_1=1}^{M_1}\sum_{j_2=1}^{M_2}\iint_{D_{j_1,j_2}}
X_t(y,z)
\sin(\pi y)\sin(\pi z)
\exp\biggl(
\frac{\theta_1^*}{2\theta_2^*}y+\frac{\eta_1^*}{2\theta_2^*}z
\biggr)
\dd y\dd z,
\end{align*}
it follows that 
\begin{align*}
\hat{\mathcal M}_i(\lambda)-\mathcal M_i(\lambda)
&=
\frac{2}{M}\sum_{j_1=1}^{M_1}\sum_{j_2=1}^{M_2}
\mathcal N_i(y_{j_1},z_{j_2};\lambda)\sin(\pi y_{j_1})\sin(\pi z_{j_2})
\\
&\qquad\qquad\times
\Biggl\{
\exp\biggl(
\frac{\hat\theta_1}{2\hat\theta_2}y_{j_1}
+\frac{\hat\eta_1}{2\hat\theta_2}z_{j_2}
\biggr)
-\exp\biggl(
\frac{\theta_1^*}{2\theta_2^*}y_{j_1}+\frac{\eta_1^*}{2\theta_2^*}z_{j_2}
\biggr)
\Biggr\}
\\
&\qquad+
2\sum_{j_1=1}^{M_1}\sum_{j_2=1}^{M_2}\iint_{D_{j_1,j_2}}
\mathcal N_i(y_{j_1},z_{j_2};\lambda)
\\
&\qquad\qquad\times
\Biggl\{\sin(\pi y_{j_1})\sin(\pi z_{j_2})
\exp\biggl(
\frac{\theta_1^*}{2\theta_2^*}y_{j_1}+\frac{\eta_1^*}{2\theta_2^*}z_{j_2}
\biggr)
\\
&\qquad\qquad\qquad-
\sin(\pi y)\sin(\pi z)
\exp\biggl(
\frac{\theta_1^*}{2\theta_2^*}y+\frac{\eta_1^*}{2\theta_2^*}z
\biggr)
\Biggr\}
\dd y\dd z
\\
&\qquad+
2\sum_{j_1=1}^{M_1}\sum_{j_2=1}^{M_2}\iint_{D_{j_1,j_2}}
\bigl\{\mathcal N_i(y_{j_1},z_{j_2};\lambda)-\mathcal N_i(y,z;\lambda)\bigr\}
\\
&\qquad\qquad\times
\sin(\pi y)\sin(\pi z)
\exp\biggl(
\frac{\theta_1^*}{2\theta_2^*}y+\frac{\eta_1^*}{2\theta_2^*}z
\biggr)
\dd y\dd z
\\
&=:2\{\mathcal A_{1,i}(\lambda)+\mathcal A_{2,i}(\lambda)+\mathcal A_{3,i}(\lambda)\}
\end{align*}
and
\begin{align*}
\hat{\mathcal A}_n
\lesssim
\sum_{i=1}^{n}
\Bigl\{
\sup_{\lambda}\mathcal A_{1,i}(\lambda)^2
+\sup_{\lambda}\mathcal A_{2,i}(\lambda)^2
+\sup_{\lambda}\mathcal A_{3,i}(\lambda)^2
\Bigr\}.
\end{align*}

For the evaluation of $\mathcal A_{1,i}$,
by the Taylor expansion, 
\begin{align*}
\sup_{\lambda}\mathcal A_{1,i}(\lambda)^2
&\le
\frac{1}{M}\sum_{j_1=1}^{M_1}\sum_{j_2=1}^{M_2}
\sup_{\lambda}\mathcal N_i(y_{j_1},z_{j_2};\lambda)^2
\sin^2(\pi y_{j_1})\sin^2(\pi z_{j_2})
\\
&\qquad\times
\Biggl\{
\exp\biggl(
\frac{\hat\theta_1}{2\hat\theta_2}y_{j_1}
+\frac{\hat\eta_1}{2\hat\theta_2}z_{j_2}
\biggr)
-\exp\biggl(
\frac{\theta_1^*}{2\theta_2^*}y_{j_1}+\frac{\eta_1^*}{2\theta_2^*}z_{j_2}
\biggr)
\Biggr\}^2
\\
&=\frac{1}{M}\sum_{j_1=1}^{M_1}\sum_{j_2=1}^{M_2}
\sup_{\lambda}\mathcal N_i(y_{j_1},z_{j_2};\lambda)^2
\sin^2(\pi y_{j_1})\sin^2(\pi z_{j_2})
\\
&\qquad\times
\Biggl\{
\frac{1}{2(\theta_2^*)^2}\int_0^1(\theta_2^* y_{j_1},\theta_2^* z_{j_2},
-\theta_1^*y_{j_1}-\eta_1^*z_{j_2})
\\
&\qquad\qquad
\left.
\times
\exp\Biggl(
\frac{y_{j_1}}{2}
\frac{\theta_1^*+u(\hat\theta_1-\theta_1^*)}{\theta_2^*+u(\hat\theta_2-\theta_2^*)}
+\frac{z_{j_2}}{2}
\frac{\eta_1^*+u(\hat\eta_1-\eta_1^*)}{\theta_2^*+u(\hat\theta_2-\theta_2^*)}
\Biggr)\dd u
\begin{pmatrix}
\hat\theta_1-\theta_1^*
\\
\hat\eta_1-\eta_1^*
\\
\hat\theta_2-\theta_2^*
\end{pmatrix}
\right\}
\\
&\le \frac{1}{mN^{2\gamma}}\frac{1}{M}\sum_{j_1=1}^{M_1}\sum_{j_2=1}^{M_2}
\sup_{\lambda}\mathcal N_i(y_{j_1},z_{j_2};\lambda)^2
\\
&\qquad\times
\Biggl|
\frac{1}{2(\theta_2^*)^2}\int_0^1(\theta_2^* y_{j_1},\theta_2^* z_{j_2},
-\theta_1^*y_{j_1}-\eta_1^*z_{j_2})
\\
&\qquad\qquad
\times
\exp\Biggl(
\frac{y_{j_1}}{2}
\frac{\theta_1^*+u(\hat\theta_1-\theta_1^*)}{\theta_2^*+u(\hat\theta_2-\theta_2^*)}
+\frac{z_{j_2}}{2}
\frac{\eta_1^*+u(\hat\eta_1-\eta_1^*)}{\theta_2^*+u(\hat\theta_2-\theta_2^*)}
\Biggr)\dd u
\Biggr|^2
\\
&\qquad\times
mN^{2\gamma}
(|\hat\theta_1-\theta_1^*|^2+|\hat\eta_1-\eta_1^*|^2+|\hat\theta_2-\theta_2^*|^2)
\\
&=:\mathcal A_{1,i}'
\times 
mN^{2\gamma}
(|\hat\theta_1-\theta_1^*|^2+|\hat\eta_1-\eta_1^*|^2+|\hat\theta_2-\theta_2^*|^2).
\end{align*}
It follows from Lemma \ref{lem1} and \eqref{eq-0001} that
\begin{equation}\label{eq-0106}
\sum_{i=1}^n \EE[(\Delta_i X)^2(y,z)]
\lesssim \epsilon^2n^{1-\alpha} \lor n^{-\beta},
\quad
\sup_{t\in[0,1]}\EE[X_t^2(y,z)]\lesssim 1
\end{equation}
uniformly in $(y,z)\in D$, and then it holds from 
\begin{equation*}
\mathcal N_i(y,z;\lambda)
=\Delta_i X(y,z)+(1-\ee^{-\lambda\Delta_n})X_{\widetilde t_{i-1}}(y,z)
\end{equation*}
and $\sup_{\lambda}|1-\ee^{-\lambda\Delta_n}|\lesssim \Delta_n$ 
that uniformly in $(y,z)\in D$,
\begin{equation*}
\sum_{i=1}^n \EE\Bigl[\sup_{\lambda}\mathcal N_i(y,z;\lambda)^2\Bigr]
\lesssim \epsilon^2n^{1-\alpha} \lor n^{-\beta} \lor n^{-1}
=\epsilon^2n^{1-\alpha} \lor n^{-\beta}.
\end{equation*}
Choose any positive numbers $\epsilon_1$ and $\epsilon_2$.
On $\Omega_{\epsilon_1}=\{
|\hat\theta_1-\theta_1^*|+|\hat\eta_1-\eta_1^*|+|\hat\theta_2-\theta_2^*|
<\epsilon_1\}$,
\begin{equation*}
\mathcal A_n'
:=
r_{n,\epsilon}
\sum_{i=1}^{n}\mathcal A_{1,i}'
\lesssim
\frac{r_{n,\epsilon}}{mN^{2\gamma}}\sum_{i=1}^{n}
\frac{1}{M}
\sum_{j_1=1}^{M_1}\sum_{j_2=1}^{M_2}
\sup_{\lambda}\mathcal N_i(y_{j_1},z_{j_2};\lambda)^2,
\end{equation*}
and
\begin{align*}
P(|\mathcal A_n'|>\epsilon_2)
&=P(\{|\mathcal A_n'|>\epsilon_2\}\cap \Omega_{\epsilon_1})
+P(\{|\mathcal A_n'|>\epsilon_2\}\cap \Omega_{\epsilon_1}^{\mathrm c})
\nonumber
\\
&\le
P\Biggl(
\frac{r_{n,\epsilon}}{mN^{2\gamma}}
\sum_{i=1}^{n}
\frac{1}{M}\sum_{j_1=1}^{M_1}\sum_{j_2=1}^{M_2}
\sup_{\lambda}\mathcal N_i(y_{j_1},z_{j_2};\lambda)^2
\gtrsim\epsilon_2
\Biggr)
+P(\Omega_{\epsilon_1}^{\mathrm c})
\nonumber
\\
&\lesssim
\frac{r_{n,\epsilon}}{\epsilon_2 m N^{2\gamma}}
\sum_{i=1}^{n}
\frac{1}{M}\sum_{j_1=1}^{M_1}\sum_{j_2=1}^{M_2}
\EE\Bigl[\sup_{\lambda}\mathcal N_i(y_{j_1},z_{j_2};\lambda)^2\Bigr]
+P(\Omega_{\epsilon_1}^{\mathrm c})
\nonumber
\\
&\lesssim
\frac{(\epsilon^2n^{1-\alpha} \lor n^{-\beta})r_{n,\epsilon}}
{\epsilon_2 mN^{2\gamma}}
+P(\Omega_{\epsilon_1}^{\mathrm c}).
\end{align*}
Therefore, one has from 
$\frac{(\epsilon^2n^{1-\alpha} \lor n^{-\beta})r_{n,\epsilon}}{m N^{2\gamma}}\to0$ 
and Theorem \ref{th1} that $\mathcal A_n'=o_p(1)$ and
\begin{equation*}
r_{n,\epsilon}\sum_{i=1}^{n}
\sup_{\lambda}\mathcal A_{1,i}(\lambda)^2
=o_p(1).
\end{equation*}

Noting that for the evaluation of $\mathcal A_{2,i}$,
\begin{align*}
\sum_{i=1}^{n}
\EE\Bigl[\sup_{\lambda}\mathcal A_{2,i}^2(\lambda)\Bigr]
&\le
\sum_{i=1}^{n}
\sum_{j_1=1}^{M_1}\sum_{j_2=1}^{M_2}\iint_{D_{j_1,j_2}}
\EE\Bigl[\sup_{\lambda}\mathcal N_i(y_{j_1},z_{j_2};\lambda)^2\Bigr]
\\
&\qquad\times
\Biggl\{\sin(\pi y_{j_1})\sin(\pi z_{j_2})
\exp\biggl(
\frac{\theta_1^*}{2\theta_2^*}y_{j_1}+\frac{\eta_1^*}{2\theta_2^*}z_{j_2}
\biggr)
\\
&\qquad\qquad-
\sin(\pi y)\sin(\pi z)
\exp\biggl(
\frac{\theta_1^*}{2\theta_2^*}y+\frac{\eta_1^*}{2\theta_2^*}z
\biggr)
\Biggr\}^2
\dd y\dd z
\\
&\lesssim
\sum_{i=1}^{n}
\sum_{j_1=1}^{M_1}\sum_{j_2=1}^{M_2}
\EE\Bigl[\sup_{\lambda}\mathcal N_i(y_{j_1},z_{j_2};\lambda)^2\Bigr]
\iint_{D_{j_1,j_2}}
\biggl(\frac{1}{M_1^2}+\frac{1}{M_2^2}\biggr)
\dd y\dd z
\\
&\lesssim
\frac{\epsilon^2n^{1-\alpha} \lor n^{-\beta}}{(M_1 \land M_2)^2},
\end{align*}
one has that under
$\frac{n^{\tau_1}r_{n,\epsilon}}{(M_1\land M_2)^{2\tau_1}}\to0$ 
and $\frac{n^{1-\alpha+\tau_2}\epsilon^2r_{n,\epsilon}}{(M_1\land M_2)^{2\tau_2}}\to0$,
\begin{equation*}
r_{n,\epsilon}
\sum_{i=1}^{n}
\sup_{\lambda}\mathcal A_{2,i}^2(\lambda)
=
O_p\biggl(
\frac{(\epsilon^2n^{1-\alpha} \lor n^{-\beta})r_{n,\epsilon}}
{(M_1 \land M_2)^2}
\biggr)
=o_p(1).
\end{equation*}

For the evaluation of $\mathcal A_{3,i}$, 
it is shown that
\begin{align*}
\sup_{\lambda}\mathcal A_{3,i}^2(\lambda)
&\le
\sum_{j_1=1}^{M_1}\sum_{j_2=1}^{M_2}\iint_{D_{j_1,j_2}}
\sup_{\lambda}
\bigl\{\mathcal N_i(y_{j_1},z_{j_2};\lambda)-\mathcal N_i(y,z;\lambda)\bigr\}^2
\\
&\qquad\qquad\times
\sin^2(\pi y)\sin^2(\pi z)
\exp\biggl(
\frac{\theta_1^*}{\theta_2^*}y+\frac{\eta_1^*}{\theta_2^*}z
\biggr)
\dd y\dd z
\\
&\lesssim
\sum_{j_1=1}^{M_1}\sum_{j_2=1}^{M_2}\iint_{D_{j_1,j_2}}
\sup_{\lambda}
\bigl\{\mathcal N_i(y_{j_1},z_{j_2};\lambda)-\mathcal N_i(y,z;\lambda)\bigr\}^2
\dd y\dd z
\\
&\lesssim
\sum_{j_1=1}^{M_1}\sum_{j_2=1}^{M_2}\iint_{D_{j_1,j_2}}
\Bigl(
\{\Delta_i X(y_{j_1},z_{j_2})-\Delta_i X(y,z)\}^2
\\
&\qquad\qquad
+\sup_{\lambda}(1-\ee^{-\lambda\Delta_n})^2
\{X_{\widetilde t_{i-1}}(y_{j_1},z_{j_2})-X_{\widetilde t_{i-1}}(y,z)\}^2
\Bigr)
\dd y\dd z
\end{align*}
Since for $(y,z)\in D_{j_1,j_2}$,
\begin{align*}
X_{\widetilde t_{i-1}}(y_{j_1},z_{j_2})-X_{\widetilde t_{i-1}}(y,z)
&=\sum_{\bs k\in\mathbb N^2}
x_{\bs k}(\widetilde t_{i-1})\{e_{\bs k}(y_{j_1},z_{j_2})-e_{\bs k}(y,z)\},
\\
\Delta_i X(y_{j_1},z_{j_2})-\Delta_i X(y,z)
&=\sum_{\bs k\in\mathbb N^2}
\Delta_i x_{\bs k}\{e_{\bs k}(y_{j_1},z_{j_2})-e_{\bs k}(y,z)\},
\\
|e_{\bs k}(y_{j_1},z_{j_2})-e_{\bs k}(y,z)|
&\lesssim \frac{\lambda_{\bs k}^{1/2}}{M_1\land M_2} \land 1,
\end{align*}
it holds that for $\tau_1\in[0,1)$ and $\tau_2\in[0,\alpha)$,
\begin{align*}
&\EE\bigl[\{\Delta_i X(y_{j_1},z_{j_2})-\Delta_i X(y,z)\}^2\bigr]
\\
&=\sum_{\bs k,\bs \ell\in\mathbb N^2}
\EE[\Delta_i x_{\bs k}\Delta_i x_{\bs \ell}]
\{e_{\bs k}(y_{j_1},z_{j_2})-e_{\bs k}(y,z)\}
\{e_{\bs \ell}(y_{j_1},z_{j_2})-e_{\bs \ell}(y,z)\}
\\
&=
\sum_{\bs k,\bs \ell\in\mathbb N^2}
A_{i,\bs k}A_{i,\bs \ell}
\{e_{\bs k}(y_{j_1},z_{j_2})-e_{\bs k}(y,z)\}
\{e_{\bs \ell}(y_{j_1},z_{j_2})-e_{\bs \ell}(y,z)\}
\\
&\qquad+
\sum_{\bs k,\bs \ell\in\mathbb N^2}
\EE[B_{i,\bs k}B_{i,\bs \ell}]
\{e_{\bs k}(y_{j_1},z_{j_2})-e_{\bs k}(y,z)\}
\{e_{\bs \ell}(y_{j_1},z_{j_2})-e_{\bs \ell}(y,z)\}
\\
&\lesssim
\sum_{\bs k\in\mathbb N^2}
\frac{(1-\ee^{-\lambda_{\bs k}\Delta_{n}})^2}{\lambda_{\bs k}^2}
\{e_{\bs k}(y_{j_1},z_{j_2})-e_{\bs k}(y,z)\}^2
\\
&\qquad+\epsilon^2
\sum_{\bs k\in\mathbb N^2}
\frac{1-\ee^{-\lambda_{\bs k}\Delta_{n}}}{\lambda_{\bs k}^{1+\alpha}}
\{e_{\bs k}(y_{j_1},z_{j_2})-e_{\bs k}(y,z)\}^2
\\
&\le
\sum_{\bs k\in\mathbb N^2}
\frac{(1-\ee^{-\lambda_{\bs k}\Delta_{n}})^2}{\lambda_{\bs k}^2}
\biggl(\frac{\lambda_{\bs k}^{1/2}}{M_1\land M_2}\biggr)^{2\tau_1}
1^{2(1-\tau_1)}
\\
&\qquad+\epsilon^2
\sum_{\bs k\in\mathbb N^2}
\frac{1-\ee^{-\lambda_{\bs k}\Delta_{n}}}{\lambda_{\bs k}^{1+\alpha}}
\biggl(\frac{\lambda_{\bs k}^{1/2}}{M_1\land M_2}\biggr)^{2\tau_2}1^{2(1-\tau_2)}
\\
&=
\frac{1}{(M_1\land M_2)^{2\tau_1}}
\sum_{\bs k\in\mathbb N^2}
\frac{(1-\ee^{-\lambda_{\bs k}\Delta_{n}})^2}{\lambda_{\bs k}^{1+(1-\tau_1)}}
+\frac{\epsilon^2}{(M_1\land M_2)^{2\tau_2}}
\sum_{\bs k\in\mathbb N^2}
\frac{1-\ee^{-\lambda_{\bs k}\Delta_{n}}}{\lambda_{\bs k}^{1+(\alpha-\tau_2)}}
\\
&=
O\biggl(\frac{n^{\tau_1-1}}{(M_1\land M_2)^{2\tau_1}}\biggr)
+O\biggl(\frac{n^{\tau_2-\alpha}\epsilon^2}{(M_1\land M_2)^{2\tau_2}}\biggr),
\end{align*}
and that for $\tau_3\in[0,1)$ and $\tau_4\in[0,\alpha)$,
\begin{align}
&\EE\bigl[
\{X_{\widetilde t_i}(y_{j_1},z_{j_2})-X_{\widetilde t_{i-1}}(y,z)\}^2\bigr]
\nonumber
\\
&=\sum_{\bs k,\bs \ell\in\mathbb N^2}
\EE\bigl[x_{\bs k}(\widetilde t_i)x_{\bs \ell}(\widetilde t_i)\bigr]
\{e_{\bs k}(y_{j_1},z_{j_2})-e_{\bs k}(y,z)\}
\{e_{\bs \ell}(y_{j_1},z_{j_2})-e_{\bs \ell}(y,z)\}
\nonumber
\\
&\lesssim
\sum_{\bs k\in\mathbb N^2}
\biggl(
\frac{1}{\lambda_{\bs k}^2}
+\frac{\epsilon^2}{\lambda_{\bs k}^{1+\alpha}}
\biggr)
\{e_{\bs k}(y_{j_1},z_{j_2})-e_{\bs k}(y,z)\}^2
\nonumber
\\
&\le
\sum_{\bs k\in\mathbb N^2}
\frac{1}{\lambda_{\bs k}^2}
\biggl(\frac{\lambda_{\bs k}^{1/2}}{M_1 \land M_2}\biggr)^{1+\tau_3}
1^{1-\tau_3}
\nonumber
\\
&\qquad+
\sum_{\bs k\in\mathbb N^2}
\frac{\epsilon^2}{\lambda_{\bs k}^{1+\alpha}}
\biggl(\frac{\lambda_{\bs k}^{1/2}}{M_1 \land M_2}\biggr)^{2\tau_4}1^{2(1-\tau_4)}
\nonumber
\\
&=
\frac{1}{(M_1 \land M_2)^{1+\tau_3}}
\sum_{\bs k\in\mathbb N^2}
\frac{1}{\lambda_{\bs k}^{(3-\tau_3)/2}}
+\frac{\epsilon^2}{(M_1 \land M_2)^{2\tau_4}}
\sum_{\bs k\in\mathbb N^2}
\frac{1}{\lambda_{\bs k}^{1+\alpha-\tau_4}}
\nonumber
\\
&=
O\biggl(
\frac{1}{(M_1 \land M_2)^{1+\tau_3}}
\biggr)
+O\biggl(
\frac{\epsilon^2}{(M_1 \land M_2)^{2\tau_4}}
\biggr).
\label{eq-0107}
\end{align}
By choosing $\tau_3>2\tau_1-1$ and $\tau_4>\tau_2$, it holds that
\begin{align*}
\frac{n^{\tau_1}r_{n,\epsilon}}{(M_1\land M_2)^{2\tau_1}}
\lor \frac{n^{-1} r_{n,\epsilon}}{(M_1 \land M_2)^{1+\tau_3}}
&=O\biggl(\frac{n^{\tau_1}r_{n,\epsilon}}{(M_1\land M_2)^{2\tau_1}}\biggr),
\\
\frac{n^{1-\alpha+\tau_2}\epsilon^2r_{n,\epsilon}}{(M_1\land M_2)^{2\tau_2}}
\lor \frac{n^{-1}\epsilon^2 r_{n,\epsilon}}{(M_1 \land M_2)^{2\tau_4}}
&=O\biggl(
\frac{n^{1-\alpha+\tau_2}\epsilon^2r_{n,\epsilon}}{(M_1\land M_2)^{2\tau_2}}
\biggr),
\end{align*}
and that under
$\frac{n^{\tau_1}r_{n,\epsilon}}{(M_1\land M_2)^{2\tau_1}}\to0$ 
and $\frac{n^{1-\alpha+\tau_2}\epsilon^2r_{n,\epsilon}}{(M_1\land M_2)^{2\tau_2}}\to0$,
\begin{align*}
r_{n,\epsilon}
\sum_{i=1}^{n}
\sup_{\lambda}\mathcal A_{3,i}^2(\lambda)
&=
O_p\biggl(
\frac{n^{\tau_1}r_{n,\epsilon}}{(M_1\land M_2)^{2\tau_1}}
\biggr)
+O_p\biggl(
\frac{n^{1-\alpha+\tau_2}\epsilon^2r_{n,\epsilon}}{(M_1\land M_2)^{2\tau_2}}
\biggr)
\\
&\qquad+
O_p\biggl(
\frac{n^{-1} r_{n,\epsilon}}{(M_1 \land M_2)^{1+\tau_3}}
\biggr)
+O_p\biggl(
\frac{n^{-1}\epsilon^2 r_{n,\epsilon}}{(M_1 \land M_2)^{2\tau_4}}
\biggr)
\\
&=o_p(1).
\end{align*}
Hence, the desired result is obtained. 

(2) It is shown that 
$\hat{\mathcal B}_n
\lesssim \sum_{i=0}^{n}(\mathcal B_{1,i}^2+\mathcal B_{2,i}^2+\mathcal B_{3,i}^2)$,
where
\begin{align*}
\mathcal B_{1,i}
&=
\frac{1}{M}\sum_{j_1=1}^{M_1}\sum_{j_2=1}^{M_2}
X_{\widetilde t_i}(y_{j_1},z_{j_2})\sin(\pi y_{j_1})\sin(\pi z_{j_2})
\\
&\qquad\qquad\times
\Biggl\{
\exp\biggl(
\frac{\hat\theta_1}{2\hat\theta_2}y_{j_1}
+\frac{\hat\eta_1}{2\hat\theta_2}z_{j_2}
\biggr)
-\exp\biggl(
\frac{\theta_1^*}{2\theta_2^*}y_{j_1}+\frac{\eta_1^*}{2\theta_2^*}z_{j_2}
\biggr)
\Biggr\},
\\
\mathcal B_{2,i}
&=
\sum_{j_1=1}^{M_1}\sum_{j_2=1}^{M_2}\iint_{D_{j_1,j_2}}
X_{\widetilde t_i}(y_{j_1},z_{j_2})
\\
&\qquad\qquad\times
\Biggl\{\sin(\pi y_{j_1})\sin(\pi z_{j_2})
\exp\biggl(
\frac{\theta_1^*}{2\theta_2^*}y_{j_1}+\frac{\eta_1^*}{2\theta_2^*}z_{j_2}
\biggr)
\\
&\qquad\qquad\qquad-
\sin(\pi y)\sin(\pi z)
\exp\biggl(
\frac{\theta_1^*}{2\theta_2^*}y+\frac{\eta_1^*}{2\theta_2^*}z
\biggr)
\Biggr\}
\dd y\dd z,
\\
\mathcal B_{3,i}
&=2\sum_{j_1=1}^{M_1}\sum_{j_2=1}^{M_2}\iint_{D_{j_1,j_2}}
\bigl\{X_{\widetilde t_i}(y_{j_1},z_{j_2})-X_{\widetilde t_i}(y,z)\bigr\}
\\
&\qquad\qquad\times
\sin(\pi y)\sin(\pi z)
\exp\biggl(
\frac{\theta_1^*}{2\theta_2^*}y+\frac{\eta_1^*}{2\theta_2^*}z
\biggr)
\dd y\dd z.
\end{align*}

We obtain from the Taylor expansion that 
\begin{align*}
\mathcal B_{1,i}^2
&\le \frac{1}{mN^{2\gamma}}\frac{1}{M}\sum_{j_1=1}^{M_1}\sum_{j_2=1}^{M_2}
X_{\widetilde t_i}^2(y_{j_1},z_{j_2})
\\
&\qquad\times
\Biggl|
\frac{1}{2(\theta_2^*)^2}\int_0^1(\theta_2^* y_{j_1},\theta_2^* z_{j_2},
-\theta_1^*y_{j_1}-\eta_1^*z_{j_2})
\\
&\qquad\qquad
\times
\exp\Biggl(
\frac{y_{j_1}}{2}
\frac{\theta_1^*+u(\hat\theta_1-\theta_1^*)}{\theta_2^*+u(\hat\theta_2-\theta_2^*)}
+\frac{z_{j_2}}{2}
\frac{\eta_1^*+u(\hat\eta_1-\eta_1^*)}{\theta_2^*+u(\hat\theta_2-\theta_2^*)}
\Biggr)\dd u
\Biggr|^2
\\
&\qquad\times
mN^{2\gamma}
(|\hat\theta_1-\theta_1^*|^2+|\hat\eta_1-\eta_1^*|^2+|\hat\theta_2-\theta_2^*|^2)
\\
&=:\mathcal B_{1,i}'
\times 
mN^{2\gamma}
(|\hat\theta_1-\theta_1^*|^2+|\hat\eta_1-\eta_1^*|^2+|\hat\theta_2-\theta_2^*|^2).
\end{align*}
Set $\mathcal B_n'=s_{n,\epsilon}\sum_{i=0}^{n}\mathcal B_{1,i}'$. 
Note that for any $\epsilon_1, \epsilon_2>0$, 
\begin{align*}
P(|\mathcal B_n'|>\epsilon_2)
&\lesssim
\frac{s_{n,\epsilon}}{\epsilon_2 m N^{2\gamma}}
\sum_{i=0}^{n}
\frac{1}{M}\sum_{j_1=1}^{M_1}\sum_{j_2=1}^{M_2}
\EE[X_{\widetilde t_i}^2(y_{j_1},z_{j_2})]
+P(\Omega_{\epsilon_1}^{\mathrm c})
\nonumber
\\
&\lesssim
\frac{n s_{n,\epsilon}}{\epsilon_2 mN^{2\gamma}}
+P(\Omega_{\epsilon_1}^{\mathrm c}).
\end{align*}
It holds from \eqref{eq-0106}, 
$\frac{n s_{n,\epsilon}}{m N^{2\gamma}}\to0$ and Theorem \ref{th1} that 
\begin{equation*}
s_{n,\epsilon}\sum_{i=0}^{n}\mathcal B_{1,i}^2
=o_p(1).
\end{equation*}
Since 
\begin{equation*}
\EE[\mathcal B_{2,i}^2]
\lesssim
\sum_{j_1=1}^{M_1}\sum_{j_2=1}^{M_2}
\EE[X_{\widetilde t_i}^2(y_{j_1},z_{j_2})]
\iint_{D_{j_1,j_2}}
\biggl(\frac{1}{M_1^2}+\frac{1}{M_2^2}\biggr)
\dd y\dd z
\lesssim
\frac{1}{(M_1 \land M_2)^2},
\end{equation*}
we obtain that under
$\frac{n s_{n,\epsilon}}{(M_1 \land M_2)^{1+\tau_3}}\to0$, 
\begin{equation*}
s_{n,\epsilon}
\sum_{i=1}^{n}
\mathcal B_{2,i}^2
=O_p\biggl(
\frac{n s_{n,\epsilon}}{(M_1 \land M_2)^2}
\biggr)
=o_p(1).
\end{equation*}
Note that
\begin{equation*}
\mathcal B_{3,i}^2
\lesssim
\sum_{j_1=1}^{M_1}\sum_{j_2=1}^{M_2}\iint_{D_{j_1,j_2}}
\bigl(X_{\widetilde t_i}(y_{j_1},z_{j_2})-X_{\widetilde t_i}(y,z)\bigr)^2
\dd y\dd z.
\end{equation*}
It follows from \eqref{eq-0107} that under
$\frac{n s_{n,\epsilon}}{(M_1 \land M_2)^{1+\tau_3}}\to0$
and $\frac{n\epsilon^2 s_{n,\epsilon}}{(M_1 \land M_2)^{2\tau_4}}\to0$, 
\begin{equation*}
s_{n,\epsilon}
\sum_{i=1}^{n}
\mathcal B_{3,i}^2
=O_p\biggl(
\frac{n s_{n,\epsilon}}{(M_1 \land M_2)^{1+\tau_3}}
\biggr)
+O_p\biggl(
\frac{n\epsilon^2 s_{n,\epsilon}}{(M_1 \land M_2)^{2\tau_4}}
\biggr)
=o_p(1).
\end{equation*}
Therefore, we get the desired result.
\end{proof}

Let
\begin{align*}
\hat{\mathcal X}_n 
&=\sup_{\lambda}\sum_{i=1}^n
\bigl|\hat{\mathcal M}_i(\lambda)^2-\mathcal M_i(\lambda)^2\bigr|,
\\
\hat{\mathcal Y}_n 
&=\sup_{\lambda}\sum_{i=1}^n
\bigl|\hat{\mathcal M}_i(\lambda)\hat x_{1,1}(\widetilde t_{i-1})
-\mathcal M_i(\lambda) x_{1,1}(\widetilde t_{i-1})\bigr|,
\\
\hat{\mathcal Z}_n
&=\sum_{i=1}^n
\bigl|\hat x_{1,1}(\widetilde t_{i-1})^2-x_{1,1}(\widetilde t_{i-1})^2\bigr|.
\end{align*}
We also set
\begin{equation}\label{eq-0201}
\mathcal C_n
=\sup_{\lambda}\sum_{i=1}^n \mathcal M_i(\lambda)^2,
\quad
\mathcal D_n
=\sum_{i=1}^n x_{1,1}(\widetilde t_{i-1})^2.
\end{equation}
Noting that
\begin{align*}
&\hat{\mathcal M}_{i}(\lambda)^2-\mathcal M_{i}(\lambda)^2
\\
&=\{\hat{\mathcal M}_{i}(\lambda)-\mathcal M_{i}(\lambda)\}^2
+2\{\hat{\mathcal M}_{i}(\lambda)-\mathcal M_{i}(\lambda)\}
\mathcal M_{i}(\lambda),
\\
&\hat{\mathcal M}_i(\lambda)\hat x_{1,1}(\widetilde t_{i-1})
-\mathcal M_i(\lambda)x_{1,1}(\widetilde t_{i-1})
\\
&=\{\hat{\mathcal M}_i(\lambda)-\mathcal M_i(\lambda)\}
\{\hat x_{1,1}(\widetilde t_{i-1})-x_{1,1}(\widetilde t_{i-1})\}
\\
&\qquad+
\mathcal M_i(\lambda)\{\hat x_{1,1}(\widetilde t_{i-1})-x_{1,1}(\widetilde t_{i-1})\}
+\{\hat{\mathcal M}_i(\lambda)-\mathcal M_i(\lambda)\}
x_{1,1}(\widetilde t_{i-1}),
\\
&\hat x_{1,1}(\widetilde t_{i-1})^2-x_{1,1}(\widetilde t_{i-1})^2
\\
&=\{\hat x_{1,1}(\widetilde t_{i-1})-x_{1,1}(\widetilde t_{i-1})\}^2
+2\{\hat x_{1,1}(\widetilde t_{i-1})-x_{1,1}(\widetilde t_{i-1})\}
x_{1,1}(\widetilde t_{i-1}),
\end{align*}
we have that
\begin{align}
\hat{\mathcal X}_n
&\lesssim
\sup_{\lambda}\sum_{i=1}^n
\{\hat{\mathcal M}_i(\lambda)-\mathcal M_i(\lambda)\}^2
\nonumber
\\
&\qquad+
\biggl(
\sup_{\lambda}\sum_{i=1}^n
\{\hat{\mathcal M}_i(\lambda)-\mathcal M_i(\lambda)\}^2
\biggr)^{1/2}
\biggl( \sup_{\lambda}\sum_{i=1}^n \mathcal M_i(\lambda)^2 \biggr)^{1/2}
\nonumber
\\
&=
\hat{\mathcal A}_n+\hat{\mathcal A}_n^{1/2}\mathcal C_n^{1/2},
\label{eq-0202}
\\
\hat{\mathcal Y}_n
&\lesssim
\biggl(
\sup_{\lambda}\sum_{i=1}^n
\{\hat{\mathcal M}_i(\lambda)-\mathcal M_i(\lambda)\}^2
\biggr)^{1/2}
\biggl(
\sum_{i=1}^n \{\hat x_{1,1}(\widetilde t_{i-1})-x_{1,1}(\widetilde t_{i-1})\}^2
\biggr)^{1/2}
\nonumber
\\
&\qquad+
\biggl( \sup_{\lambda}\sum_{i=1}^n \mathcal M_i(\lambda)^2 \biggr)^{1/2}
\biggl( \sum_{i=1}^n  
\{\hat x_{1,1}(\widetilde t_{i-1})-x_{1,1}(\widetilde t_{i-1})\}^2
\biggr)^{1/2}
\nonumber
\\
&\qquad+
\biggl(
\sup_{\lambda}\sum_{i=1}^n
\{\hat{\mathcal M}_i(\lambda)-\mathcal M_i(\lambda) \}^2
\biggr)^{1/2}
\biggl( \sum_{i=1}^n x_{1,1}(\widetilde t_{i-1})^2 \biggr)^{1/2}
\nonumber
\\
&=
\hat{\mathcal A}_n^{1/2}\hat{\mathcal B}_n^{1/2}
+\hat{\mathcal B}_n^{1/2}\mathcal C_n^{1/2}
+\hat{\mathcal A}_n^{1/2} \mathcal D_n^{1/2},
\label{eq-0203}
\\
\hat{\mathcal Z}_n
&\lesssim
\sum_{i=1}^n \{\hat x_{1,1}(\widetilde t_{i-1})-x_{1,1}(\widetilde t_{i-1})\}^2
\nonumber
\\
&\qquad
+\biggl( \sum_{i=1}^n \{\hat x_{1,1}(\widetilde t_{i-1})-x_{1,1}(\widetilde t_{i-1})\}^2 
\biggr)^{1/2}
\biggl( \sum_{i=1}^n x_{1,1}(\widetilde t_{i-1})^2 \biggr)^{1/2}
\nonumber
\\
&=\hat{\mathcal B}_n+\hat{\mathcal B}_n^{1/2}\mathcal D_n^{1/2}.
\label{eq-0204}
\end{align}
For the evaluation of $\mathcal C_n$ and $\mathcal D_n$, 
it follows from 
\begin{equation*}
\mathcal M_i(\lambda)=\mathcal M_i(\lambda^*)
+(\ee^{-\lambda^*\Delta_n}-\ee^{-\lambda\Delta_n})x_{1,1}(\widetilde{t}_{i-1}),
\end{equation*}
$\EE[\mathcal M_i(\lambda^*)^2] \lesssim \epsilon^2\Delta_n$
and
$\sup_{\lambda}|\ee^{-\lambda^*\Delta_n}-\ee^{-\lambda\Delta_n}|\lesssim \Delta_n$
that
\begin{equation*}
\mathcal C_n=O_p(\epsilon^2 \lor n^{-1}),
\quad 
\mathcal D_n=O_p(n).
\end{equation*}

\begin{proof}[\bf Proof of Theorem \ref{th2}]
Let $F(s)=s/(1-\ee^{-s})$ and $F_n(\lambda)=\lambda^\alpha F(2\lambda\Delta_n)$.
The differences between 
the contrast function, the score function, and the observed information 
based on the approximate coordinate process 
$\{\hat x_{1,1}(\widetilde t_{i})\}_{i=1}^n$ 
and those based on the coordinate process
$\{x_{1,1}(\widetilde t_{i})\}_{i=1}^n$ are defined as follows.
\begin{align*}
\mathcal S_1(n,\epsilon)
&=\epsilon^{2}\sup_{\lambda}
\bigl|V_{n,\epsilon}^{(1)}(\lambda|\hat x_{1,1})
-V_{n,\epsilon}^{(1)}(\lambda|x_{1,1})\bigr|,
\\
\mathcal S_2(n,\epsilon)
&=\epsilon^2 \sup_{\lambda}
\bigl|\partial_\lambda^2 V_{n,\epsilon}^{(1)}(\lambda|\hat x_{1,1}) 
-\partial_\lambda^2 V_{n,\epsilon}^{(1)}(\lambda|x_{1,1})
\bigr|,
\\
\mathcal S_3(n,\epsilon)
&=\epsilon \sup_{\lambda}
\bigl|\partial_\lambda V_{n,\epsilon}^{(1)}(\lambda|\hat x_{1,1})
-\partial_\lambda V_{n,\epsilon}^{(1)}(\lambda|x_{1,1})
\bigr|,
\\
\mathcal T_1(n,\epsilon)
&=\frac{1}{n}\sup_{\lambda}
\bigl|V_{n,\epsilon}^{(1)}(\lambda|\hat x_{1,1})
-V_{n,\epsilon}^{(1)}(\lambda|x_{1,1})\bigr|,
\\
\mathcal T_2(n,\epsilon)
&=\frac{1}{n}\sup_{\lambda}
\bigl|\partial_\lambda^2 V_{n,\epsilon}^{(1)}(\lambda|\hat x_{1,1}) 
-\partial_\lambda^2 V_{n,\epsilon}^{(1)}(\lambda|x_{1,1})
\bigr|,
\\
\mathcal T_3(n,\epsilon)
&=\frac{1}{\sqrt n} \sup_{\lambda}
\bigl|\partial_\lambda V_{n,\epsilon}^{(1)}(\lambda|\hat x_{1,1})
-\partial_\lambda V_{n,\epsilon}^{(1)}(\lambda|x_{1,1})
\bigr|.
\end{align*}
Since $V_{n,\epsilon}^{(1)}(\lambda|x_{1,1})$ and its derivatives 
with respect to $\lambda$ up to the second order 
can be expressed as
\begin{align*}
V_{n,\epsilon}^{(1)}(\lambda|x_{1,1})
&=\frac{F_n(\lambda)}{\epsilon^2\Delta_n}
\sum_{i=1}^n \mathcal M_i(\lambda)^2
-n \log F_n(\lambda),
\\
\partial_\lambda V_{n,\epsilon}^{(1)}(\lambda|x_{1,1})
&=\frac{F_n'(\lambda)}{\epsilon^2\Delta_n}
\sum_{i=1}^n \mathcal M_i(\lambda)^2
+\frac{2\ee^{-\lambda\Delta_n}F_n(\lambda)}{\epsilon^2}
\sum_{i=1}^n \mathcal M_i(\lambda)x_{1,1}(\widetilde t_{i-1})
-n\partial_\lambda \log F_n(\lambda),
\\
\partial_\lambda^2 V_{n,\epsilon}^{(1)}(\lambda|x_{1,1})
&=\frac{F_n''(\lambda)}{\epsilon^2\Delta_n}
\sum_{i=1}^n \mathcal M_i(\lambda)^2
+\frac{4\ee^{-\lambda\Delta_n}F_n'(\lambda)}{\epsilon^2}
\sum_{i=1}^n \mathcal M_i(\lambda)x_{1,1}(\widetilde t_{i-1})
\\
&\qquad+
\frac{2\ee^{-2\lambda\Delta_n}F_n(\lambda)\Delta_n}{\epsilon^2}
\sum_{i=1}^n x_{1,1}(\widetilde t_{i-1})^2
-n\partial_\lambda^2 \log F_n(\lambda),
\end{align*}
and it follows that 
\begin{equation*}
F_n(\lambda)\to\lambda^\alpha,
\quad
F_n'(\lambda)
\to\frac{\alpha}{\lambda^{1-\alpha}},
\quad
F_n''(\lambda)
\to\frac{\alpha(\alpha-1)}{\lambda^{2-\alpha}}
\end{equation*}
uniformly in $\lambda$ as $n\to\infty$, one has that
\begin{align}
\sup_{\lambda}
\bigl|V_{n,\epsilon}^{(1)}(\lambda|\hat x_{1,1})
-V_{n,\epsilon}^{(1)}(\lambda,\mu|x_{1,1})\bigr|
&\lesssim
n\epsilon^{-2}\hat{\mathcal X}_n,
\label{eq-0301}
\\
\sup_{\lambda}
\bigl|\partial_\lambda V_{n,\epsilon}^{(1)}(\lambda|\hat x_{1,1})
-\partial_\lambda V_{n,\epsilon}^{(1)}(\lambda|x_{1,1})\bigr|
&\lesssim
n\epsilon^{-2}\hat{\mathcal X}_n+\epsilon^{-2}\hat{\mathcal Y}_n,
\label{eq-0302}
\\
\sup_{\lambda}
\bigl|\partial_\lambda^2 V_{n,\epsilon}^{(1)}(\lambda|\hat x_{1,1})
-\partial_\lambda^2 V_{n,\epsilon}^{(1)}(\lambda|x_{1,1})\bigr|
&\lesssim
n\epsilon^{-2}\hat{\mathcal X}_n+\epsilon^{-2}\hat{\mathcal Y}_n
+(n\epsilon^2)^{-1}\hat{\mathcal Z}_n.
\label{eq-0303}
\end{align}

(a) The consistency of $\hat\theta_0$ is obtained 
by the consistency of $\hat\lambda_{1,1}$ and Theorem \ref{th1}, 
and hence we prove that $\hat\lambda_{1,1}$ is consistent.
For the consistency of $\hat\lambda_{1,1}$, 
it is enough to show that under [B1] and [C1],
\begin{equation}\label{eq-0304}
\mathcal S_1(n,\epsilon)=o_p(1),
\end{equation}
or that under [B2] and [C2],
\begin{equation}\label{eq-0305}
\mathcal T_1(n,\epsilon)=o_p(1).
\end{equation}
Indeed, since it follows from \eqref{ap-eq-0301} 
in the proof of Theorem \ref{ap-th1} that under [B1],
\begin{equation}
\sup_{\lambda}
\bigl|
\epsilon^2\{V_{n,\epsilon}^{(1)}(\lambda|x_{1,1})
-V_{n,\epsilon}^{(1)}(\lambda_{1,1}^*|x_{1,1})\}
-V_{1,1}(\lambda,\lambda_{1,1}^*)
\bigr|=o_p(1),
\label{eq-0306}
\end{equation}
where
\begin{equation*}
V_{1,1}(\lambda,\lambda^*)
=\lambda^\alpha(\lambda-\lambda^*)^2
\frac{1-\ee^{-2\lambda^*}}{2\lambda^*}x_{1,1}(0)^2,
\end{equation*}
it holds from \eqref{eq-0304} and \eqref{eq-0306} that
\begin{align*}
&\sup_{\lambda}
\bigl|\epsilon^2\{V_{n,\epsilon}^{(1)}(\lambda|\hat x_{1,1})
-V_{n,\epsilon}^{(1)}(\lambda_{1,1}^*|\hat x_{1,1})\}
-V_{1,1}(\lambda,\lambda_{1,1}^*)\bigr|
\\
&\le
\sup_{\lambda}
\bigl|\epsilon^2\{V_{n,\epsilon}^{(1)}(\lambda|x_{1,1})
-V_{n,\epsilon}^{(1)}(\lambda_{1,1}^*|x_{1,1})\}
-V_{1,1}(\lambda,\lambda_{1,1}^*)\bigr|
+2\mathcal S_1(n,\epsilon)
\\
&=o_p(1),
\end{align*}
which induces $\hat\lambda_{1,1}\pto\lambda_{1,1}^*$.
Similarly, if \eqref{eq-0305} holds, 
then it follows from \eqref{eq-0305} and \eqref{ap-eq-0302} that 
\begin{equation*}
\sup_{\lambda}
\biggl|
\frac{1}{n}V_{n,\epsilon}^{(1)}(\lambda|\hat x_{1,1})
-V_2(\lambda,\lambda_{1,1}^*)
\biggr|=o_p(1),
\end{equation*}
where
\begin{equation*}
V_{1,2}(\lambda,\lambda^*)
=\biggl(\frac{\lambda}{\lambda^*}\biggr)^\alpha
-\log \lambda^\alpha
+c\lambda^\alpha(\lambda-\lambda^*)^2
\frac{1-\ee^{-2\lambda^*}}{2\lambda^*}x_{1,1}(0)^2
\end{equation*}
and hence $\hat\lambda_{1,1}$ is consistent.

\textit{Proof of \eqref{eq-0304}. }
Since it holds from \eqref{eq-0301} that
\begin{equation}\label{eq-0307}
\mathcal S_1(n,\epsilon) \lesssim \epsilon^2(n\epsilon^{-2} \hat{\mathcal X}_n)
=n\hat{\mathcal X}_n
\end{equation}
and from \eqref{eq-0202} and $\mathcal C_n=O_p(n^{-1})$ under [B1] that 
\begin{equation*}
n \hat{\mathcal X}_n
\lesssim
n(\hat{\mathcal A}_n
+\hat{\mathcal A}_n^{1/2}\mathcal C_n^{1/2})
=n\hat{\mathcal A}_n
+(n\hat{\mathcal A}_n)^{1/2}
(n\mathcal C_n)^{1/2},
\end{equation*}
it is sufficient to show that $n\hat{\mathcal A}_n=o_p(1)$ under [C1].
This is verified by \eqref{eq-0102}.

\textit{Proof of \eqref{eq-0305}. }
Noting that
\begin{equation}\label{eq-0308}
\mathcal T_1(n,\epsilon) \lesssim n^{-1}(n\epsilon^{-2}\hat{\mathcal X}_n)
=\epsilon^{-2} \hat{\mathcal X}_n,
\end{equation}
$\mathcal C_n=O_p(\epsilon^2)$ under [B2] and
\begin{equation*}
\epsilon^{-2} \hat{\mathcal X}_n
\lesssim
\epsilon^{-2}\hat{\mathcal A}_n
+(\epsilon^{-2}\hat{\mathcal A}_n)^{1/2}(\epsilon^{-2}\mathcal C_n)^{1/2},
\end{equation*}
we obtain the desired result  from \eqref{eq-0103}.

(b) For the asymptotic normality of $\hat\theta_0$, 
it is sufficient to prove that under [B1] and [C3],
\begin{equation}\label{eq-0309}
\mathcal S_1(n,\epsilon)=o_p(1),
\quad
\mathcal S_2(n,\epsilon)=o_p(1),
\quad
\mathcal S_3(n,\epsilon)=o_p(1),
\end{equation}
or that under [B2] and [C3],
\begin{equation}\label{eq-0310}
\mathcal T_1(n,\epsilon)=o_p(1),
\quad
\mathcal T_2(n,\epsilon)=o_p(1),
\quad
\mathcal T_3(n,\epsilon)=o_p(1).
\end{equation}
Indeed, since it follows from \eqref{ap-eq-0401}-\eqref{ap-eq-0403} 
in the proof of Theorem \ref{ap-th1} that under [B1],
\begin{equation}
\epsilon^2\partial_\lambda^2 V_{n,\epsilon}^{(1)}(\lambda_{1,1}^*|x_{1,1})
\pto 2 G_1(\lambda_{1,1}^*),
\label{eq-0311}
\end{equation}
\begin{equation}
\epsilon^2\sup_{|\lambda-\lambda_{1,1}^*|<\delta_{n,\epsilon}}
\bigl|\partial_\lambda^2 V_{n,\epsilon}^{(1)}(\lambda|x_{1,1})
-\partial_\lambda^2 V_{n,\epsilon}^{(1)}(\lambda_{1,1}^*|x_{1,1})\bigr|
=o_p(1)
\ \text{ for }\delta_{n,\epsilon}\to0,
\label{eq-0312}
\end{equation}
\begin{equation}
\epsilon \partial_\lambda V_{n,\epsilon}^{(1)}(\lambda_{1,1}^*|x_{1,1})
\dto N(0,4G_1(\lambda_{1,1}^*)),
\label{eq-0313}
\end{equation}
it holds from the consistency of $\hat\lambda_{1,1}$,
\eqref{eq-0309} and \eqref{eq-0311}-\eqref{eq-0313} that
\begin{equation*}
\epsilon^2\partial_\lambda^2 V_{n,\epsilon}^{(1)}(\lambda_{1,1}^*|\hat x_{1,1})
\pto 2 G_1(\lambda_{1,1}^*)
\end{equation*}
\begin{equation*}
\epsilon^2\sup_{|\lambda-\lambda_{1,1}^*|<\delta_{n,\epsilon}}
\bigl|\partial_\lambda^2 V_{n,\epsilon}^{(1)}(\lambda|\hat x_{1,1})
-\partial_\lambda^2 V_{n,\epsilon}^{(1)}(\lambda_{1,1}^*|\hat x_{1,1})\bigr|
=o_p(1)
\ \text{ for }\delta_{n,\epsilon}\to0,
\end{equation*}
\begin{equation*}
\epsilon \partial_\lambda V_{n,\epsilon}^{(1)}(\lambda_{1,1}^*|\hat x_{1,1})
\dto N(0,4G_1(\lambda_{1,1}^*)),
\end{equation*}
which yield
\begin{equation*}
\epsilon^{-1}(\hat\lambda_{1,1}-\lambda_{1,1}^*)
\dto N(0,G_1(\lambda_{1,1}^*)^{-1}).
\end{equation*}
Hence, one has from 
$\theta_0=-\lambda_{1,1}+(\theta_1^2+\eta_1^2)/4\theta_2+2\pi^2\theta_2$
and Theorem \ref{th1} that under [C3],
\begin{align*}
\epsilon^{-1}(\hat\theta_0-\theta_0^*)
&=\epsilon^{-1}
\Biggl\{
-\hat\lambda_{1,1}
+\frac{\hat\theta_1^2+\hat\eta_1^2}{4\hat\theta_2}+2\pi^2\hat\theta_2
-\biggl(
-\lambda_{1,1}^*
+\frac{(\theta_1^*)^2+(\eta_1^*)^2}{4\theta_2^*}+2\pi^2\theta_2^*
\biggr)
\Biggr\}
\\
&=-\epsilon^{-1}(\hat\lambda_{1,1}-\lambda_{1,1}^*)
\\
&\qquad+
\frac{\epsilon^{-1}}{m^{1/2}N^\gamma}
m^{1/2}N^\gamma
\Biggl(
\frac{\hat\theta_1^2+\hat\eta_1^2}{4\hat\theta_2}+2\pi^2\hat\theta_2
-\frac{(\theta_1^*)^2+(\eta_1^*)^2}{4\theta_2^*}-2\pi^2\theta_2^*
\Biggr)
\\
&=-\epsilon^{-1}
(\hat\lambda_{1,1}-\lambda_{1,1}^*)
+o_p(1)
\\
&\dto N(0,G_1(\lambda_{1,1}^*)^{-1}),
\end{align*}
and the desired result can be obtained. 
Similarly, if \eqref{eq-0310} holds, then it follows 
from the consistency of $\hat\lambda_{1,1}$,
\eqref{ap-eq-0501}-\eqref{ap-eq-0503} and Theorem \ref{th1} that
\begin{equation*}
\sqrt n(\hat\theta_0-\theta_0^*)\dto N(0,I_1(\lambda_{1,1}^*)^{-1}).
\end{equation*}

\textit{Proof of \eqref{eq-0309}. }
It follows from \eqref{eq-0302} and \eqref{eq-0303} that
\begin{align}
\mathcal S_2(n,\epsilon)
&\lesssim
\epsilon^2
\{n\epsilon^{-2}\hat{\mathcal X}_n+\epsilon^{-2}\hat{\mathcal Y}_n
+(n\epsilon^2)^{-1}\hat{\mathcal Z}_n\}
=n \hat{\mathcal X}_n+\hat{\mathcal Y}_n+n^{-1}\hat{\mathcal Z}_n,
\label{eq-0314}
\\
\mathcal S_3(n,\epsilon)
&\lesssim
\epsilon(n\epsilon^{-2}\hat{\mathcal X}_n+\epsilon^{-2}\hat{\mathcal Y}_n)
=
n\epsilon^{-1}\hat{\mathcal X}_n+\epsilon^{-1}\hat{\mathcal Y}_n.
\label{eq-0315}
\end{align}
From \eqref{eq-0307}, \eqref{eq-0314}, \eqref{eq-0315},
$n \lor n\epsilon^{-1}=n\epsilon^{-1}$
and $1 \lor \epsilon^{-1}=\epsilon^{-1}$,
it suffices to show that under [B1] and [C3],
$n\epsilon^{-1}\hat{\mathcal X}_n=o_p(1)$, $\epsilon^{-1}\hat{\mathcal Y}_n=o_p(1)$
and $n^{-1} \hat{\mathcal Z}_n=o_p(1)$.
Since it holds from \eqref{eq-0202}-\eqref{eq-0204} and $\mathcal D_n=O_p(n)$ 
that under [B1],
\begin{align*}
n\epsilon^{-1}\hat{\mathcal X}_n
&\lesssim
n\epsilon^{-1}(\hat{\mathcal A}_n
+\hat{\mathcal A}_n^{1/2}\mathcal C_n^{1/2})
=n\epsilon^{-1}\hat{\mathcal A}_n+(n\epsilon^{-2}\hat{\mathcal A}_n)^{1/2}
(n\mathcal C_n)^{1/2},
\\
\epsilon^{-1}\hat{\mathcal Y}_n
&\lesssim
\epsilon^{-1}
(\hat{\mathcal A}_n^{1/2}\hat{\mathcal B}_n^{1/2}
+\hat{\mathcal B}_n^{1/2}\mathcal C_n^{1/2}
+\hat{\mathcal A}_n^{1/2}\mathcal D_n^{1/2})
\nonumber
\\
&=(n\epsilon^{-2}\hat{\mathcal A}_n)^{1/2}(n^{-1}\hat{\mathcal B}_n)^{1/2}
+\{(n\epsilon^2)^{-1}\hat{\mathcal B}_n\}^{1/2}(n\mathcal C_n)^{1/2}
+(n\epsilon^{-2}\hat{\mathcal A}_n)^{1/2}(n^{-1}\mathcal D_n)^{1/2},
\\
n^{-1} \hat{\mathcal Z}_n
&\lesssim
n^{-1}(\hat{\mathcal B}_n+\hat{\mathcal B}_n^{1/2}\mathcal D_n^{1/2})
=
n^{-1} \hat{\mathcal B}_n
+(n^{-1}\hat{\mathcal B}_n)^{1/2}(n^{-1}\mathcal D_n)^{1/2},
\end{align*}
and then it follows that $n\epsilon^{-1} \lor n\epsilon^{-2}=n\epsilon^{-2}$
and $n^{-1} \lor (n\epsilon^2)^{-1}=(n\epsilon^2)^{-1}$,
it suffices to prove that under [C3],
\begin{equation*}
n\epsilon^{-2}\hat{\mathcal A}_n=o_p(1),
\quad
(n\epsilon^2)^{-1}\hat{\mathcal B}_n=o_p(1),
\end{equation*}
which can be obtained from \eqref{eq-0104} and \eqref{eq-0105}.

\textit{Proof of \eqref{eq-0310}.}
By using \eqref{eq-0308} and the estimates that
\begin{align*}
\mathcal T_2(n,\epsilon)
&\lesssim
n^{-1}\{
n\epsilon^{-2}\hat{\mathcal X}_n+\epsilon^{-2}\hat{\mathcal Y}_n
+(n\epsilon^2)^{-1}\hat{\mathcal Z}_n
\}
=
\epsilon^{-2}\hat{\mathcal X}_n
+(n\epsilon^2)^{-1}\hat{\mathcal Y}_n
+(n\epsilon)^{-2}\hat{\mathcal Z}_n,
\\
\mathcal T_3(n,\epsilon)
&\lesssim
n^{-1/2}
(n\epsilon^{-2}\hat{\mathcal X}_n+\epsilon^{-2}\hat{\mathcal Y}_n)
=n^{1/2}\epsilon^{-2}\hat{\mathcal X}_n+(n^{1/2}\epsilon^2)^{-1} \hat{\mathcal Y}_n,
\end{align*}
$\epsilon^{-2} \lor n^{1/2}\epsilon^{-2}=n^{1/2}\epsilon^{-2}$
and $(n\epsilon^2)^{-1} \lor (n^{1/2}\epsilon^2)^{-1}=(n^{1/2}\epsilon^2)^{-1}$,
the proof of \eqref{eq-0310} is reduced to showing that
$n^{1/2}\epsilon^{-2}\hat{\mathcal X}_n=o_p(1)$,
$(n^{1/2}\epsilon^2)^{-1}\hat{\mathcal Y}_n=o_p(1)$
and $(n\epsilon)^{-2}\hat{\mathcal Z}_n=o_p(1)$
under [B2] and [C3]. 
These are shown by the following evaluations 
\begin{align*}
n^{1/2}\epsilon^{-2}\hat{\mathcal X}_n
&\lesssim
n^{1/2}\epsilon^{-2}\hat{\mathcal A}_n
+(n\epsilon^{-2}\hat{\mathcal A}_n)^{1/2}(\epsilon^{-2}\mathcal C_n)^{1/2},
\\
(n^{1/2}\epsilon^2)^{-1} \hat{\mathcal Y}_n
&\lesssim
(\epsilon^{-2}\hat{\mathcal A}_n)^{1/2}
\{(n\epsilon^2)^{-1}\hat{\mathcal B}_n\}^{1/2}
+\{(n\epsilon^2)^{-1}\hat{\mathcal B}_n\}^{1/2}(\epsilon^{-2}\mathcal C_n)^{1/2}
\\
&\qquad+(\epsilon^{-4}\hat{\mathcal A}_n)^{1/2}(n^{-1}\mathcal D_n)^{1/2},
\\
(n\epsilon)^{-2}\hat{\mathcal Z}_n
&\lesssim
(n\epsilon)^{-2} \hat{\mathcal A}_n
+\{(n^3\epsilon^4)^{-1}\hat{\mathcal A}_n\}^{1/2}(n^{-1}\mathcal C_n)^{1/2},
\end{align*}
\begin{equation*}
n^{1/2}\epsilon^{-2} \lor n\epsilon^{-2} \lor 
\epsilon^{-2}  \lor \epsilon^{-4} \lor 
(n\epsilon)^{-2} \lor (n^3\epsilon^4)^{-1}
=O(n\epsilon^{-2})
\end{equation*}
under [B2] and the properties
\begin{equation*}
n\epsilon^{-2}\hat{\mathcal A}_n=o_p(1),
\quad 
(n\epsilon^2)^{-1}\hat{\mathcal B}_n=o_p(1),
\end{equation*}
which are obtained from \eqref{eq-0104} and \eqref{eq-0105}.
\end{proof}

\subsection{Proofs of Proposition \ref{prop2}, Theorems \ref{th3} and \ref{th4}}
\label{sec5.2}
In this subsection, we give proofs of our results  
in Subsection \ref{sec3.2}.

\begin{proof}[\bf{Proof of Proposition \ref{prop2}}]
By setting $B_{i,\bs k}^{Q_2}=B_{1,i,\bs k}^{Q_2}+B_{2,i,\bs k}^{Q_2}$, where
\begin{align*}
B_{1,i,\bs k}^{Q_2}
&=
-\frac{\epsilon(1-\ee^{-\lambda_{\bs k}\Delta_N})}{\mu_{\bs k}^{\alpha/2}}
\int_0^{(i-1)\Delta_N}\ee^{-\lambda_{\bs k}((i-1)\Delta_N-s)}\dd w_{\bs k}(s),
\\
B_{2,i,\bs k}^{Q_2}
&=
\frac{\epsilon}{\mu_{\bs k}^{\alpha/2}}
\int_{(i-1)\Delta_N}^{i\Delta_N}
\ee^{-\lambda_{\bs k}(i\Delta_N-s)}\dd w_{\bs k}(s),
\end{align*}
it holds that uniformly in $(y,z)\in D_\delta$,
\begin{equation*}
\sum_{\bs k,\bs \ell\in\mathbb N^2}
\EE[B_{i,\bs k}^{Q_2} B_{i,\bs \ell}^{Q_2}]e_{\bs k}(y,z)e_{\bs \ell}(y,z)
=\epsilon^2\biggl\{
\Delta_N^{\alpha}\frac{\Gamma(1-\alpha)}{4\pi\alpha\theta_2^{1-\alpha}}
\exp\biggl(
-\frac{\theta_1}{\theta_2}y
\biggr)
\exp\biggl(
-\frac{\eta_1}{\theta_2}z
\biggr)
+r_{N,i}+O(\Delta_N)\biggr\},
\end{equation*}
where $\sum_{i=1}^N |r_{N,i}|=O(\Delta_N^{\beta})$ for any $\beta\in(0,1)$
similar to  Lemma \ref{lem2} and Lemma 5.6 in Tonaki et al. \cite{TKU2022arXiv}.
Therefore, it can be proved 
in an analogous way to the proof of Proposition \ref{prop1}.
\end{proof}

\begin{proof}[\bf Proof of Theorem \ref{th3}]
It can be shown in the same way as the proof of Theorem \ref{th1}. 
\end{proof}

Let $x_{\bs k}=x_{\bs k}^{Q_2}$ and 
$\tilde x_{\bs k}=\tilde x_{\bs k}^{Q_2}$.
In a similar way to Subsection \ref{sec5.1}, we set 
\begin{equation*}
\tilde{\mathcal M}_{i}(\lambda)
=\tilde x_{1,1}(\widetilde t_i)
-\ee^{-\lambda\Delta_n} \tilde x_{1,1}(\widetilde t_{i-1}),
\end{equation*}
\begin{equation*}
\tilde{\mathcal A}_n
=\sup_{\lambda}\sum_{i=1}^n
\{\tilde{\mathcal M}_i(\lambda)-\mathcal M_i(\lambda)\}^2,
\quad
\tilde{\mathcal B}_n
=\sum_{i=1}^n
\{\tilde x_{1,1}(\widetilde t_{i-1})-x_{1,1}(\widetilde t_{i-1})\}^2,
\end{equation*}
\begin{equation*}
\tilde{\mathcal X}_n 
=\sup_{\lambda}\sum_{i=1}^n
\bigl|\tilde{\mathcal M}_i(\lambda)^2-\mathcal M_i(\lambda)^2\bigr|,
\quad
\tilde{\mathcal Y}_n 
=\sup_{\lambda}\sum_{i=1}^n
\bigl|\tilde{\mathcal M}_i(\lambda)\tilde x_{1,1}(\widetilde t_{i-1})
-\mathcal M_i(\lambda) x_{1,1}(\widetilde t_{i-1})\bigr|,
\end{equation*}
\begin{equation*}
\tilde{\mathcal Z}_n
=\sum_{i=1}^n
\bigl|\tilde x_{1,1}(\widetilde t_{i-1})^2-x_{1,1}(\widetilde t_{i-1})^2\bigr|,
\end{equation*}
and obtain the following evaluations
\begin{equation}\label{eq-0401}
\tilde{\mathcal X}_n
\lesssim
\tilde{\mathcal A}_n+\tilde{\mathcal A}_n^{1/2}\mathcal C_n^{1/2},
\quad
\tilde{\mathcal Y}_n
\lesssim
\tilde{\mathcal A}_n^{1/2}\tilde{\mathcal B}_n^{1/2}
+\tilde{\mathcal B}_n^{1/2}\mathcal C_n^{1/2}
+\tilde{\mathcal A}_n^{1/2}\mathcal D_n^{1/2},
\quad
\tilde{\mathcal Z}_n
\lesssim
\tilde{\mathcal B}_n+\tilde{\mathcal B}_n^{1/2}\mathcal D_n^{1/2},
\end{equation}
where $\mathcal M_i(\lambda)$, $\mathcal C_n$ and $\mathcal D_n$ are given by 
\eqref{eq-0101} and \eqref{eq-0201}.

We set the following condition.
\begin{enumerate}
\item[\textbf{[C3$\boldsymbol{''}$]}]
$\frac{n}{m N^{2\gamma}}\to0$,
$\frac{n}{(M_1 \land M_2)^{1+\tau_3}}\to0$ 
and $\frac{n\epsilon^2}{(M_1 \land M_2)^{2\tau_4}}\to0$
for some $\tau_3\in[0,1)$ and $\tau_4\in[0,\alpha)$.
\end{enumerate}
In an analogous way to the proof of Lemma \ref{lem3}, the followings hold.
\begin{align}
n\tilde{\mathcal A}_n&=o_p(1)\quad \text{under [C1]},
\label{eq-0402}
\\
n\epsilon^{-2} \tilde{\mathcal A}_n&=o_p(1)\quad 
\text{under [C3]},
\label{eq-0403}
\\
(n\epsilon)^2\tilde{\mathcal A}_n&=o_p(1)\quad \text{under [C4]},
\label{eq-0404}
\\
(n\epsilon^4)^{-1}\tilde{\mathcal A}_n&=o_p(1)\quad \text{under [C5]},
\label{eq-0405}
\\
(n\epsilon^2)^{-1}\tilde{\mathcal B}_n&=o_p(1)\quad 
\text{under [C3]},
\label{eq-0406}
\\
\tilde{\mathcal B}_n&=o_p(1)\quad \text{under [C3$''$]}.
\nonumber
\end{align}
Note that if [C3] holds, then [C3$''$] is satisfied. It follows that 
\begin{equation}
\tilde{\mathcal B}_n=o_p(1)\quad \text{under [C3]}.
\label{eq-0407}
\end{equation}

\begin{proof}[\bf Proof of Theorem \ref{th4}]
(2) Let $F_n(\lambda,\mu)=\mu^\alpha F(2\lambda\Delta_n)$,
\begin{align*}
C_{n,\epsilon}^{(2)}(\lambda,\mu|x)
&=
\begin{pmatrix}
\epsilon^2 \partial_\lambda^2 V_{n,\epsilon}^{(2)}(\lambda,\mu|x) 
& \frac{\epsilon}{\sqrt n}
\partial_\lambda \partial_\mu V_{n,\epsilon}^{(2)}(\lambda,\mu|x)
\\
\frac{\epsilon}{\sqrt n}
\partial_\mu \partial_\lambda V_{n,\epsilon}^{(2)}(\lambda,\mu|x)
& \frac{1}{n} \partial_\mu^2 V_{n,\epsilon}^{(2)}(\lambda,\mu|x)
\end{pmatrix},
\\
K_{n,\epsilon}^{(2)}(\lambda,\mu|x)
&=
\begin{pmatrix}
-\epsilon\partial_\lambda V_{n,\epsilon}^{(2)}(\lambda,\mu|x)
\\
-\frac{1}{\sqrt n}\partial_\mu V_{n,\epsilon}^{(2)}(\lambda,\mu|x)
\end{pmatrix}.
\end{align*}
We set the difference between 
$V_{n,\epsilon}^{(2)}$, $C_{n,\epsilon}^{(2)}$, and $K_{n,\epsilon}^{(2)}$ 
based on the approximate coordinate process 
$\{\tilde x_{1,1}(\widetilde t_i)\}_{i=1}^n$ 
and those based on the coordinate process $\{x_{1,1}(\widetilde t_i)\}_{i=1}^n$ as
\begin{align*}
\mathcal U_1(n,\epsilon)
&=\epsilon^{2}\sup_{\lambda,\mu}
\bigl|V_{n,\epsilon}^{(2)}(\lambda,\mu|\tilde x_{1,1})
-V_{n,\epsilon}^{(2)}(\lambda,\mu|x_{1,1})\bigr|,
\\
\mathcal U_2(n,\epsilon)
&=\frac{1}{n}\sup_{\lambda,\mu}
\bigl|V_{n,\epsilon}^{(2)}(\lambda,\mu|\tilde x_{1,1})
-V_{n,\epsilon}^{(2)}(\lambda,\mu|x_{1,1})\bigr|,
\\
\mathcal U_3(n,\epsilon)
&=\sup_{\lambda,\mu}
\bigl|C_{n,\epsilon}^{(2)}(\lambda,\mu|\tilde x_{1,1})
-C_{n,\epsilon}^{(2)}(\lambda,\mu|x_{1,1})\bigr|,
\\
\mathcal U_4(n,\epsilon)
&=\sup_{\lambda,\mu}
\bigl|K_{n,\epsilon}^{(2)}(\lambda,\mu|\tilde x_{1,1})
-K_{n,\epsilon}^{(2)}(\lambda,\mu|x_{1,1})\bigr|.
\end{align*}
Noting that
\begin{align*}
V_{n,\epsilon}^{(2)}(\lambda,\mu|x)
&=\frac{F_n(\lambda,\mu)}{\epsilon^2\Delta_n}
\sum_{i=1}^n \mathcal M_i(\lambda)^2
-n \log F_n(\lambda,\mu),
\\
\partial_\lambda V_{n,\epsilon}^{(2)}(\lambda,\mu|x)
&=\frac{\partial_\lambda F_n(\lambda,\mu)}{\epsilon^2\Delta_n}
\sum_{i=1}^n \mathcal M_i(\lambda)^2
+\frac{2\ee^{-\lambda\Delta_n}F_n(\lambda,\mu)}{\epsilon^2}
\sum_{i=1}^n \mathcal M_i(\lambda)x(t_{i-1})
\\
&\qquad-n \partial_\lambda \log F_n(\lambda,\mu),
\\
\partial_\mu V_{n,\epsilon}^{(2)}(\lambda,\mu|x)
&=\frac{\partial_\mu F_n(\lambda,\mu)}{\epsilon^2\Delta_n}
\sum_{i=1}^n \mathcal M_i(\lambda)^2
-n \partial_\mu \log F_n(\lambda,\mu),
\\
\partial_\lambda^2 V_{n,\epsilon}^{(2)}(\lambda,\mu|x)
&=\frac{\partial_\lambda^2 F_n(\lambda,\mu)}{\epsilon^2\Delta_n}
\sum_{i=1}^n \mathcal M_i(\lambda)^2
+\frac{4\ee^{-\lambda\Delta_n}\partial_\lambda F_n(\lambda,\mu)}{\epsilon^2}
\sum_{i=1}^n \mathcal M_i(\lambda)x(t_{i-1})
\\
&\qquad+
\frac{2\ee^{-2\lambda\Delta_n}F_n(\lambda,\mu)\Delta_n}{\epsilon^2}
\sum_{i=1}^n x(t_{i-1})^2
-n \partial_\lambda^2 \log F_n(\lambda,\mu),
\\
\partial_\mu^2 V_{n,\epsilon}^{(2)}(\lambda,\mu|x)
&=\frac{\partial_\mu^2 F_n(\lambda,\mu)}{\epsilon^2\Delta_n}
\sum_{i=1}^n \mathcal M_i(\lambda)^2
-n \partial_\mu^2 \log F_n(\lambda,\mu),
\\
\partial_\mu\partial_\lambda V_{n,\epsilon}^{(2)}(\lambda,\mu|x)
&=\frac{\partial_\mu\partial_\lambda F_n(\lambda,\mu)}{\epsilon^2\Delta_n}
\sum_{i=1}^n \mathcal M_i(\lambda)^2
+\frac{2\ee^{-\lambda\Delta_n}\partial_\mu F_n(\lambda,\mu)}{\epsilon^2}
\sum_{i=1}^n \mathcal M_i(\lambda)x(t_{i-1})
\\
&\qquad
-n \partial_\mu \partial_\lambda \log F_n(\lambda,\mu),
\end{align*}
\begin{align*}
F_n(\lambda,\mu) &\to \mu^\alpha,
\quad
\Delta_n^{-1}\partial_\lambda F_n(\lambda,\mu) \to \mu^\alpha,
\quad
\partial_\mu F_n(\lambda,\mu) \to \alpha\mu^{\alpha-1},
\\
\Delta_n^{-2}\partial_\lambda^2 F_n(\lambda,\mu) &\to \frac{2\mu^\alpha}{3},
\quad
\Delta_n^{-1}\partial_\mu\partial_\lambda F_n(\lambda,\mu)
\to \alpha\mu^{\alpha-1},
\quad
\partial_\mu^2 F_n(\lambda,\mu) \to \alpha(\alpha-1)\mu^{\alpha-2},
\end{align*}
uniformly in $(\lambda,\mu)$ as $n\to\infty$, we have that
\begin{align}
\sup_{\lambda,\mu}
\bigl|V_{n,\epsilon}^{(2)}(\lambda,\mu|\tilde x_{1,1})
-V_{n,\epsilon}^{(2)}(\lambda,\mu|x_{1,1})\bigr|
&\lesssim
n\epsilon^{-2}\tilde{\mathcal X}_n,
\label{eq-0501}
\\
\sup_{\lambda,\mu}
\bigl|\partial_\lambda V_{n,\epsilon}^{(2)}(\lambda,\mu|\tilde x_{1,1})
-\partial_\lambda V_{n,\epsilon}^{(2)}(\lambda,\mu|x_{1,1})\bigr|
&\lesssim
\epsilon^{-2}(\tilde{\mathcal X}_n+\tilde{\mathcal Y}_n),
\label{eq-0502}
\\
\sup_{\lambda,\mu}
\bigl|\partial_\mu V_{n,\epsilon}^{(2)}(\lambda,\mu|\tilde x_{1,1})
-\partial_\mu V_{n,\epsilon}^{(2)}(\lambda,\mu|x_{1,1})\bigr|
&\lesssim
n\epsilon^{-2}\tilde{\mathcal X}_n,
\label{eq-0503}
\\
\sup_{\lambda,\mu}
\bigl|\partial_\lambda^2 V_{n,\epsilon}^{(2)}(\lambda,\mu|\tilde x_{1,1})
-\partial_\lambda^2 V_{n,\epsilon}^{(2)}(\lambda,\mu|x_{1,1})\bigr|
&\lesssim
(n\epsilon^2)^{-1}(\tilde{\mathcal X}_n+\tilde{\mathcal Y}_n+\tilde{\mathcal Z}_n),
\label{eq-0504}
\\
\sup_{\lambda,\mu}
\bigl|\partial_\mu^2 V_{n,\epsilon}^{(2)}(\lambda,\mu|\tilde x_{1,1})
-\partial_\mu^2 V_{n,\epsilon}^{(2)}(\lambda,\mu|x_{1,1})\bigr|
&\lesssim
n\epsilon^{-2}\tilde{\mathcal X}_n,
\label{eq-0505}
\\
\sup_{\lambda,\mu}
\bigl|\partial_\mu\partial_\lambda V_{n,\epsilon}^{(2)}(\lambda,\mu|\tilde x_{1,1})
-\partial_\mu\partial_\lambda V_{n,\epsilon}^{(2)}(\lambda,\mu|x_{1,1})\bigr|
&\lesssim
\epsilon^{-2}(\tilde{\mathcal X}_n+\tilde{\mathcal Y}_n).
\label{eq-0506}
\end{align}

(a) For proving the consistency of the estimators $\tilde\theta_0$ and $\tilde\mu_0$,
it is enough to show that under [C4] and [C5],
\begin{equation}\label{eq-0507}
\mathcal U_1(n,\epsilon)=o_p(1),
\quad
\mathcal U_2(n,\epsilon)=o_p(1).
\end{equation}

\textit{Proof of \eqref{eq-0507}.}
By using \eqref{eq-0501} and the fact that
\begin{equation}\label{eq-0508}
\mathcal U_1(n,\epsilon) \lor \mathcal U_2(n,\epsilon)
\lesssim
(\epsilon^2 \lor n^{-1})(n\epsilon^{-2}\tilde{\mathcal X}_n)
=(n\lor \epsilon^{-2})\tilde{\mathcal X}_n,
\end{equation}
it is sufficient to show that
$(n \lor \epsilon^{-2}) \tilde{\mathcal X}_n=o_p(1)$ under [C4] and [C5]. 
Since it holds from \eqref{eq-0401} and 
$\mathcal C_n=O_p(\epsilon^2\lor n^{-1})$ that
\begin{equation}\label{eq-0509}
(n\lor \epsilon^{-2})\tilde{\mathcal X}_n
\lesssim
(n\lor \epsilon^{-2})\tilde{\mathcal A}_n
+\{(n\lor \epsilon^{-2})^2(\epsilon^2 \lor n^{-1})\tilde{\mathcal A}_n\}^{1/2}
\{(\epsilon^2 \lor n^{-1})^{-1} \mathcal C_n\}^{1/2},
\end{equation}
\eqref{eq-0507} is obtained from
$n \lor \epsilon^{-2} \lor (n\lor \epsilon^{-2})^2(\epsilon^2 \lor n^{-1})
=(n\epsilon^4)^{-1} \lor (n\epsilon)^2$, 
\eqref{eq-0404} and \eqref{eq-0405}.

(b) Note that 
$\theta_0=-\lambda_{1,1}+(\theta_1^2+\eta_1^2)/4\theta_2+2\pi^2\theta_2$,
$\mu_0=\mu_{1,1}-2\pi^2$.
It holds from Theorem \ref{th3} that under [C3],
\begin{equation*}
\begin{pmatrix}
\epsilon^{-1}(\tilde\theta_0-\theta_0^*)
\\
\sqrt n(\tilde\mu_0-\mu_0^*)
\end{pmatrix}
=
\begin{pmatrix}
-1 & 0\\
0 & 1
\end{pmatrix}
\begin{pmatrix}
\epsilon^{-1}(\tilde\lambda_{1,1}-\lambda_{1,1}^*)
\\
\sqrt n(\tilde\mu_{1,1}-\mu_{1,1}^*)
\end{pmatrix}
+o_p(1).
\end{equation*}
Therefore, 
for the asymptotic normality of the estimators $\tilde\theta_0$ and $\tilde\mu_0$,
it suffices to prove that under [C3]-[C5],
\begin{equation}\label{eq-0510}
\mathcal U_1(n,\epsilon)=o_p(1), 
\quad
\mathcal U_2(n,\epsilon)=o_p(1),
\quad
\mathcal U_3(n,\epsilon)=o_p(1),
\quad
\mathcal U_4(n,\epsilon)=o_p(1).
\end{equation}

\textit{Proof of \eqref{eq-0510}. }
According to \eqref{eq-0501}-\eqref{eq-0506}, it follows that \eqref{eq-0508}, 
\begin{align*}
\mathcal U_3(n,\epsilon)
&\lesssim
\epsilon^2\{(n\epsilon^2)^{-1}(\tilde{\mathcal X}_n
+\tilde{\mathcal Y}_n+\tilde{\mathcal Z}_n)\}
+n^{-1}(n\epsilon^{-2}\tilde{\mathcal X}_n)
+\epsilon n^{-1/2}\{\epsilon^{-2}(\tilde{\mathcal X}_n+\tilde{\mathcal Y}_n)\}
\\
&\lesssim
\{n^{-1} \lor \epsilon^{-2} \lor (n^{1/2}\epsilon)^{-1}\}
\tilde{\mathcal X}_n
+\{n^{-1} \lor (n^{1/2}\epsilon)^{-1}\} \tilde{\mathcal Y}_n
+n^{-1} \tilde{\mathcal Z}_n,
\\
\mathcal U_4(n,\epsilon)
&\lesssim
\epsilon\{\epsilon^{-2}(\tilde{\mathcal X}_n+\tilde{\mathcal Y}_n)\}
+n^{-1/2}(n\epsilon^{-2}\tilde{\mathcal X}_n)
\\
&\lesssim
\{\epsilon^{-1} \lor (n^{1/2}\epsilon^2)^{-1}\} \tilde{\mathcal X}_n
+\epsilon^{-1} \tilde{\mathcal Y}_n.
\end{align*}
Since
$n \lor \epsilon^{-2} \lor n^{-1} 
\lor (n^{1/2}\epsilon)^{-1} \lor \epsilon^{-1}
\lor (n^{1/2}\epsilon^2)^{-1}
=n \lor \epsilon^{-2}$
and $n^{-1} \lor (n^{1/2}\epsilon)^{-1} \lor \epsilon^{-1}=\epsilon^{-1}$,
the proof of \eqref{eq-0510} is reduced to showing that under [C3]-[C5],
\begin{equation*}
(n \lor \epsilon^{-2})\tilde{\mathcal X}_n=o_p(1), 
\quad
\epsilon^{-1}\tilde{\mathcal Y}_n=o_p(1),
\quad
n^{-1} \tilde{\mathcal Z}_n=o_p(1).
\end{equation*}
These can be derived from \eqref{eq-0509},
the following evaluations
\begin{align*}
\epsilon^{-1}\tilde{\mathcal Y}_n
&\lesssim
(\epsilon^{-2}\tilde{\mathcal A}_n)^{1/2}\tilde{\mathcal B}_n^{1/2}
+\{(1 \lor (n\epsilon^2)^{-1})\tilde{\mathcal B}_n\}^{1/2}
\{(\epsilon^2 \lor n^{-1})^{-1} \mathcal C_n\}^{1/2}
\\
&\qquad
+(n\epsilon^{-2}\tilde{\mathcal A}_n)^{1/2}(n^{-1}\mathcal D_n)^{1/2},
\\
n^{-1} \tilde{\mathcal Z}_n
&\lesssim
n^{-1} \tilde{\mathcal B}_n
+(n^{-1}\tilde{\mathcal B}_n)^{1/2}(n^{-1}\mathcal D_n)^{1/2},
\end{align*}
\begin{equation*}
(n\epsilon^4)^{-1} \lor (n\epsilon)^2 
\lor \epsilon^{-2} \lor n\epsilon^{-2}
=(n\epsilon^4)^{-1} \lor (n\epsilon)^2 \lor n\epsilon^{-2},
\quad
1 \lor (n\epsilon^2)^{-1} \lor n^{-1}=1 \lor (n\epsilon^2)^{-1},
\end{equation*}
and the properties that under [C3]-[C5], 
\begin{equation*}
\{(n\epsilon^4)^{-1} \lor (n\epsilon)^2 \lor n\epsilon^{-2}\} 
\tilde{\mathcal A}_n=o_p(1),
\quad
\{1 \lor (n\epsilon^2)^{-1}\} \tilde{\mathcal B}_n=o_p(1),
\end{equation*}
which are revealed by \eqref{eq-0403}-\eqref{eq-0407}.

(1) Set
\begin{align*}
\mathcal V_1(n,\epsilon)
&=\epsilon^{2}\sup_{\lambda}
\bigl|V_{n,\epsilon}^{(2)}(\lambda,\mu_{1,1}|\check x_{1,1})
-V_{n,\epsilon}^{(2)}(\lambda,\mu_{1,1}|x_{1,1})\bigr|,
\\
\mathcal V_2(n,\epsilon)
&=
\epsilon^2 \sup_{\lambda}
\bigl|\partial_\lambda^2 V_{n,\epsilon}^{(2)}(\lambda,\mu_{1,1}|\check x_{1,1})
-\partial_\lambda^2 V_{n,\epsilon}^{(2)}(\lambda_{1,1}^*,\mu_{1,1}|x_{1,1})\bigr|,
\\
\mathcal V_3(n,\epsilon)
&=\epsilon\sup_{\lambda}
\bigl|\partial_\lambda V_{n,\epsilon}^{(2)}(\lambda,\mu_{1,1}|\check x_{1,1})
-\partial_\lambda V_{n,\epsilon}^{(2)}(\lambda,\mu_{1,1}|x_{1,1})\bigr|.
\end{align*}

(a) 
For the consistency of $\tilde\theta_0$, 
it is enough to prove that under [C1] and [C4],
\begin{equation}\label{eq-0511}
\mathcal V_1(n,\epsilon)=o_p(1).
\end{equation}

\textit{Proof of \eqref{eq-0511}. }
Since
\begin{equation}\label{eq-0512}
\mathcal V_1(n,\epsilon) \lesssim \epsilon^2 (n\epsilon^{-2} \tilde{\mathcal X}_n)
=n\tilde{\mathcal X}_n, 
\end{equation}
$\mathcal C_n=O_p(\epsilon^2 \lor n^{-1})$,
$n\tilde{\mathcal X}_n
\lesssim
n\tilde{\mathcal A}_n
+\{((n\epsilon)^2 \lor n)\tilde{\mathcal A}_n\}^{1/2}
\{(\epsilon^2 \lor n^{-1})^{-1} \mathcal C_n\}^{1/2}$,
\eqref{eq-0402} and \eqref{eq-0404}, 
we conclude the proof of \eqref{eq-0511}.

(b) For the asymptotic normality of $\tilde\theta_0$, 
it suffices to show that under [C3] and [C4],
\begin{equation}\label{eq-0513}
\mathcal V_1(n,\epsilon)=o_p(1),
\quad
\mathcal V_2(n,\epsilon)=o_p(1),
\quad
\mathcal V_3(n,\epsilon)=o_p(1).
\end{equation}

\textit{Proof of \eqref{eq-0513}. }
We obtain from \eqref{eq-0502} and \eqref{eq-0504} that 
\begin{align*}
\mathcal V_2(n,\epsilon)
&\lesssim 
\epsilon^2\{(n\epsilon^2)^{-1}(\tilde{\mathcal X}_n
+\tilde{\mathcal Y}_n+\tilde{\mathcal Z}_n)\}
=n^{-1} (\tilde{\mathcal X}_n+\tilde{\mathcal Y}_n+\tilde{\mathcal Z}_n),
\\
\mathcal V_3(n,\epsilon)
&\lesssim 
\epsilon\{\epsilon^{-2} (\tilde{\mathcal X}_n+\tilde{\mathcal Y}_n)\}
=\epsilon^{-1} (\tilde{\mathcal X}_n+\tilde{\mathcal Y}_n).
\end{align*}
Therefore, from \eqref{eq-0512} and
$n \lor n^{-1} \lor \epsilon^{-1}=n \lor \epsilon^{-1}$,
it is sufficient to prove that 
$(n\lor\epsilon^{-1})\tilde{\mathcal X}_n=o_p(1)$,
$(n^{-1} \lor \epsilon^{-1})\tilde{\mathcal Y}_n=o_p(1)$ 
and $n^{-1}\tilde{\mathcal Z}_n=o_p(1)$
under [C3] and [C4]. 
Since
\begin{align*}
(n\lor\epsilon^{-1})\tilde{\mathcal X}_n
&\lesssim
(n \lor \epsilon^{-1})\tilde{\mathcal A}_n
+\{(n \lor \epsilon^{-1})^2(\epsilon^2 \lor n^{-1})\tilde{\mathcal A}_n\}^{1/2}
\{(\epsilon^2 \lor n^{-1})^{-1} \mathcal C_n\}^{1/2},
\\
\epsilon^{-1}\tilde{\mathcal Y}_n
&\lesssim
(\epsilon^{-2}\tilde{\mathcal A}_n)^{1/2} \tilde{\mathcal B}_n^{1/2}
+\{(1 \lor (n\epsilon^2)^{-1})\tilde{\mathcal B}_n\}^{1/2}
\{(\epsilon^2 \lor n^{-1})^{-1}\mathcal C_n\}^{1/2}
\\
&\qquad
+(n\epsilon^{-2}\tilde{\mathcal A}_n)^{1/2}(n^{-1}\mathcal D_n)^{1/2},
\\
n^{-1}\tilde{\mathcal Z}_n
&\lesssim
n^{-1} \tilde{\mathcal B}_n
+(n^{-1}\tilde{\mathcal B}_n)^{1/2}(n^{-1}\mathcal D_n)^{1/2},
\end{align*}
\begin{equation*}
(n \lor \epsilon^{-1}) 
\lor (n \lor \epsilon^{-1})^2(\epsilon^2 \lor n^{-1})
\lor \epsilon^{-2} \lor n\epsilon^{-2}
=(n\epsilon)^2 \lor n\epsilon^{-2}
\end{equation*}
and 
\begin{equation*}
1 \lor (n\epsilon^2)^{-1} \lor n^{-1}
= 1 \lor (n\epsilon^2)^{-1},
\end{equation*}
it eventually suffices to prove that under [C3] and [C4],
\begin{equation*}
\{(n\epsilon)^2 \lor n\epsilon^{-2}\} \tilde{\mathcal A}_n=o_p(1),
\quad
\{1\lor(n\epsilon^2)^{-1}\}\tilde{\mathcal B}_n=o_p(1).
\end{equation*}
By using \eqref{eq-0403}, \eqref{eq-0404}, \eqref{eq-0406} and \eqref{eq-0407}, 
we complete the proof of \eqref{eq-0513}.
\end{proof}


\setcounter{section}{0}
\setcounter{subsection}{0}
\setcounter{thm}{0}
\setcounter{lem}{0}
\setcounter{rmk}{0}
\setcounter{equation}{0}
\renewcommand{\thethm}{A.\arabic{thm}}
\renewcommand{\thelem}{A.\arabic{lem}}
\renewcommand{\thermk}{\Alph{rmk}}
\renewcommand{\thesection}{\Alph{section}}
\renewcommand{\theequation}{%
A.\arabic{equation}}
\renewcommand{\thesubsection}{%
A.\arabic{subsection}}

\section*{Appendix I}
The rest of this paper is devoted to 
parameter estimation for a diffusion process with a small noise $\{x(t)\}_{t\in[0,1]}$
defined by the following stochastic differential equation
\begin{equation}\label{small_OU}
\dd x(t)
=-\lambda x(t)\dd t
+\epsilon\upsilon^{-\alpha/2}\dd w(t),
\quad t\in[0,1],
\end{equation}
where 
$\zeta=(\lambda,\upsilon)$ is an unknown parameter, 
$\epsilon\in(0,1)$ and $\alpha\in(0,1)$ are known constants,
$\{w(t)\}_{t\ge0}$ is 
a one-dimensional
standard Brownian motion.
We assume that the process $\{x(t)\}_{t\in[0,1]}$ is discretely observed 
at time points $t_i=i\Delta_n$, $i=0,1,\ldots,n$, where $\Delta_n=1/n$.

Parameter estimation for diffusion processes with a small noise 
based on discrete observations has been studied by many researchers,
see 
Genon-Catalot \cite{Genon-Catalot1990},
Laredo \cite{Laredo1990},
S{\o}rensen and Uchida \cite{Sorensen_Uchida2003},
Uchida \cite{Uchida2003}, \cite{Uchida2004},
Gloter and S{\o}rensen \cite{Gloter_Sorensen2009},
Guy et al. \cite{Guy_etal2014},
Nomura and Uchida \cite{Nomura_Uchida2016},
Kaino and Uchida \cite{Kaino_Uchida2018}
Kawai and Uchida \cite{Kawai_Uchida2022} and reference therein.
In particular, Uchida \cite{Uchida2003} estimated a parameter 
appearing in both the drift and diffusion coefficients 
under $(n\epsilon)^{-1}=O(1)$,
and Gloter and S{\o}rensen \cite{Gloter_Sorensen2009} 
considered joint estimation for both the drift and diffusion parameters
under $(n\epsilon^\rho)^{-1}=O(1)$ for some $\rho>0$.
In the previous studies, they dealt with general diffusion process models
and hence they considered parameter estimation based on 
the approximate martingale estimating function.
However, since the solution of \eqref{small_OU} can be explicitly expressed as 
\begin{equation}\label{OU_ex1}
x(t)=\ee^{-\lambda t}x(0)
+\epsilon \int_0^t \upsilon^{-\alpha/2} \ee^{-\lambda(t-s)} \dd w(s)
\end{equation}
and it follows that
\begin{equation}\label{OU_ex2}
x(t_i)-\ee^{-\lambda \Delta_n}x(t_{i-1})
=\epsilon\int_{t_{i-1}}^{t_i}\upsilon^{-\alpha/2}\ee^{-\lambda(t_i-s)}\dd w(s)
\sim N\biggl(0,\frac{\epsilon^2(1-\ee^{-2\lambda\Delta_n})}
{2\lambda \upsilon^{\alpha}}\biggr),
\end{equation}
it is not necessary to impose the conditions 
such as $(n\epsilon)^{-1}=O(1)$ and $(n\epsilon^\rho)^{-1}=O(1)$ 
assumed in the previous studies 
to evaluate the approximation error 
by constructing the contrast function based on \eqref{OU_ex2}. 
Therefore, we consider parameter estimation without such conditions. 
Specifically, we study estimation of the parameters 
in the diffusion process \eqref{small_OU} in the following two cases: 
\begin{description}
\item[Case 1]
$\lambda$ appears in both the drift and diffusion coefficients
($\upsilon=\lambda$),

\item[Case 2]
$\lambda$ only appears in the drift coefficient
($\upsilon=\mu$, and $\mu$ may be known),
\end{description}
under weaker conditions than those 
in Uchida \cite{Uchida2003} 
and Gloter and S{\o}rensen \cite{Gloter_Sorensen2009}, respectively.
We treat parameter estimation for 
Cases 1 and 2 in Subsections \ref{secA1} and \ref{secA2}, respectively.

Let $(\Omega,\mathscr F, \{{\mathscr F}_t\}_{t\ge0}, P)$ 
be a stochastic basis with usual conditions,
and let $\{w(t)\}_{t\ge0}$ be independent real valued standard Brownian motion 
on this basis.

\subsection{The case where $\lambda$ appears in both coefficients (Case 1)}\label{secA1}
In this subsection, we deal with  a one-dimensional diffusion process 
$\{x_1(t)\}_{t\in[0,1]}$ satisfying 
the following stochastic differential equation
\begin{equation}\label{small_OU_ver1}
\dd x_1(t)
=-\lambda x_1(t)\dd t
+\epsilon\lambda^{-\alpha/2}\dd w(t), 
\quad t\in[0,1],
\end{equation}
where 
$\lambda\in\Xi$, the parameter space $\Xi$ is a compact convex subset of $(0,\infty)$,
$\lambda^*\in \mathrm{Int}\,\Xi$ is the true value of $\lambda$, and
$\epsilon\in(0,1)$, $x_1(0)\neq0$ and $\alpha\in(0,1)$ are known constants.

We consider the following asymptotics for $n$ and $\epsilon$. 
\begin{enumerate}
\item[\textbf{[B1]}]
$\lim_{n\to\infty,\epsilon\to0}n\epsilon^2=0$.

\item[\textbf{[B2]}]
$\varlimsup_{n\to\infty,\epsilon\to0}(n\epsilon^2)^{-1}<\infty$, that is,

\noindent
(I) $\lim_{n\to\infty,\epsilon\to0}(n\epsilon^2)^{-1}=0$, or
(II) $0<\lim_{n\to\infty,\epsilon\to0}(n\epsilon^2)^{-1}<\infty$.
\end{enumerate}
The contrast function is as follows. 
\begin{equation*}
V_{n,\epsilon}^{(1)}(\lambda|x)
=\sum_{i=1}^n
\frac{(x(t_i)-\ee^{-\lambda \Delta_n}x(t_{i-1}))^2}
{\frac{\epsilon^2(1-\ee^{-2\lambda\Delta_n})}{2\lambda^{1+\alpha}}}
+n\log \frac{1-\ee^{-2\lambda\Delta_n}}{2\lambda^{1+\alpha}\Delta_n}.
\end{equation*}
Set 
\begin{equation*}
\hat \lambda
=\underset{\lambda}
{\mathrm{arginf}}\, V_{n,\epsilon}^{(1)}(\lambda|x_1)
\end{equation*}
as the estimator of $\lambda$. Define 
\begin{equation*}
G_1(\lambda)=\frac{1-\ee^{-2\lambda}}{2\lambda^{1-\alpha}}x_1(0)^2,
\quad
H_1(\lambda)=\frac{\alpha^2}{2\lambda^2}.
\end{equation*}
Set $I_1(\lambda)=H_1(\lambda)+c\, G_1(\lambda)$ under [B2], 
where $c=\lim_{n\to\infty,\epsilon\to0}(n\epsilon^2)^{-1}$. 

\begin{thm}\label{ap-th1}\  
\begin{enumerate}
\item[(i)]
If [B1] holds, then as $n\to\infty$ and $\epsilon\to0$,
\begin{equation*}
\epsilon^{-1}(\hat\lambda-\lambda^*)
\dto N(0,G_1(\lambda^*)^{-1}).
\end{equation*}

\item[(ii)]
If [B2] holds, then as $n\to\infty$ and $\epsilon\to0$,
\begin{equation*}
\sqrt{n}(\hat\lambda-\lambda^*)
\dto N(0,I_1(\lambda^*)^{-1}).
\end{equation*}
In particular, if [B2](I) holds, then
$\sqrt n(\hat\lambda-\lambda^*) \dto N(0,H_1(\lambda^*)^{-1})$
as $n\to\infty$ and $\epsilon\to0$,
and if [B2](II) holds, then 
$\epsilon^{-1}(\hat\lambda-\lambda^*) \dto N(0,c I_1(\lambda^*)^{-1})$
as $n\to\infty$ and $\epsilon\to0$.
\end{enumerate}
\end{thm}

\begin{rmk}\label{ap-rmk1}
By comparing Theorem \ref{ap-th1} (i) with Corollary 1 in Uchida \cite{Uchida2003},
we can see that Corollary 1 in \cite{Uchida2003} requires the condition
$\lim_{n\to\infty,\epsilon\to0}(n\epsilon)^{-1}=0$
for asymptotic normality of the estimator, while Theorem \ref{ap-th1} (i) does not.
This is because \cite{Uchida2003} constructed the contrast function based on 
the Euler-Maruyama approximation, whereas we construct that based on 
the explicit representation 
of the solution of \eqref{small_OU_ver1}.
See Remark \ref{ap-rmk3} for details.
\end{rmk}

\subsection{The case where $\lambda$ appears in drift coefficient (Case 2)}\label{secA2}
In this subsection, we treat a one-dimensional diffusion process 
$\{x_2(t)\}_{t\in[0,1]}$ defined by the stochastic differential equation
\begin{equation}\label{small_OU_ver2}
\dd x_2(t)=-\lambda x_2(t)\dd t+\epsilon\mu^{-\alpha/2}\dd w(t),
\quad t\in[0,1],
\end{equation}
where 
$(\lambda,\mu)\in\Xi$, the parameter space $\Xi$ is a compact convex subset 
of $(0,\infty)^2$,
$(\lambda^*,\mu^*)\in \mathrm{Int}\,\Xi$ is the true value of $(\lambda,\mu)$,
and $\epsilon\in(0,1)$, $x_2(0)\neq0$ and $\alpha\in(0,1)$ are known constants.

The contrast function is as follows. 
\begin{equation*}
V_{n,\epsilon}^{(2)}(\lambda,\mu|x)
=\sum_{i=1}^n
\frac{(x(t_i)-\ee^{-\lambda \Delta_n}x(t_{i-1}))^2}
{\frac{\epsilon^2(1-\ee^{-2\lambda\Delta_n})}{2\lambda\mu^\alpha}}
+n\log \frac{1-\ee^{-2\lambda\Delta_n}}{2\lambda\mu^\alpha\Delta_n}.
\end{equation*}
If $\mu$ is known, then let
\begin{equation*}
\tilde \lambda
=\underset{\lambda}
{\mathrm{arginf}}\, V_{n,\epsilon}^{(2)}(\lambda,\mu|x_2)
\end{equation*}
as the estimator of $\lambda$, or if $\mu$ is unknown, then let  
\begin{equation*}
(\tilde \lambda, \tilde \mu)
=\underset{\lambda,\mu}
{\mathrm{arginf}}\, V_{n,\epsilon}^{(2)}(\lambda,\mu|x_2)
\end{equation*}
as the estimator of $(\lambda,\mu)$. Moreover, set
\begin{align*}
G_2(\lambda,\mu)=\frac{1-\ee^{-2\lambda}}{2\lambda}\mu^\alpha x_2(0)^2,
\quad
H_2(\mu)=\frac{\alpha^2}{2\mu^2},
\quad
I_2(\lambda,\mu)=\mathrm{diag}\{G_2(\lambda,\mu),H_2(\mu) \}.
\end{align*}

\begin{thm}\label{ap-th2}\  
\begin{enumerate}
\item[(1)]
If $\mu$ is known, 
then as $n\to\infty$ and $\epsilon\to0$,
\begin{equation*}
\epsilon^{-1}(\tilde \lambda-\lambda^*)
\dto N(0,G_2(\lambda^*,\mu)^{-1}).
\end{equation*}

\item[(2)]
If $\mu$ is unknown, then as $n\to\infty$ and $\epsilon\to0$,
\begin{equation*}
\begin{pmatrix}
\epsilon^{-1}(\tilde \lambda-\lambda^*)
\\
\sqrt{n}(\tilde \mu-\mu^*)
\end{pmatrix}
\dto N(0,I_2(\lambda^*,\mu^*)^{-1}).
\end{equation*}

\end{enumerate}
\end{thm}

\begin{rmk}\label{ap-rmk2}
By comparing Theorem \ref{ap-th2} (2) with 
Theorem 1 in Gloter and S{\o}rensen \cite{Gloter_Sorensen2009},
we notice that 
Theorem 1 in \cite{Gloter_Sorensen2009} imposes the condition
that there exists $\rho>0$ such that 
$\varlimsup_{n\to0,\epsilon\to0}(\epsilon n^\rho)^{-1}<\infty$
for asymptotic normality of the estimator,
but Theorem \ref{ap-th2} (2) does not.
This is because, as we mentioned in Remark \ref{ap-rmk1}, 
\cite{Gloter_Sorensen2009} constructed the contrast function based on 
the It{\^o}-Taylor expansion, 
while our contrast function is based on 
the explicit likelihood of \eqref{small_OU_ver2}.
See Remark \ref{ap-rmk3} for details.
\end{rmk}

\subsection{Proofs of Theorems \ref{ap-th1} and \ref{ap-th2}}
In this subsection, we give proofs of the assertions 
in Subsections \ref{secA1} and \ref{secA2}.
Our proofs are based on 
S{\o}rensen and Uchida \cite{Sorensen_Uchida2003},
Uchida \cite{Uchida2003} and
Gloter and S{\o}rensen \cite{Gloter_Sorensen2009}.

Let $C^{k,\ell}_{\uparrow}(\mathbb R\times\Xi)$ be 
the space of all functions $f$ satisfying the following conditions.

\begin{enumerate}
\item[(i)] 
$f$ is continuously differentiable with respect to $x\in\mathbb R$ up to
order $k$ for all $\zeta\in\Xi$,
\item[(ii)]  
$f$ and all its $x$-derivatives up to order $k$ are 
$\ell$ times continuously differentiable with respect to $\zeta\in\Xi$,
\item[(iii)]  
$f$ and all derivatives are of polynomial growth in $x\in\mathbb R$ 
uniformly in $\zeta\in\Xi$.
\end{enumerate}

For $\{x(t)\}_{t\in[0,1]}$ defined by \eqref{small_OU}, let
\begin{equation*}
\mathcal M_{i}(\lambda)=x(t_i)-\ee^{-\lambda\Delta_n} x(t_{i-1}).
\end{equation*}
Moreover, let
$\zeta^*=(\lambda^*,\upsilon^*)$ be the true value of $\zeta=(\lambda,\upsilon)$
and $x_0(t)=x(0)\ee^{-\lambda^* t}$.

\begin{lem}\label{ap-lem1}
Let $f\in C_{\uparrow}^{1,1}(\mathbb R\times\Xi)$.
Then, the followings hold.
\begin{enumerate}
\item[(1)] As $n\to\infty$ and $\epsilon\to0$,
\begin{equation*}
\frac{1}{n}\sum_{i=1}^n f(x(t_{i-1}),\zeta)
\pto \int_0^1 f(x_0(s),\zeta)\dd s
\quad\text{uniformly in }\zeta.
\end{equation*}

\item[(2)] As $n\to\infty$ and $\epsilon\to0$,
\begin{equation*}
\sum_{i=1}^n f(x(t_{i-1}),\zeta)\mathcal M_i(\lambda^*)\pto 0
\quad\text{uniformly in }\zeta.
\end{equation*}

\item[(3)] As $n\to\infty$ and $\epsilon\to0$,
\begin{equation*}
\epsilon^{-2}\sum_{i=1}^n f(x(t_{i-1}),\zeta)\mathcal M_i(\lambda^*)^2
\pto \frac{1}{(\upsilon^*)^\alpha}\int_0^1 f(x_0(s),\zeta) \dd s
\quad\text{uniformly in }\zeta.
\end{equation*}
\end{enumerate}
\end{lem}

\begin{proof}
In the same way as the proof of Lemma 2 in S{\o}rensen and Uchida 
\cite{Sorensen_Uchida2003}, (1) and (2) can be proved.
(3) Let $F(s)=\frac{s}{1-\ee^{-s}}$ and 
$F_n(\zeta)=\upsilon^\alpha F(2\lambda\Delta_n)$.
Since
\begin{equation}\label{ap-eq-0001}
\EE[\mathcal M_i(\lambda^*)^2|\GG]
=\frac{\epsilon^2\Delta_n}{F_n(\zeta^*)},
\quad
\EE[\mathcal M_i(\lambda^*)^4|\GG]
=3\biggl(\frac{\epsilon^2\Delta_n}{F_n(\zeta^*)}\biggr)^2,
\end{equation}
it follows from (1) and $F_n(\zeta)\to\upsilon^\alpha$  uniformly in $\zeta$ that
\begin{equation}\label{ap-eq-0002}
\epsilon^{-2}\sum_{i=1}^n \EE[f(x(t_{i-1}),\zeta)\mathcal M_i(\lambda^*)^2|\GG]
=\Delta_n\sum_{i=1}^n \frac{f(x(t_{i-1}),\zeta)}{F_n(\zeta^*)}
\pto \frac{1}{(\upsilon^*)^\alpha}\int_0^1 f(x_0(s),\zeta) \dd s,
\end{equation}
\begin{equation}\label{ap-eq-0003}
\epsilon^{-4}\sum_{i=1}^n \EE[f(x(t_{i-1}),\zeta)^2\mathcal M_i(\lambda^*)^4|\GG]
\lesssim
\Delta_n^2
\sum_{i=1}^n \frac{f(x(t_{i-1}),\zeta)^2}{F_n(\zeta^*)^2}\pto0
\end{equation}
as $n\to\infty$ and $\epsilon\to0$.
Therefore, it holds from Lemma 9 in Genon-Catalot and Jacod 
\cite{Genon-Catalot_Jacod1993} that
as $n\to\infty$ and $\epsilon\to0$,
\begin{equation*}
\epsilon^{-2}\sum_{i=1}^n f(x(t_{i-1}),\zeta)\mathcal M_i(\lambda^*)^2
\pto \frac{1}{(\upsilon^*)^\alpha}\int_0^1 f(x_0(s),\zeta) \dd s.
\end{equation*}
The tightness of the family of distribution of 
$\epsilon^{-2}\sum_{i=1}^n f(x(t_{i-1}),\cdot)\mathcal M_i(\lambda^*)^2$
follows from
\begin{align}
&\sup_{n,\epsilon}\EE\Biggl[
\sup_{\zeta}
\biggl|
\epsilon^{-2}\sum_{i=1}^n \partial_\zeta f(x(t_{i-1}),\zeta)\mathcal M_i(\lambda^*)^2
\biggr|\Biggr]
\nonumber
\\
&\le
\sup_{n,\epsilon}\EE\Biggl[
\epsilon^{-2}
\sum_{i=1}^n \sup_{\zeta}
\bigl|\partial_\zeta f(x(t_{i-1}),\zeta)\bigr|
\EE[\mathcal M_i(\lambda^*)^2|\GG]
\Biggr]
\nonumber
\\
&=
\sup_{n,\epsilon}\EE\Biggl[
\frac{\Delta_n}{F_n(\zeta^*)}
\sum_{i=1}^n \sup_{\zeta}
\bigl|\partial_\zeta f(x(t_{i-1}),\zeta)\bigr|
\Biggr]
<\infty.
\label{ap-eq-0004}
\end{align}
\end{proof}

\begin{rmk}\label{ap-rmk3}
Since S{\o}rensen and Uchida \cite{Sorensen_Uchida2003} considered parameter estimation 
by using a contrast function based on the Euler-Maruyama approximation, 
they introduced the condition (B1) in \cite{Sorensen_Uchida2003} 
to asymptotically ignore its approximation error.
Indeed, by viewing the proof of Lemma 3 (i) in \cite{Sorensen_Uchida2003}, 
we see that it is necessary to impose this condition in order to obtain 
the convergences corresponding to \eqref{ap-eq-0002} and \eqref{ap-eq-0003}.
However, in our model, the solution can be explicitly represented 
and the estimator is constructed as the minimum contrast estimator 
based on that solution, and therefore 
$\{\mathcal M_i(\lambda^*)\}_{i=1}^n$ is a martingale 
and no remainder term appears in 
\eqref{ap-eq-0002}-\eqref{ap-eq-0004},
unlike \cite{Sorensen_Uchida2003}, that is, such a condition is not required.
For this reason, the assumptions in Theorem \ref{ap-th1} (i) 
and Theorem \ref{ap-th2} (2) are weaker 
than those in Uchida \cite{Uchida2003} and 
Gloter and S{\o}rensen \cite{Gloter_Sorensen2009}, respectively. 
\end{rmk}

Before proving the theorems, 
we note that the following properties hold.
\begin{align}
\epsilon^{-2}\sum_{i=1}^n \mathcal M_i(\lambda^*)^2
&\pto (\upsilon^*)^{-\alpha},
\label{ap-eq-0101}
\\
\epsilon^{-1}
\sum_{i=1}^n \mathcal M_i(\lambda^*) x(t_{i-1})
&\dto N\biggl(0,(\upsilon^*)^{-\alpha}\frac{1-\ee^{-2\lambda^*}}{2\lambda^*}x(0)^2
\biggr),
\label{ap-eq-0102}
\\
\frac{1}{n}\sum_{i=1}^n x(t_{i-1})^2
&\pto \frac{1-\ee^{-2\lambda^*}}{2\lambda^*}x(0)^2.
\label{ap-eq-0103}
\end{align}
Indeed, \eqref{ap-eq-0101} and \eqref{ap-eq-0103} are obtained by
Lemma \ref{ap-lem1} and 
$\int_0^1 x_0(s)^2 \dd s=\frac{1-\ee^{-2\lambda^*}}{2\lambda^*}x(0)^2$.
\eqref{ap-eq-0102} follows from 
\begin{align}
\epsilon^{-1}
\sum_{i=1}^n 
x(t_{i-1})\EE[\mathcal M_i(\lambda^*)|\GG]
&=0,
\label{ap-eq-0104}
\\
\epsilon^{-2}
\sum_{i=1}^n x(t_{i-1})^2\EE[\mathcal M_i(\lambda^*)^2|\GG]
&=\frac{\Delta_n}{F_n(\lambda^*,\upsilon^*)}\sum_{i=1}^n x(t_{i-1})^2
\pto
(\upsilon^*)^{-\alpha}\frac{1-\ee^{-2\lambda^*}}{2\lambda^*}x(0)^2,
\label{ap-eq-0105}
\\
\epsilon^{-4}
\sum_{i=1}^n x(t_{i-1})^4 \EE[\mathcal M_i(\lambda^*)^4|\GG]
&\lesssim
\biggl(\frac{\Delta_n}{F_n(\lambda^*,\upsilon^*)}\biggr)^2
\sum_{i=1}^n x(t_{i-1})^4
\pto 0
\label{ap-eq-0106}
\end{align}
and Theorems 3.2 and 3.4 in Hall and Heyde \cite{Hall_Heyde1980}.

\begin{proof}[\bf{Proof of Theorem \ref{ap-th1}}]
Let $x(t)=x_1(t)$, $F_n(\lambda)=F_n(\lambda,\lambda)$,
\begin{align*}
V_{1,1}(\lambda,\lambda^*)
&=\lambda^\alpha(\lambda-\lambda^*)^2
\frac{1-\ee^{-2\lambda^*}}{2\lambda^*}x(0)^2,
\\
V_{1,2}(\lambda,\lambda^*)
&=
\biggl(\frac{\lambda}{\lambda^*}\biggr)^\alpha
-\log \lambda^\alpha
+c\lambda^\alpha(\lambda-\lambda^*)^2
\frac{1-\ee^{-2\lambda^*}}{2\lambda^*}x(0)^2,
\end{align*}
where $c=\lim_{n\to\infty,\epsilon\to0}(n\epsilon^2)^{-1}$. 
We have
\begin{align}
V_{n,\epsilon}^{(1)}(\lambda|x)
&=\frac{F_n(\lambda)}{\epsilon^2\Delta_n}
\sum_{i=1}^n \mathcal M_i(\lambda)^2
-n \log F_n(\lambda),
\label{ap-eq-0201}
\\
\partial_\lambda V_{n,\epsilon}^{(1)}(\lambda|x)
&=\frac{F_n'(\lambda)}{\epsilon^2\Delta_n}
\sum_{i=1}^n \mathcal M_i(\lambda)^2
+\frac{2\ee^{-\lambda\Delta_n}F_n(\lambda)}{\epsilon^2}
\sum_{i=1}^n \mathcal M_i(\lambda)x(t_{i-1})
-\frac{n F_n'(\lambda)}{F_n(\lambda)},
\label{ap-eq-0202}
\\
\partial_\lambda^2 V_{n,\epsilon}^{(1)}(\lambda|x)
&=\frac{F_n''(\lambda)}{\epsilon^2\Delta_n}
\sum_{i=1}^n \mathcal M_i(\lambda)^2
+\frac{4\ee^{-\lambda\Delta_n}F_n'(\lambda)}{\epsilon^2}
\sum_{i=1}^n \mathcal M_i(\lambda)x(t_{i-1})
\nonumber
\\
&\qquad+
\frac{2\ee^{-2\lambda\Delta_n}F_n(\lambda)\Delta_n}{\epsilon^2}
\sum_{i=1}^n x(t_{i-1})^2
-\frac{n \{F_n''(\lambda)F_n(\lambda)-F_n'(\lambda)^2\}}{F_n(\lambda)^2}.
\label{ap-eq-0203}
\end{align}
Since
\begin{align*}
F_n'(\lambda)
&=\frac{\alpha }{\lambda^{1-\alpha}}F(2\lambda\Delta_n)
+2\Delta_n\lambda^\alpha F'(2\lambda\Delta_n),
\\
F_n''(\lambda)
&=\frac{\alpha(\alpha-1)}{\lambda^{2-\alpha}}F(2\lambda\Delta_n)
+\frac{4\alpha\Delta_n}{\lambda^{1-\alpha}}F'(2\lambda\Delta_n)
+4\Delta_n^2\lambda^\alpha F''(2\lambda\Delta_n)
\end{align*}
and $F(s)\to1$ as $s\downarrow0$, it follows that
\begin{align}
F_n(\lambda)&\to\lambda^\alpha
\quad\text{ uniformly in } \lambda, 
\label{ap-eq-0204}
\\
F_n'(\lambda)
&\to\frac{\alpha}{\lambda^{1-\alpha}}
\quad\text{ uniformly in } \lambda, 
\label{ap-eq-0205}
\\
F_n''(\lambda)
&\to\frac{\alpha(\alpha-1)}{\lambda^{2-\alpha}}
\quad\text{ uniformly in } \lambda.
\label{ap-eq-0206}
\end{align}

(a) For proving the consistency of $\hat\lambda$, 
it is sufficient to show that under [B1],
\begin{equation}\label{ap-eq-0301}
\sup_{\lambda}
\bigl|
\epsilon^2\{V_{n,\epsilon}^{(1)}(\lambda|x)
-V_{n,\epsilon}^{(1)}(\lambda^*|x)\}
-V_{1,1}(\lambda,\lambda^*)
\bigr|=o_p(1),
\end{equation}
or that under [B2],
\begin{equation}\label{ap-eq-0302}
\sup_{\lambda}
\biggl|
\frac{1}{n}V_{n,\epsilon}^{(1)}(\lambda|x)
-V_{1,2}(\lambda,\lambda^*)
\biggr|=o_p(1).
\end{equation}

\textit{Proof of \eqref{ap-eq-0301}. }
By using \eqref{ap-eq-0201} and the fact that
\begin{align}
\mathcal M_i(\lambda)
&=\mathcal M_i(\lambda^*)
+(\ee^{-\lambda^*\Delta_n}-\ee^{-\lambda\Delta_n})x(t_{i-1}),
\label{ap-eq-0303}
\\
\mathcal M_i(\lambda)^2
&=\mathcal M_i(\lambda^*)^2
+2(\ee^{-\lambda^*\Delta_n}-\ee^{-\lambda\Delta_n})
\mathcal M_i(\lambda^*)x(t_{i-1})
+(\ee^{-\lambda^*\Delta_n}-\ee^{-\lambda\Delta_n})^2x(t_{i-1})^2,
\label{ap-eq-0304}
\end{align}
it follows that under [B1],
\begin{align*}
&\epsilon^2\{V_{n,\epsilon}^{(1)}(\lambda|x)-V_{n,\epsilon}^{(1)}(\lambda^*|x)\}
\\
&=
\Delta_n^{-1}
\sum_{i=1}^n 
\{F_n(\lambda)\mathcal M_i(\lambda)^2
-F_n(\lambda^*)\mathcal M_i(\lambda^*)^2\}
-n\epsilon^2\log\frac{F_n(\lambda)}{F_n(\lambda^*)}
\\
&=
\Delta_n^{-1}\{F_n(\lambda)-F_n(\lambda^*)\}
\sum_{i=1}^n 
\mathcal M_i(\lambda^*)^2
\\
&\qquad+
2F_n(\lambda)\Delta_n^{-1}
(\ee^{-\lambda^*\Delta_n}-\ee^{-\lambda\Delta_n})
\sum_{i=1}^n \mathcal M_i(\lambda^*)x(t_{i-1})
\\
&\qquad+
F_n(\lambda)\Delta_n^{-1}
(\ee^{-\lambda^*\Delta_n}-\ee^{-\lambda\Delta_n})^2
\sum_{i=1}^n x(t_{i-1})^2
-n\epsilon^2\log\frac{F_n(\lambda)}{F_n(\lambda^*)}
\\
&\pto
\lambda^\alpha(\lambda-\lambda^*)^2
\frac{1-\ee^{-2\lambda^*}}{2\lambda^*}x(0)^2
=V_{1,1}(\lambda,\lambda^*)
\end{align*}
uniformly in $\lambda$, where the last convergence holds from 
\eqref{ap-eq-0101}-\eqref{ap-eq-0103}, \eqref{ap-eq-0204},
$\lim_{n\to\infty,\epsilon\to0}n\epsilon^2=0$ and
\begin{equation}
\sup_{\lambda}
\Bigl|\Delta_n^{-1}(\ee^{-\lambda^*\Delta_n}-\ee^{-\lambda\Delta_n})
-(\lambda-\lambda^*)\Bigr|
\to0.
\label{ap-eq-0305}
\end{equation}

\textit{Proof of \eqref{ap-eq-0302}. }
It follows from \eqref{ap-eq-0201}, \eqref{ap-eq-0304},  
\eqref{ap-eq-0101}-\eqref{ap-eq-0103} and \eqref{ap-eq-0305}
that under [B2],
\begin{align*}
\frac{1}{n}V_{n,\epsilon}^{(1)}(\lambda|x)
&=
F_n(\lambda)
\Biggl\{
\frac{1}{\epsilon^2}
\sum_{i=1}^n 
\mathcal M_i(\lambda^*)^2
+\frac{2(\ee^{-\lambda^*\Delta_n}-\ee^{-\lambda\Delta_n})}{\epsilon^2}
\sum_{i=1}^n \mathcal M_i(\lambda^*)x(t_{i-1})
\\
&\qquad\qquad+
\frac{(\ee^{-\lambda^*\Delta_n}-\ee^{-\lambda\Delta_n})^2}{\epsilon^2}
\sum_{i=1}^n x(t_{i-1})^2
\Biggr\}
-\log F_n(\lambda)
\\
&\pto
\lambda^\alpha
\Biggl\{
(\lambda^*)^{-\alpha}
+c(\lambda-\lambda^*)^2
\frac{1-\ee^{-2\lambda^*}}{2\lambda^*}x(0)^2
\Biggr\}-\log \lambda^\alpha
=V_{1,2}(\lambda,\lambda^*)
\end{align*}
uniformly in $\lambda$, 
where $c=\lim_{n\to\infty,\epsilon\to0}(n\epsilon^2)^{-1}$.

(i) For proving the asymptotic normality of $\hat\lambda$, 
it is enough to show that under [B1],
\begin{equation}\label{ap-eq-0401}
\epsilon^2\partial_\lambda^2 V_{n,\epsilon}^{(1)}(\lambda^*|x)
\pto 2 G_1(\lambda^*),
\end{equation}
\begin{equation}\label{ap-eq-0402}
\epsilon^2\sup_{|\lambda-\lambda^*|<\delta_{n,\epsilon}}
\bigl|\partial_\lambda^2 V_{n,\epsilon}^{(1)}(\lambda|x)
-\partial_\lambda^2 V_{n,\epsilon}^{(1)}(\lambda^*|x)\bigr|
=o_p(1)
\ \text{ for }\delta_{n,\epsilon}\to0,
\end{equation}
\begin{equation}\label{ap-eq-0403}
-\epsilon \partial_\lambda V_{n,\epsilon}^{(1)}(\lambda^*|x)
\dto N(0,4G_1(\lambda^*)).
\end{equation}

\textit{Proof of \eqref{ap-eq-0401} and \eqref{ap-eq-0402}. }
By using \eqref{ap-eq-0203}, \eqref{ap-eq-0303}, \eqref{ap-eq-0304}, 
\eqref{ap-eq-0101}-\eqref{ap-eq-0103},
\eqref{ap-eq-0204}-\eqref{ap-eq-0206}
and the fact that $\lim_{n\to\infty,\epsilon\to0}n\epsilon^2=0$,
it holds that under [B1],
\begin{align}
\epsilon^2\partial_\lambda^2 V_{n,\epsilon}^{(1)}(\lambda|x)
&=\frac{F_n''(\lambda)}{\Delta_n}
\sum_{i=1}^n \mathcal M_i(\lambda)^2
+4\ee^{-\lambda\Delta_n}F_n'(\lambda)
\sum_{i=1}^n \mathcal M_i(\lambda)x(t_{i-1})
\nonumber
\\
&\qquad+
2\ee^{-2\lambda\Delta_n}F_n(\lambda)\Delta_n
\sum_{i=1}^n x(t_{i-1})^2
-\frac{n\epsilon^2 \{F_n''(\lambda)F_n(\lambda)-F_n'(\lambda)^2\}}{F_n(\lambda)^2}
\nonumber
\\
&=\frac{F_n''(\lambda)}{\Delta_n}
\sum_{i=1}^n \mathcal M_i(\lambda^*)^2
\nonumber
\\
&\qquad+
\biggl\{
\frac{2F_n''(\lambda)}{\Delta_n}(\ee^{-\lambda^*\Delta_n}-\ee^{-\lambda\Delta_n})
+4\ee^{-\lambda\Delta_n}F_n'(\lambda)
\biggr\}
\sum_{i=1}^n \mathcal M_i(\lambda^*)x(t_{i-1})
\nonumber
\\
&\qquad+
\biggl\{
\frac{F_n''(\lambda)}{\Delta_n}
(\ee^{-\lambda^*\Delta_n}-\ee^{-\lambda\Delta_n})^2
+4\ee^{-\lambda\Delta_n}F_n'(\lambda)(\ee^{-\lambda^*\Delta_n}-\ee^{-\lambda\Delta_n})
\nonumber
\\
&\qquad\qquad+
2\ee^{-2\lambda\Delta_n}F_n(\lambda)\Delta_n
\biggr\}
\sum_{i=1}^n x(t_{i-1})^2
-\frac{n\epsilon^2 \{F_n''(\lambda)F_n(\lambda)-F_n'(\lambda)^2\}}{F_n(\lambda)^2}
\nonumber
\\
&\pto
\biggl\{
\frac{\alpha(1-\alpha)}{\lambda^{2-\alpha}}(\lambda-\lambda^*)^2
+\frac{4\alpha}{\lambda^{1-\alpha}}(\lambda-\lambda^*)
+2\lambda^\alpha\biggr\}
\frac{1-\ee^{-2\lambda^*}}{2\lambda^*}x(0)^2
\label{ap-eq-0404}
\end{align}
uniformly in $\lambda$, and hence \eqref{ap-eq-0401} holds:
\begin{equation*}
\epsilon^2\partial_\lambda^2 V_{n,\epsilon}^{(1)}(\lambda^*|x)
\pto 2(\lambda^*)^\alpha \frac{1-\ee^{-2\lambda^*}}{2\lambda^*}x(0)^2
=2G_1(\lambda^*).
\end{equation*}
Moreover, the limit of \eqref{ap-eq-0404} is continuous with respect to $\lambda$, 
which completes the proof of \eqref{ap-eq-0402}.

\textit{Proof of \eqref{ap-eq-0403}. }
From \eqref{ap-eq-0202}, one has
\begin{align*}
-\epsilon\partial_\lambda V_{n,\epsilon}^{(1)}(\lambda^*|x)
&=-\frac{F_n'(\lambda^*)}{\epsilon\Delta_n}
\sum_{i=1}^n \mathcal M_i(\lambda^*)^2
-\frac{2\ee^{-\lambda^*\Delta_n}F_n(\lambda^*)}{\epsilon}
\sum_{i=1}^n \mathcal M_i(\lambda^*)x(t_{i-1})
+\frac{n\epsilon F_n'(\lambda^*)}{F_n(\lambda^*)}
\\
&=-\frac{F_n'(\lambda^*)}{\epsilon\Delta_n}
\sum_{i=1}^n 
\biggl(\mathcal M_i(\lambda^*)^2-\frac{\epsilon^2\Delta_n}{F_n(\lambda^*)}\biggr)
-\frac{2\ee^{-\lambda^*\Delta_n}F_n(\lambda^*)}{\epsilon}
\sum_{i=1}^n \mathcal M_i(\lambda^*)x(t_{i-1})
\\
&=:
\sum_{i=1}^n \mathcal N_{1,i}(\lambda^*)
+\sum_{i=1}^n \mathcal N_{2,i}(\lambda^*).
\end{align*}
Since it follows from \eqref{ap-eq-0001} that under [B1], 
\begin{align*}
\frac{1}{\epsilon\Delta_n}
\sum_{i=1}^n 
\biggl(
\EE[\mathcal M_i(\lambda^*)^2|\GG]-\frac{\epsilon^2\Delta_n}{F_n(\lambda^*)}
\biggr)
&=0,
\\
\frac{1}{(\epsilon \Delta_n)^2}
\sum_{i=1}^n 
\EE\Biggl[
\biggl(
\mathcal M_i(\lambda^*)^2-\frac{\epsilon^2\Delta_n}{F_n(\lambda^*)}
\biggr)^2
\Bigg|\GG\Biggr]
&\lesssim
\frac{1}{(\epsilon \Delta_n)^2}
\sum_{i=1}^n 
\biggl(\frac{\epsilon^2\Delta_n}{F_n(\lambda^*)}\biggr)^2
=O_p(n\epsilon^2)=o_p(1),
\end{align*}
it holds from Lemma 9 in Genon-Catalot and Jacod \cite{Genon-Catalot_Jacod1993} that
$\sum_{i=1}^n \mathcal N_{1,i}(\lambda^*)\pto0$.
It also holds from \eqref{ap-eq-0102} and \eqref{ap-eq-0204} that
$\sum_{i=1}^n \mathcal N_{2,i}(\lambda^*)
\dto
N(0,4G_1(\lambda^*))$.
Hence, we obtain that
\begin{equation*}
-\epsilon\partial_\lambda V_{n,\epsilon}^{(1)}(\lambda^*|x)
=\sum_{i=1}^n \mathcal N_{1,i}(\lambda^*)
+\sum_{i=1}^n \mathcal N_{2,i}(\lambda^*)
\dto N(0,4G_1(\lambda^*)).
\end{equation*}

(ii) We will prove that the followings hold under [B2].

\begin{equation}\label{ap-eq-0501}
\frac{1}{n}\partial_\lambda^2 V_{n,\epsilon}^{(1)}(\lambda^*|x)
\pto 2 I_1(\lambda^*),
\end{equation}
\begin{equation}\label{ap-eq-0502}
\frac{1}{n}\sup_{|\lambda-\lambda^*|<\delta_{n,\epsilon}}
\bigl|\partial_\lambda^2 V_{n,\epsilon}^{(1)}(\lambda|x)
-\partial_\lambda^2 V_{n,\epsilon}^{(1)}(\lambda^*|x)\bigr|
=o_p(1)
\ \text{ for }\delta_{n,\epsilon}\to0,
\end{equation}
\begin{equation}\label{ap-eq-0503}
-\frac{1}{\sqrt n}\partial_\lambda V_{n,\epsilon}^{(1)}(\lambda^*|x)
\dto N(0,4I_1(\lambda^*)).
\end{equation}

\textit{Proof of \eqref{ap-eq-0501} and \eqref{ap-eq-0502}.}
It follows from
\eqref{ap-eq-0101}-\eqref{ap-eq-0103}, 
\eqref{ap-eq-0204}-\eqref{ap-eq-0206}, 
\eqref{ap-eq-0303}-\eqref{ap-eq-0305}, 
and $c=\lim_{n\to\infty,\epsilon\to0}(n\epsilon^2)^{-1}$ that
under [B2],
\begin{align}
\frac{1}{n}\partial_\lambda^2 V_{n,\epsilon}^{(1)}(\lambda|x)
&=\frac{F_n''(\lambda)}{\epsilon^2}
\sum_{i=1}^n \mathcal M_i(\lambda^*)^2
\nonumber
\\
&\qquad+
\biggl\{
\frac{2F_n''(\lambda)}{\epsilon^2}(\ee^{-\lambda^*\Delta_n}-\ee^{-\lambda\Delta_n})
+\frac{4\ee^{-\lambda\Delta_n}F_n'(\lambda)}{n\epsilon^2}
\biggr\}
\sum_{i=1}^n \mathcal M_i(\lambda^*)x(t_{i-1})
\nonumber
\\
&\qquad+
\biggl\{
\frac{F_n''(\lambda)}{\epsilon^2}
(\ee^{-\lambda^*\Delta_n}-\ee^{-\lambda\Delta_n})^2
+\frac{4\ee^{-\lambda\Delta_n}F_n'(\lambda)}{n\epsilon^2}
(\ee^{-\lambda^*\Delta_n}-\ee^{-\lambda\Delta_n})
\nonumber
\\
&\qquad\qquad+
\frac{2\ee^{-2\lambda\Delta_n}F_n(\lambda)\Delta_n}{n\epsilon^2}
\biggr\}
\sum_{i=1}^n x(t_{i-1})^2
-\frac{F_n''(\lambda)F_n(\lambda)-F_n'(\lambda)^2}{F_n(\lambda)^2}
\nonumber
\\
&\pto
\frac{\alpha(\alpha-1)}{\lambda^{2-\alpha}(\lambda^*)^\alpha}
+c
\biggl\{
\frac{\alpha(\alpha-1)}{\lambda^{2-\alpha}}(\lambda-\lambda^*)^2
\nonumber
\\
&\qquad\qquad
+\frac{4\alpha}{\lambda^{1-\alpha}}(\lambda-\lambda^*)
+2\lambda^\alpha
\biggr\}
\frac{1-\ee^{-2\lambda^*}}{2\lambda^*}x(0)^2
+\frac{\alpha}{\lambda^2}
\label{ap-eq-0504}
\end{align}
uniformly in $\lambda$, and thus we have \eqref{ap-eq-0501}:
\begin{equation*}
\frac{1}{n}\partial_\lambda^2 V_{n,\epsilon}^{(1)}(\lambda^*|x)
\pto
\biggl(\frac{\alpha}{\lambda^*}\biggr)^2
+2c(\lambda^*)^\alpha
\frac{1-\ee^{-2\lambda^*}}{2\lambda^*}x(0)^2
=2I_1(\lambda^*).
\end{equation*}
Moreover, the limit of \eqref{ap-eq-0504} is continuous with respect to $\lambda$, 
which completes the proof of \eqref{ap-eq-0502}.

\textit{Proof of \eqref{ap-eq-0503}. }
We obtain that
\begin{align*}
-\frac{1}{\sqrt n}\partial_\lambda V_{n,\epsilon}^{(1)}(\lambda^*|x)
&=
\frac{F_n'(\lambda^*)}{\epsilon^2\Delta_n\sqrt n}
\sum_{i=1}^n \mathcal M_i(\lambda^*)^2
+\frac{2\ee^{-\lambda\Delta_n}F_n(\lambda^*)}{\epsilon^2\sqrt n}
\sum_{i=1}^n \mathcal M_i(\lambda^*)x(t_{i-1})
-\frac{\sqrt n F_n'(\lambda^*)}{F_n(\lambda^*)}
\\
&=
\frac{F_n'(\lambda^*)}{\epsilon^2\Delta_n\sqrt n}
\sum_{i=1}^n
\biggl(\mathcal M_i(\lambda^*)^2-\frac{\epsilon^2\Delta_n}{F_n(\lambda^*)}\biggr)
+\frac{2\ee^{-\lambda\Delta_n}F_n(\lambda^*)}{\epsilon^2\sqrt n}
\sum_{i=1}^n \mathcal M_i(\lambda^*)x(t_{i-1})
\\
&=:
\sum_{i=1}^n \mathcal N_{3,i}(\lambda^*)
+\sum_{i=1}^n \mathcal N_{4,i}(\lambda^*).
\end{align*}
Since
\begin{align*}
\frac{1}{\epsilon^2 \Delta_n \sqrt n}
\sum_{i=1}^n 
\EE\biggl[\mathcal M_i(\lambda^*)^2-\frac{\epsilon^2\Delta_n}{F_n(\lambda^*)}
\bigg|\GG\biggr]
=0,
\end{align*}
\begin{align*}
&\frac{1}{(\epsilon^2 \Delta_n \sqrt n)^2}
\sum_{i=1}^n 
\EE\Biggl[
\biggl(
\mathcal M_i(\lambda^*)^2-\frac{\epsilon^2\Delta_n}{F_n(\lambda^*)}
\biggr)^2
\Bigg|\GG\Biggr]
\\
&=
\frac{1}{(\epsilon^2 \Delta_n \sqrt n)^2}
\sum_{i=1}^n 
\Biggl\{
\EE[\mathcal M_i(\lambda^*)^4|\GG]
-\frac{2\epsilon^2\Delta_n}{F_n(\lambda^*)}\EE[\mathcal M_i(\lambda^*)^2|\GG]
+\biggl(\frac{\epsilon^2\Delta_n}{F_n(\lambda^*)}\biggr)^2
\Biggr\}
\\
&=
\frac{1}{(\epsilon^2 \Delta_n \sqrt n)^2}
\sum_{i=1}^n 
2\biggl(\frac{\epsilon^2\Delta_n}{F_n(\lambda^*)}\biggr)^2
\pto \frac{2}{(\lambda^*)^{2\alpha}},
\end{align*}
\begin{align*}
&\frac{1}{(\epsilon^2 \Delta_n \sqrt n)^4}
\sum_{i=1}^n 
\EE\Biggl[
\biggl(
\mathcal M_i(\lambda^*)^2-\frac{\epsilon^2\Delta_n}{F_n(\lambda^*)}
\biggr)^4
\Bigg|\GG\Biggr]
\\
&\lesssim
\frac{1}{(\epsilon^2 \Delta_n \sqrt n)^4}
\sum_{i=1}^n 
\Biggl\{
\EE[\mathcal M_i(\lambda^*)^8|\GG]
+\biggl(\frac{\epsilon^2\Delta_n}{F_n(\lambda^*)}\biggr)^4
\Biggr\}
\\
&\lesssim
\frac{1}{(\epsilon^2 \Delta_n \sqrt n)^4}
\sum_{i=1}^n 
\biggl(\frac{\epsilon^2\Delta_n}{F_n(\lambda^*)}\biggr)^4
\pto 0,
\end{align*}
it follows from \eqref{ap-eq-0205} that
\begin{align*}
\sum_{i=1}^n \EE[\mathcal N_{3,i}(\lambda^*)|\GG]&=0,
\\
\sum_{i=1}^n \EE[\mathcal N_{3,i}(\lambda^*)^2|\GG]
&\pto\biggl(\frac{\alpha}{(\lambda^*)^{1-\alpha}}\biggr)^2 
\frac{2}{(\lambda^*)^{2\alpha}}
=4H_1(\lambda^*),
\\
\sum_{i=1}^n \EE[\mathcal N_{3,i}(\lambda^*)^4|\GG]&\pto0.
\end{align*}
Furthermore, it holds from \eqref{ap-eq-0104}-\eqref{ap-eq-0106},
\eqref{ap-eq-0204} and $c=\lim_{n\to\infty,\epsilon\to0}(n\epsilon^2)^{-1}$ 
that under [B2], 
\begin{align*}
\sum_{i=1}^n \EE[\mathcal N_{4,i}(\lambda^*)|\GG]&=0,
\\
\sum_{i=1}^n \EE[\mathcal N_{4,i}(\lambda^*)^2|\GG]
&\pto4c(\lambda^*)^{2\alpha} \frac{G_1(\lambda^*)}{(\lambda^*)^{2\alpha}}
=4c G_1(\lambda^*),
\\
\sum_{i=1}^n \EE[\mathcal N_{4,i}(\lambda^*)^4|\GG]&\pto0.
\end{align*}
Therefore, noting that
$\sum_{i=1}^n \EE[\mathcal N_{3,i}(\lambda^*)\mathcal N_{4,i}(\lambda^*)|\GG]=0$
and $I_1(\lambda)=H_1(\lambda)+c G_1(\lambda)$, one has 
from Theorems 3.2 and 3.4 in Hall and Heyde \cite{Hall_Heyde1980} that
\begin{align*}
-\frac{1}{\sqrt n}\partial_\lambda V_{n,\epsilon}^{(1)}(\lambda^*|x)
=\sum_{i=1}^n \mathcal N_{3,i}(\lambda^*)
+\sum_{i=1}^n \mathcal N_{4,i}(\lambda^*)
\dto 
N(0,4I_1(\lambda^*)).
\end{align*}
\end{proof}

\begin{proof}[\bf{Proof of Theorem \ref{ap-th2}}]
Let $x(t)=x_2(t)$.
$V_{n,\epsilon}^{(2)}(\lambda,\mu|x)$ and its derivatives 
with respect to $(\lambda,\mu)$ up to second order  
can be expressed by using $F_n(\lambda,\mu)$ and $\mathcal M_i(\lambda)$ as follows.
\begin{align}
V_{n,\epsilon}^{(2)}(\lambda,\mu|x)
&=\frac{F_n(\lambda,\mu)}{\epsilon^2\Delta_n}
\sum_{i=1}^n \mathcal M_i(\lambda)^2
-n \log F_n(\lambda,\mu),
\label{ap-eq-0601}
\\
\partial_\lambda V_{n,\epsilon}^{(2)}(\lambda,\mu|x)
&=\frac{\partial_\lambda F_n(\lambda,\mu)}{\epsilon^2\Delta_n}
\sum_{i=1}^n \mathcal M_i(\lambda)^2
+\frac{2\ee^{-\lambda\Delta_n}F_n(\lambda,\mu)}{\epsilon^2}
\sum_{i=1}^n \mathcal M_i(\lambda)x(t_{i-1})
\nonumber
\\
&\qquad
-\frac{n \partial_\lambda F_n(\lambda,\mu)}{F_n(\lambda,\mu)},
\label{ap-eq-0602}
\\
\partial_\mu V_{n,\epsilon}^{(2)}(\lambda,\mu|x)
&=\frac{\partial_\mu F_n(\lambda,\mu)}{\epsilon^2\Delta_n}
\sum_{i=1}^n \mathcal M_i(\lambda)^2
-\frac{n \partial_\mu F_n(\lambda,\mu)}{F_n(\lambda,\mu)},
\label{ap-eq-0603}
\\
\partial_\lambda^2 V_{n,\epsilon}^{(2)}(\lambda,\mu|x)
&=\frac{\partial_\lambda^2 F_n(\lambda,\mu)}{\epsilon^2\Delta_n}
\sum_{i=1}^n \mathcal M_i(\lambda)^2
+\frac{4\ee^{-\lambda\Delta_n}\partial_\lambda F_n(\lambda,\mu)}{\epsilon^2}
\sum_{i=1}^n \mathcal M_i(\lambda)x(t_{i-1})
\nonumber
\\
&\qquad+
\frac{2\ee^{-2\lambda\Delta_n}F_n(\lambda,\mu)\Delta_n}{\epsilon^2}
\sum_{i=1}^n x(t_{i-1})^2
-n \partial_\lambda^2 \log F_n(\lambda,\mu),
\label{ap-eq-0604}
\\
\partial_\mu^2 V_{n,\epsilon}^{(2)}(\lambda,\mu|x)
&=\frac{\partial_\mu^2 F_n(\lambda,\mu)}{\epsilon^2\Delta_n}
\sum_{i=1}^n \mathcal M_i(\lambda)^2
-n\partial_\mu^2 \log F_n(\lambda,\mu),
\label{ap-eq-0605}
\\
\partial_\mu\partial_\lambda V_{n,\epsilon}^{(2)}(\lambda,\mu|x)
&=\frac{\partial_\mu\partial_\lambda F_n(\lambda,\mu)}{\epsilon^2\Delta_n}
\sum_{i=1}^n \mathcal M_i(\lambda)^2
+\frac{2\ee^{-\lambda\Delta_n}\partial_\mu F_n(\lambda,\mu)}{\epsilon^2}
\sum_{i=1}^n \mathcal M_i(\lambda)x(t_{i-1})
\nonumber
\\
&\qquad
-n\partial_\mu\partial_\lambda \log F_n(\lambda,\mu),
\label{ap-eq-0606}
\end{align}
where the derivatives of $F_n(\lambda,\mu)$ are given by 
\begin{align*}
\partial_\lambda F_n(\lambda,\mu)
&=2\Delta_n\mu^\alpha F'(2\lambda\Delta_n),
\quad
\partial_\mu F_n(\lambda,\mu)
=\alpha\mu^{\alpha-1} F(2\lambda\Delta_n),
\\
\partial_\lambda^2 F_n(\lambda,\mu)
&=4\Delta_n^2 \mu^\alpha F''(2\lambda\Delta_n),
\quad
\partial_\mu \partial_\lambda F_n(\lambda,\mu)
=2\alpha\Delta_n\mu^{\alpha-1} F'(2\lambda\Delta_n),
\\
\partial_\mu^2 F_n(\lambda,\mu)
&=\alpha(\alpha-1)\mu^{\alpha-2}F(2\lambda\Delta_n)
\end{align*}
and it follows from $F(s)\to1$, $F'(s)\to1/2$ and $F''(s)\to1/6$ as $s\downarrow0$ that
\begin{align}
F_n(\lambda,\mu) &\to \mu^\alpha
\quad\text{uniformly in }(\lambda,\mu),
\label{ap-eq-0607}
\\
\Delta_n^{-1}\partial_\lambda F_n(\lambda,\mu) &\to \mu^\alpha
\quad\text{uniformly in }(\lambda,\mu),
\label{ap-eq-0608}
\\
\partial_\mu F_n(\lambda,\mu) &\to \alpha\mu^{\alpha-1},
\quad\text{uniformly in }(\lambda,\mu),
\label{ap-eq-0609}
\\
\Delta_n^{-2}\partial_\lambda^2 F_n(\lambda,\mu) &\to \frac{2\mu^\alpha}{3}
\quad\text{uniformly in }(\lambda,\mu),
\label{ap-eq-0610}
\\
\Delta_n^{-1}\partial_\mu\partial_\lambda F_n(\lambda,\mu)
&\to \alpha\mu^{\alpha-1}
\quad\text{uniformly in }(\lambda,\mu),
\label{ap-eq-0611}
\\
\partial_\mu^2 F_n(\lambda,\mu) &\to \alpha(\alpha-1)\mu^{\alpha-2}
\quad\text{uniformly in }(\lambda,\mu).
\label{ap-eq-0612}
\end{align}

We first prove $(2)$.

(2) Let
\begin{align*}
V_{2,1}(\lambda,\mu,\lambda^*)
&=\mu^\alpha(\lambda-\lambda^*)^2
\frac{1-\ee^{-2\lambda^*}}{2\lambda^*}x(0)^2,
\\
V_{2,2}(\mu,\mu^*)
&=\biggl(\frac{\mu}{\mu^*}\biggr)^\alpha-1
-\log\biggl(\frac{\mu}{\mu^*}\biggr)^\alpha,
\\
C_{n,\epsilon}^{(2)}(\lambda,\mu|x)
&=
\begin{pmatrix}
\epsilon^2 \partial_\lambda^2 V_{n,\epsilon}^{(2)}(\lambda,\mu|x) 
& \frac{\epsilon}{\sqrt n}
\partial_\lambda \partial_\mu V_{n,\epsilon}^{(2)}(\lambda,\mu|x)
\\
\frac{\epsilon}{\sqrt n}
\partial_\mu \partial_\lambda V_{n,\epsilon}^{(2)}(\lambda,\mu|x)
& \frac{1}{n} \partial_\mu^2 V_{n,\epsilon}^{(2)}(\lambda,\mu|x)
\end{pmatrix},
\\
K_{n,\epsilon}^{(2)}(\lambda,\mu|x)
&=
\begin{pmatrix}
-\epsilon\partial_\lambda V_{n,\epsilon}^{(2)}(\lambda,\mu|x)
\\
-\frac{1}{\sqrt n}\partial_\mu V_{n,\epsilon}^{(2)}(\lambda,\mu|x)
\end{pmatrix}.
\end{align*}
We set $\lambda_u=\lambda^*+u(\tilde\lambda-\lambda^*)$ 
and $\mu_u=\mu^*+u(\tilde\mu-\mu^*)$ for $u\in[0,1]$.
For the proof of the asymptotic normality of $(\tilde\lambda,\tilde\mu)$, 
we will show that
\begin{equation}\label{ap-eq-0701}
\sup_{\lambda,\mu}
\bigl|
\epsilon^2\{V_{n,\epsilon}^{(2)}(\lambda,\mu|x)
-V_{n,\epsilon}^{(2)}(\lambda^*,\mu|x)\}
-V_{2,1}(\lambda,\mu,\lambda^*)
\bigr|=o_p(1),
\end{equation}
\begin{equation}\label{ap-eq-0702}
\sup_{\mu}
\biggl|
\frac{1}{n}\{V_{n,\epsilon}^{(2)}(\tilde\lambda,\mu|x)
-V_{n,\epsilon}^{(2)}(\tilde\lambda,\mu^*|x)\}
-V_{2,2}(\mu,\mu^*)
\biggr|=o_p(1),
\end{equation}
\begin{equation}\label{ap-eq-0703}
\sup_{u\in[0,1]}
\bigl|C_{n,\epsilon}^{(2)}(\lambda_u,\mu_u|x)-2I_2(\lambda^*,\mu^*)\bigr|
=o_p(1),
\end{equation}
\begin{equation}\label{ap-eq-0704}
K_{n,\epsilon}^{(2)}(\lambda^*,\mu^*|x) \dto N(0,4I_2(\lambda^*,\mu^*)).
\end{equation}

\textit{Proof of \eqref{ap-eq-0701}. }
By using 
\eqref{ap-eq-0601} and the fact that
\begin{equation*}
\sup_{\lambda}
\biggl|
n\log\frac{F(2\lambda\Delta_n)}{F(2\lambda^*\Delta_n)}
\biggr|\lesssim 1,
\end{equation*}
it follows from \eqref{ap-eq-0607} and \eqref{ap-eq-0305} that
\begin{align*}
&\epsilon^2\{V_{n,\epsilon}^{(2)}(\lambda,\mu|x)-V_{n,\epsilon}^{(2)}(\lambda^*,\mu|x)\}
\\
&=
\Delta_n^{-1}
\sum_{i=1}^n 
\{F_n(\lambda,\mu)\mathcal M_i(\lambda)^2
-F_n(\lambda^*,\mu)\mathcal M_i(\lambda^*)^2\}
-n\epsilon^2\log\frac{F_n(\lambda,\mu)}{F_n(\lambda^*,\mu)}
\\
&=
\Delta_n^{-1}\{F_n(\lambda,\mu)-F_n(\lambda^*,\mu)\}
\sum_{i=1}^n 
\mathcal M_i(\lambda^*)^2
\\
&\qquad+
2F_n(\lambda,\mu)\Delta_n^{-1}
(\ee^{-\lambda^*\Delta_n}-\ee^{-\lambda\Delta_n})
\sum_{i=1}^n \mathcal M_i(\lambda^*)x(t_{i-1})
\\
&\qquad+
F_n(\lambda,\mu)\Delta_n^{-1}
(\ee^{-\lambda^*\Delta_n}-\ee^{-\lambda\Delta_n})^2
\sum_{i=1}^n x(t_{i-1})^2
-n\epsilon^2\log\frac{F(2\lambda\Delta_n)}{F(2\lambda^*\Delta_n)}
\\
&\pto
\mu^\alpha(\lambda-\lambda^*)^2
\frac{1-\ee^{-2\lambda^*}}{2\lambda^*}x(0)^2
=V_{2,1}(\lambda,\mu,\lambda^*)
\end{align*}
uniformly in $(\lambda,\mu)$. 

\textit{Proof of \eqref{ap-eq-0702}. }
We first show that $\epsilon^{-1}(\tilde\lambda-\lambda^*)=O_p(1)$.
By using the Taylor expansion,
\begin{equation*}
-\partial_\lambda V_{n,\epsilon}^{(2)}(\lambda^*,\tilde\mu|x)
=\partial_\lambda V_{n,\epsilon}^{(2)}(\tilde\lambda,\tilde\mu|x)
-\partial_\lambda V_{n,\epsilon}^{(2)}(\lambda^*,\tilde\mu|x)
=\int_0^1 \partial_\lambda^2 V_{n,\epsilon}^{(2)}(\lambda_u,\tilde\mu|x)\dd u
(\tilde\lambda-\lambda^*),
\end{equation*}
i.e.,
\begin{equation*}
-\epsilon\partial_\lambda V_{n,\epsilon}^{(2)}(\lambda^*,\tilde\mu|x)
=\epsilon^2\int_0^1 \partial_\lambda^2 V_{n,\epsilon}^{(2)}(\lambda_u,\tilde\mu|x)\dd u
\,\epsilon^{-1}(\tilde\lambda-\lambda^*).
\end{equation*}
Since 
$\inf_\mu|G_2(\lambda^*,\mu)|>0$ and $\tilde\lambda$ has consistency, 
the tightness
$\epsilon^{-1}(\tilde\lambda-\lambda^*)=O_p(1)$ follows from the following properties.
\begin{equation}\label{ap-eq-0705}
\epsilon\sup_{\mu}|\partial_\lambda V_{n,\epsilon}^{(2)}(\lambda^*,\mu|x)|=O_p(1),
\end{equation}
\begin{equation}\label{ap-eq-0706}
\sup_{\lambda,\mu}
\Bigl|\epsilon^2\partial_\lambda^2 V_{n,\epsilon}^{(2)}(\lambda,\mu|x)
-2G_2(\lambda^*,\mu)\Bigr|
\pto0.
\end{equation}

\textit{Proof of \eqref{ap-eq-0705}. }
Setting 
$-\epsilon\partial_\lambda V_{n,\epsilon}^{(2)}(\lambda,\mu|x)
=\sum_{i=1}^n \mathcal K_{i}(\lambda,\mu)$,
we have
\begin{align*}
\sum_{i=1}^n \mathcal K_{i}(\lambda^*,\mu)
&=-\frac{\partial_\lambda F_n(\lambda^*,\mu)}{\epsilon\Delta_n}
\sum_{i=1}^n
\biggl(\mathcal M_i(\lambda^*)^2-\frac{\epsilon^2\Delta_n}
{F_n(\lambda^*,\mu^*)}\biggr)
\\
&\qquad
-\frac{2\ee^{-\lambda^*\Delta_n}F_n(\lambda^*,\mu)}{\epsilon}
\sum_{i=1}^n \mathcal M_i(\lambda^*)x(t_{i-1})
\\
&\qquad
-\frac{\epsilon\partial_\lambda F_n(\lambda^*,\mu)}{\Delta_n}
\biggl(\frac{1}{F_n(\lambda^*,\mu^*)}-\frac{1}{F_n(\lambda^*,\mu)}\biggr)
\\
&=:
\sum_{i=1}^n \mathcal K_{1,i}(\lambda^*,\mu)
+\sum_{i=1}^n \mathcal K_{2,i}(\lambda^*,\mu)
+\mathcal K_{3,n}(\lambda^*,\mu).
\end{align*}
It follows from
\begin{align*}
\epsilon^{-1}\sum_{i=1}^n 
\EE\biggl[\mathcal M_i(\lambda^*)^2-\frac{\epsilon^2\Delta_n}{F_n(\lambda^*,\mu^*)}
\bigg|\GG\biggr]
&=0,
\\
\epsilon^{-2}\sum_{i=1}^n 
\EE\Biggl[
\biggl(
\mathcal M_i(\lambda^*)^2-\frac{\epsilon^2\Delta_n}{F_n(\lambda^*,\mu^*)}
\biggr)^2
\Bigg|\GG\Biggr]
&\lesssim
\epsilon^{-2}\sum_{i=1}^n 
\biggl(\frac{\epsilon^2\Delta_n}{F_n(\lambda^*,\mu^*)}\biggr)^2
\pto0,
\end{align*}
\eqref{ap-eq-0608} and 
Lemma 9 in Genon-Catalot and Jacod \cite{Genon-Catalot_Jacod1993}
that $\sup_{\mu}|\sum_{i=1}^n \mathcal K_{1,i}(\lambda^*,\mu)|=o_p(1)$.
It also holds from \eqref{ap-eq-0102} and \eqref{ap-eq-0607} that
$\sup_\mu|\sum_{i=1}^n \mathcal K_{2,i}(\lambda^*,\mu)|=O_p(1)$.
Furthermore, it follows from \eqref{ap-eq-0607} and \eqref{ap-eq-0608} 
that $\sup_{\mu}|\mathcal K_{3,n}(\lambda^*,\mu)|=o(1)$. 
Therefore, the proof of \eqref{ap-eq-0705} is completed.

\textit{Proof of \eqref{ap-eq-0706}.}
It holds from \eqref{ap-eq-0604}, \eqref{ap-eq-0303} and \eqref{ap-eq-0304} that
\begin{align}
\epsilon^2\partial_\lambda^2 V_{n,\epsilon}^{(2)}(\lambda,\mu|x)
&=\frac{\partial_\lambda^2 F_n(\lambda,\mu)}{\Delta_n}
\sum_{i=1}^n \mathcal M_i(\lambda)^2
+4\ee^{-\lambda\Delta_n}\partial_\lambda F_n(\lambda,\mu)
\sum_{i=1}^n \mathcal M_i(\lambda)x(t_{i-1})
\nonumber
\\
&\qquad+
2\ee^{-2\lambda\Delta_n}F_n(\lambda,\mu)\Delta_n
\sum_{i=1}^n x(t_{i-1})^2
-n\epsilon^2 \partial_\lambda^2 \log F_n(\lambda,\mu)
\nonumber
\\
&=\frac{\partial_\lambda^2 F_n(\lambda,\mu)}{\Delta_n}
\sum_{i=1}^n 
\mathcal M_i(\lambda^*)^2
\nonumber
\\
&\qquad+
\biggl\{
2(\ee^{-\lambda^*\Delta_n}-\ee^{-\lambda\Delta_n})
\frac{\partial_\lambda^2 F_n(\lambda,\mu)}{\Delta_n}
\nonumber
\\
&\qquad\qquad+
4\ee^{-\lambda\Delta_n}\partial_\lambda F_n(\lambda,\mu)
\biggr\}
\sum_{i=1}^n \mathcal M_i(\lambda^*)x(t_{i-1})
\nonumber
\\
&\qquad+
\biggl\{
(\ee^{-\lambda^*\Delta_n}-\ee^{-\lambda\Delta_n})^2
\frac{\partial_\lambda^2 F_n(\lambda,\mu)}{\Delta_n}
\nonumber
\\
&\qquad\qquad+
4\ee^{-\lambda\Delta_n}\partial_\lambda F_n(\lambda,\mu)
(\ee^{-\lambda^*\Delta_n}-\ee^{-\lambda\Delta_n})
\nonumber
\\
&\qquad\qquad
+2\ee^{-2\lambda\Delta_n}F_n(\lambda,\mu)\Delta_n
\biggr\}
\sum_{i=1}^n x(t_{i-1})^2
-n\epsilon^2 \partial_\lambda^2 \log F_n(\lambda,\mu)
\nonumber
\\
&\pto 2\mu^\alpha
\frac{1-\ee^{-2\lambda^*}}{2\lambda^*}x(0)^2
=2G_2(\lambda^*,\mu)
\label{ap-eq-0707}
\end{align}
uniformly in $(\lambda,\mu)$, 
where the last convergence follows from 
\eqref{ap-eq-0101}-\eqref{ap-eq-0103}, \eqref{ap-eq-0305}, 
\eqref{ap-eq-0607}, \eqref{ap-eq-0608} and \eqref{ap-eq-0610}.
This concludes the proof of \eqref{ap-eq-0706}.

Let us begin the proof of \eqref{ap-eq-0702}.
It follows from \eqref{ap-eq-0101}-\eqref{ap-eq-0103},
\begin{equation}
|\ee^{-\lambda^*\Delta_n}-\ee^{-\lambda\Delta_n}|
\lesssim
|\lambda-\lambda^*|\Delta_n,
\label{ap-eq-0708}
\end{equation}
and the tightness of $\epsilon^{-1}(\tilde\lambda-\lambda^*)$ that
\begin{align*}
&\frac{1}{n}\{
V_{n,\epsilon}^{(2)}(\tilde\lambda,\mu|x)
-V_{n,\epsilon}^{(2)}(\tilde\lambda,\mu^*|x)\}
\\
&=
\epsilon^{-2}\{\mu^\alpha-(\mu^*)^\alpha\}
\sum_{i=1}^n 
\mathcal M_i(\tilde\lambda)^2-\log\biggl(\frac{\mu}{\mu^*}\biggr)^\alpha
\\
&=
\{\mu^\alpha-(\mu^*)^\alpha\}
\Biggl\{
\epsilon^{-2}\sum_{i=1}^n \mathcal M_i(\lambda^*)^2
\\
&\qquad+
2(\ee^{-\lambda^*\Delta_n}-\ee^{-\tilde\lambda\Delta_n})
\epsilon^{-2}\sum_{i=1}^n x(t_{i-1})\mathcal M_i(\lambda^*)
\\
&\qquad+
(\ee^{-\lambda^*\Delta_n}-\ee^{-\tilde\lambda\Delta_n})^2
\epsilon^{-2}\sum_{i=1}^n x(t_{i-1})^2
\Biggr\}
-\log\biggl(\frac{\mu}{\mu^*}\biggr)^\alpha
\\
&\pto
\biggl(\frac{\mu}{\mu^*}\biggr)^\alpha-1
-\log\biggl(\frac{\mu}{\mu^*}\biggr)^\alpha
=V_{2,2}(\mu,\mu^*)
\end{align*}
uniformly in $\mu$.

\textit{Proof of \eqref{ap-eq-0703}.}
It follows from  \eqref{ap-eq-0707}, the consistency of $\tilde\mu$ 
and the continuity of $G_2(\lambda,\mu)$ with respect to $\mu$ that
\begin{equation*}
\epsilon^2\partial_\lambda^2 V_{n,\epsilon}^{(2)}(\lambda_u,\mu_u|x)
\pto 2G_2(\lambda^*,\mu^*)
\end{equation*}
uniformly in $u\in[0,1]$.
On the other hand, it holds 
from \eqref{ap-eq-0605}, \eqref{ap-eq-0101}-\eqref{ap-eq-0103}, 
\eqref{ap-eq-0612}, \eqref{ap-eq-0708} and 
\begin{equation*}
\partial_\mu^2 \log F_n(\lambda,\mu)\to -\frac{\alpha}{\mu^2}
\quad\text{uniformly in }(\lambda,\mu)
\end{equation*}
that
\begin{align*}
\frac{1}{n}\partial_\mu^2 V_{n,\epsilon}^{(2)}(\lambda_u,\mu_u|x)
&=\frac{\partial_\mu^2 F_n(\lambda_u,\mu_u)}{\epsilon^2}
\sum_{i=1}^n \mathcal M_i(\lambda_u)^2
-\partial_\mu^2 \log F_n(\lambda_u,\mu_u)
\\
&=
\partial_\mu^2 F_n(\lambda_u,\mu_u)
\Biggl\{
\epsilon^{-2}
\sum_{i=1}^n 
\mathcal M_i(\lambda^*)^2
\\
&\qquad+
\frac{2(\ee^{-\lambda^*\Delta_n}-\ee^{-\lambda_u\Delta_n})}{\epsilon^2}
\sum_{i=1}^n \mathcal M_i(\lambda^*)x(t_{i-1})
\\
&\qquad+
\frac{(\ee^{-\lambda^*\Delta_n}-\ee^{-\lambda_u\Delta_n})^2}{\epsilon^2}
\sum_{i=1}^n x(t_{i-1})^2
\Biggr\}
\\
&\qquad-
\partial_\mu^2 \log F_n(\lambda_u,\mu_u)
\\
&\pto
\frac{\alpha(\alpha-1)}{(\mu^*)^2}+\frac{\alpha}{(\mu^*)^2}
=2H_2(\mu^*)
\end{align*}
uniformly in $u\in[0,1]$. Moreover, it follows
from \eqref{ap-eq-0606}, \eqref{ap-eq-0101}-\eqref{ap-eq-0103}, 
\eqref{ap-eq-0608}, \eqref{ap-eq-0609}, \eqref{ap-eq-0611} and \eqref{ap-eq-0708} that
\begin{align*}
\frac{\epsilon}{\sqrt n}
\partial_\mu\partial_\lambda V_{n,\epsilon}^{(2)}(\lambda_u,\mu_u|x)
&=\frac{\partial_\mu\partial_\lambda F_n(\lambda_u,\mu_u)}{\epsilon\Delta_n\sqrt n}
\sum_{i=1}^n \mathcal M_i(\lambda_u)^2
\\
&\qquad
+\frac{2\ee^{-\lambda_u\Delta_n}\partial_\mu F_n(\lambda_u,\mu_u)}{\epsilon\sqrt n}
\sum_{i=1}^n \mathcal M_i(\lambda_u)x(t_{i-1})
\\
&\qquad
-\epsilon\sqrt n \partial_\mu\partial_\lambda \log F_n(\lambda_u,\mu_u)
\\
&=\frac{\partial_\mu\partial_\lambda F_n(\lambda_u,\mu_u)}{\epsilon\Delta_n\sqrt n}
\sum_{i=1}^n \mathcal M_i(\lambda^*)^2
\\
&\qquad+
\biggl\{
\frac{2\partial_\mu\partial_\lambda F_n(\lambda_u,\mu_u)}{\epsilon\Delta_n\sqrt n}
(\ee^{-\lambda^*\Delta_n}-\ee^{-\lambda_u\Delta_n})
\\
&\qquad\qquad
+\frac{2\ee^{-\lambda_u\Delta_n}\partial_\mu F_n(\lambda_u,\mu_u)}{\epsilon\sqrt n}
\biggr\}
\sum_{i=1}^n 
\mathcal M_i(\lambda^*)x(t_{i-1})
\\
&\qquad+
\biggl\{
\frac{2\partial_\mu\partial_\lambda F_n(\lambda_u,\mu_u)}{\epsilon\Delta_n\sqrt n}
(\ee^{-\lambda^*\Delta_n}-\ee^{-\lambda_u\Delta_n})^2
\\
&\qquad\qquad+
\frac{2\ee^{-\lambda_u\Delta_n}\partial_\mu F_n(\lambda_u,\mu_u)}{\epsilon\sqrt n}
(\ee^{-\lambda^*\Delta_n}-\ee^{-\lambda_u\Delta_n})
\biggr\}
\sum_{i=1}^n x(t_{i-1})^2
\\
&\qquad-
\epsilon\sqrt n \partial_\mu\partial_\lambda \log F_n(\lambda_u,\mu_u)
\\
&\pto0
\end{align*}
uniformly in $u\in[0,1]$. 
This completes the proof of \eqref{ap-eq-0703}.

\textit{Proof of \eqref{ap-eq-0704}.}
From
\eqref{ap-eq-0104}-\eqref{ap-eq-0106} and 
the proof of \eqref{ap-eq-0705}, 
we obtain
\begin{equation*}
-\epsilon\partial_\lambda V_{n,\epsilon}^{(2)}(\lambda^*,\mu^*|x)
=\sum_{i=1}^n \mathcal K_{2,i}(\lambda^*,\mu^*)+o_p(1),
\end{equation*}
where
\begin{align*}
\sum_{i=1}^n \EE[\mathcal K_{2,i}(\lambda^*,\mu^*)|\GG]&=0,
\\
\sum_{i=1}^n \EE[\mathcal K_{2,i}(\lambda^*,\mu^*)^2|\GG]
&\pto4G_2(\lambda^*,\mu^*),
\\
\sum_{i=1}^n \EE[\mathcal K_{2,i}(\lambda^*,\mu^*)^4|\GG]&\pto0.
\end{align*}
Setting
$-\frac{1}{\sqrt n}\partial_\mu V_{n,\epsilon}^{(2)}(\lambda,\mu|x)
=\sum_{i=1}^n \mathcal L_{i}(\lambda,\mu)$, 
one has
\begin{equation*}
\sum_{i=1}^n \mathcal L_{i}(\lambda^*,\mu^*)
=-\frac{\partial_\mu F_n(\lambda^*,\mu^*)}{\epsilon^2\Delta_n\sqrt n}
\sum_{i=1}^n
\biggl(\mathcal M_i(\lambda^*)^2-\frac{\epsilon^2\Delta_n}{F_n(\lambda^*,\mu^*)}\biggr).
\end{equation*}
Noting that
\begin{equation*}
\frac{1}{\epsilon^2 \Delta_n \sqrt n}
\sum_{i=1}^n 
\EE\biggl[\mathcal M_i(\lambda^*)^2-\frac{\epsilon^2\Delta_n}{F_n(\lambda^*,\mu^*)}
\bigg|\GG\biggr]
=0,
\end{equation*}
\begin{align*}
&\frac{1}{(\epsilon^2 \Delta_n \sqrt n)^2}
\sum_{i=1}^n 
\EE\Biggl[
\biggl(
\mathcal M_i(\lambda^*)^2-\frac{\epsilon^2\Delta_n}{F_n(\lambda^*,\mu^*)}
\biggr)^2
\Bigg|\GG\Biggr]
\\
&=
\frac{1}{(\epsilon^2 \Delta_n \sqrt n)^2}
\sum_{i=1}^n 
2\biggl(\frac{\epsilon^2\Delta_n}{F_n(\lambda^*,\mu^*)}\biggr)^2
\pto \frac{2}{(\mu^*)^{2\alpha}},
\end{align*}
\begin{align*}
&\frac{1}{(\epsilon^2 \Delta_n \sqrt n)^4}
\sum_{i=1}^n 
\EE\Biggl[
\biggl(
\mathcal M_i(\lambda^*)^2-\frac{\epsilon^2\Delta_n}{F_n(\lambda^*,\mu^*)}
\biggr)^4
\Bigg|\GG\Biggr]
\\
&\lesssim
\frac{1}{(\epsilon^2 \Delta_n \sqrt n)^4}
\sum_{i=1}^n 
\biggl(\frac{\epsilon^2\Delta_n}{F_n(\lambda^*,\mu^*)}\biggr)^4
\pto 0,
\end{align*}
we see from \eqref{ap-eq-0609} that
\begin{align*}
\sum_{i=1}^n \EE[\mathcal L_{i}(\lambda^*,\mu^*)|\GG]&=0,
\\
\sum_{i=1}^n \EE[\mathcal L_{i}(\lambda^*,\mu^*)^2|\GG]
&\pto 4H_2(\lambda^*,\mu^*),
\\
\sum_{i=1}^n \EE[\mathcal L_{i}(\lambda^*,\mu^*)^4|\GG]&\pto0.
\end{align*}
It follows from
\begin{equation*}
\frac{1}{\epsilon^3 \Delta_n \sqrt n}
\sum_{i=1}^n 
x(t_{i-1})\EE\Biggl[\mathcal M_i(\lambda^*)
\biggl(
\mathcal M_i(\lambda^*)^2-\frac{\epsilon^2\Delta_n}{F_n(\zeta^*)}
\biggr)
\Bigg|\GG\Biggr]=0
\end{equation*}
that
\begin{equation*}
\sum_{i=1}^n \EE[\mathcal K_{2,i}(\lambda^*,\mu^*)\mathcal L_{i}(\lambda^*,\mu^*)|\GG]
=0.
\end{equation*}
Therefore, we obtain that
\begin{equation*}
K_{n,\epsilon}^{(2)}(\lambda^*,\mu^*|x)
=
\begin{pmatrix}
\sum_{i=1}^n \mathcal K_{2,i}(\lambda^*,\mu^*)
\\
\sum_{i=1}^n \mathcal L_{i}(\lambda^*,\mu^*)
\end{pmatrix}
+o_p(1)
\dto N(0,4I_2(\lambda^*,\mu^*)).
\end{equation*}

(1) In the case where $\mu_0$ is known, 
the asymptotic normality of $\tilde\lambda$ can be shown 
in the same way as the proof of (2).
\end{proof}

\section*{Appendix II}  \label{appendix2}

In order to 
illustrate the properties 
of the parameters 
in SPDE \eqref{small_2d_spde}
driven by the $Q_1$-Wiener process defined as \eqref{QWp_ver1},
we generate sample paths
with different values of the parameters.

\subsection*{Characteristic of $\lambda_{1,1}$}

Figures \ref{fig1}-\ref{fig4} show the sample paths
with $\epsilon = 0.01$ and the initial condition $\xi(y,z)=30y(1-y)z(1-z)$.
The rough shape of the sample path depends on the value of $\lambda_{1,1}$.

\begin{en-text}
Figures \ref{fig1}-\ref{fig2} are the sample paths
with 
$\theta=(\theta_0,\theta_1,\eta_1,\theta_2)=(2,0.1,0.1,0.1) $ and $\lambda_{1,1} \approx 0$.
Figure \ref{fig1}-(A) is a cross-section of a sample path at $t=0.1$, 
Figure \ref{fig1}-(B) is a cross-section of a sample path at $t=0.5$ and 
Figure \ref{fig1}-(C) is a cross-section of a sample path at $t=0.9$.
Figure \ref{fig2}-(A) is a cross-section of a sample path at $y=0.5$ and 
Figure \ref{fig2}-(B) is a cross-section of a sample path at $z=0.5$.
\end{en-text}

Figures \ref{fig1}-(A), (B), (C) and
Figures \ref{fig2}-(A), (B) are the cross sections at
$t=0.1, 0.5, 0.9$ and  $y=0.5$, $z=0.5$, respectively. 
For example, the cross section at $y=0.5$ in Figure \ref{fig1}-(A) corresponds to the cross section at $t=0.1$ 
in Figure \ref{fig2}-(A), 
and the cross section at $z=0.5$ in Figure \ref{fig1}-(A) corresponds to the cross section at $t=0.1$ 
in Figure \ref{fig2}-(B).
In case that $\lambda_{1,1}$ is close to 0, 
when $y$ and $z$ are fixed and $t$ is varied, the value of $ X^{Q_1}_t(y,z) $ hardly changes.
When $y=z=0.5$, 
$X^{Q_1}_{0.1}(0.5,0.5) $, $ X^{Q_1}_{0.5}(0.5,0.5) $ and $ X^{Q_1}_{0.9}(0.5,0.5) $ are all close to 2.0.

\begin{figure}[H]
\begin{minipage}{0.32\hsize}
\begin{center}
\includegraphics[width=4.5cm]{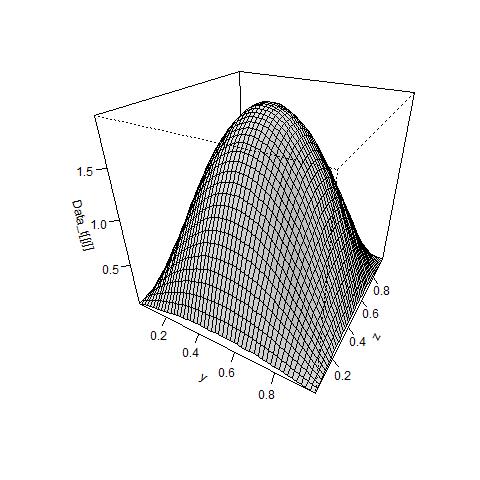}
\captionsetup{labelformat=empty,labelsep=none}
\subcaption{$t = 0.1$}
\end{center}
\end{minipage}
\begin{minipage}{0.32\hsize}
\begin{center}
\includegraphics[width=4.5cm]{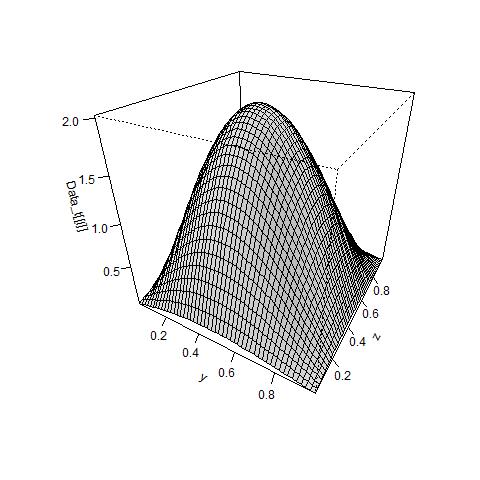}
\captionsetup{labelformat=empty,labelsep=none}
\subcaption{$t = 0.5$}
\end{center}
\end{minipage}
\begin{minipage}{0.32\hsize}
\begin{center}
\includegraphics[width=4.5cm]{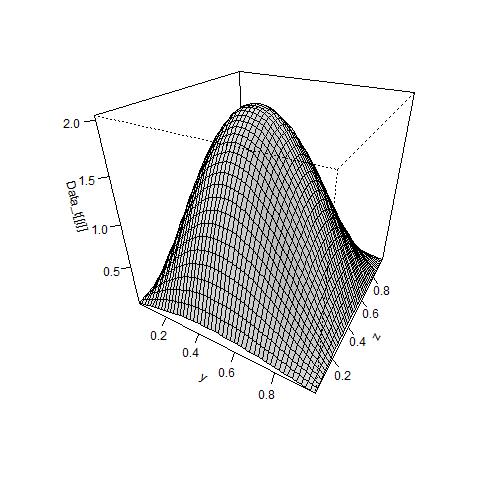}
\captionsetup{labelformat=empty,labelsep=none}
\subcaption{$t = 0.9$}
\end{center}
\end{minipage}
\caption{Cross sections of sample path with $\theta=(2,0.1,0.1,0.1)$, $\epsilon = 0.01$ and $\lambda_{1,1} \approx 0$}
\label{fig1}
\end{figure}

\begin{figure}[H]
\begin{minipage}{0.49\hsize}
\begin{center}
\includegraphics[width=4.5cm]{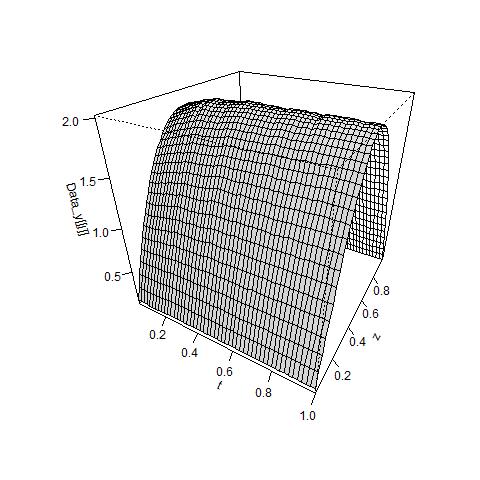}
\captionsetup{labelformat=empty,labelsep=none}
\subcaption{$y = 0.5$}
\end{center}
\end{minipage}
\begin{minipage}{0.49\hsize}
\begin{center}
\includegraphics[width=4.5cm]{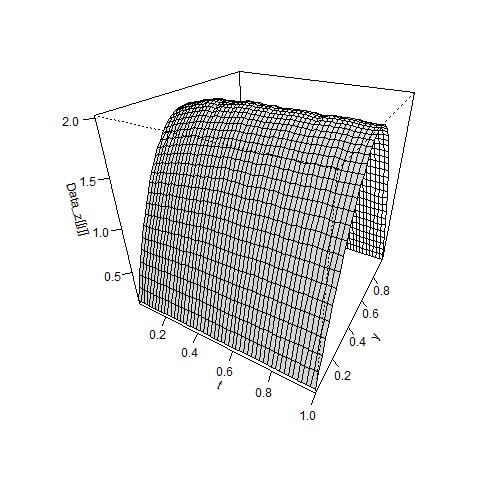}
\captionsetup{labelformat=empty,labelsep=none}
\subcaption{$z = 0.5$}
\end{center}
\end{minipage}
\caption{Cross sections of sample path with $\theta=(2,0.1,0.1,0.1)$, $\epsilon = 0.01$ and $\lambda_{1,1} \approx 0$}
\label{fig2}
\end{figure}

Figures \ref{fig3}-\ref{fig4} are the sample paths
with $\theta =(4,0.3,0.3,0.3) $ and $\lambda_{1,1} \approx 2.07$.
The cross section at $y=0.5$ in Figure \ref{fig3}-(A) corresponds to the cross section at $t=0.1$ 
in Figure \ref{fig4}-(A), 
and the cross section at $z=0.5$ in Figure \ref{fig3}-(A) corresponds to the cross section at $t=0.1$ 
in Figure \ref{fig4}-(B).
For the case that $\lambda_{1,1}$ is positive, when $y$ and $z$ are fixed 
and $t$ tends to $1$, the value of $X^{Q_1}_t(y,z)$ approaches $0$.
For example, when $y=z=0.5$, 
$ X^{Q_1}_{0.1}(0.5,0.5) \approx 1.68 $, $ X^{Q_1}_{0.5}(0.5,0.5) \approx 0.74 $, $ X^{Q_1}_{0.9}(0.5,0.5) \approx 0.33 $.

\begin{figure}[H]
\begin{minipage}{0.32\hsize}
\begin{center}
\includegraphics[width=4.5cm]{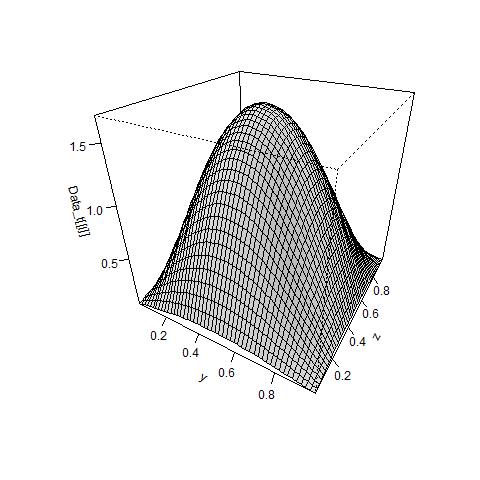}
\captionsetup{labelformat=empty,labelsep=none}
\subcaption{$t = 0.1$}
\end{center}
\end{minipage}
\begin{minipage}{0.32\hsize}
\begin{center}
\includegraphics[width=4.5cm]{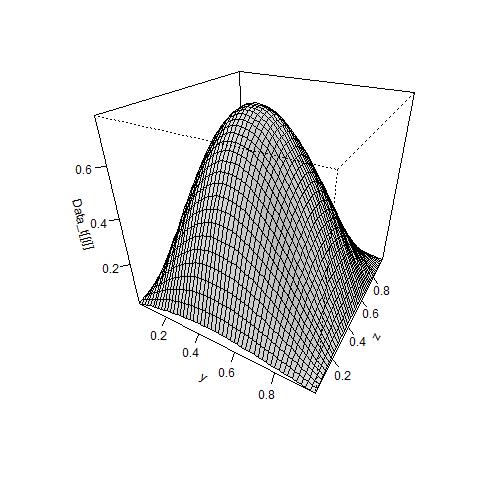}
\captionsetup{labelformat=empty,labelsep=none}
\subcaption{$t = 0.5$}
\end{center}
\end{minipage}
\begin{minipage}{0.32\hsize}
\begin{center}
\includegraphics[width=4.5cm]{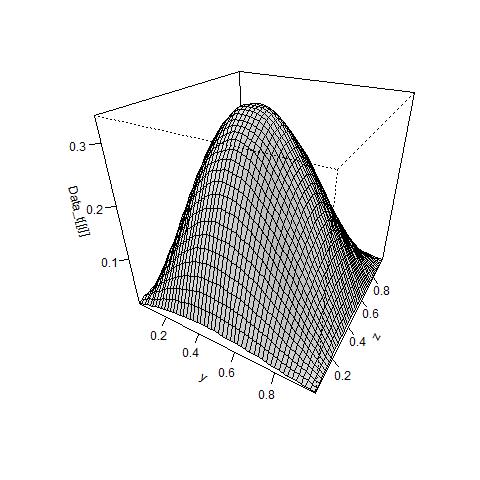}
\captionsetup{labelformat=empty,labelsep=none}
\subcaption{$t = 0.9$}
\end{center}
\end{minipage}
\caption{Cross sections of sample path with $\theta=(4,0.3,0.3,0.3)$, $\epsilon = 0.01$ and $\lambda_{1,1} \approx 2.07$}
\label{fig3}
\end{figure}

\begin{figure}[H]
\begin{minipage}{0.49\hsize}
\begin{center}
\includegraphics[width=4.5cm]{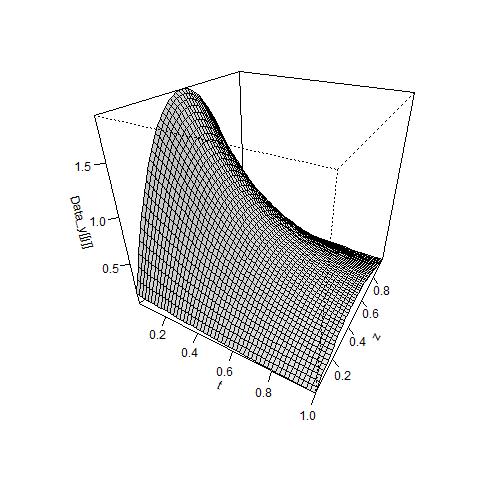}
\captionsetup{labelformat=empty,labelsep=none}
\subcaption{$y = 0.5$}
\end{center}
\end{minipage}
\begin{minipage}{0.49\hsize}
\begin{center}
\includegraphics[width=4.5cm]{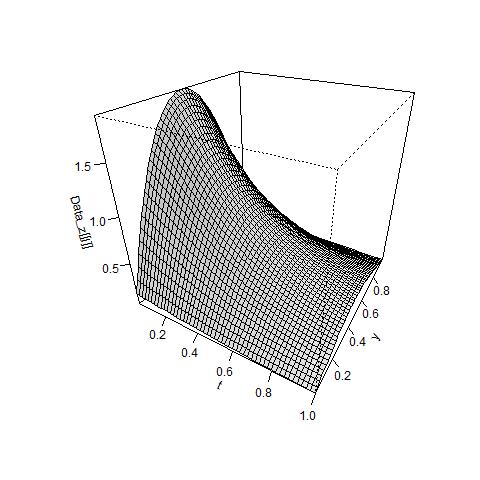}
\captionsetup{labelformat=empty,labelsep=none}
\subcaption{$z = 0.5$}
\end{center}
\end{minipage}
\caption{Cross sections of sample path with $\theta=(4,0.3,0.3,0.3)$, $\epsilon = 0.01$ and $\lambda_{1,1} \approx 2.07$}
\label{fig4}
\end{figure}

\subsection*{Characteristic of $\epsilon$}
Figures \ref{fig5}-\ref{fig7} show the sample paths
with $\theta = (4,0.3,0.3,0.3)$ and the initial condition $\xi(y,z)=30y(1-y)z(1-z)$.
$\theta_0$, $\theta_1$, $\eta_1$ and $\theta_2$ are fixed and only $\epsilon$ is changed.
\begin{en-text}
(A) in Figure \ref{fig5}, Figure \ref{fig6} and Figure \ref{fig7} are sample paths with $\epsilon=0.01$.
(B) in Figure \ref{fig5}, Figure \ref{fig6} and Figure \ref{fig7} are sample paths with $\epsilon=0.25$.
(C) in Figure \ref{fig5}, Figure \ref{fig6} and Figure \ref{fig7} are sample paths with $\epsilon=0.5$.
From Figures \ref{fig5}-\ref{fig7}, it can be seen that the noise increases as $\epsilon$ increases.
\end{en-text}
(A) in Figures \ref{fig5}-\ref{fig7} are cross sections with $\epsilon=0.01$.
(B) in Figures \ref{fig5}-\ref{fig7} are cross sections with $\epsilon=0.25$.
(C) in Figures \ref{fig5}-\ref{fig7} are cross sections with $\epsilon=0.5$.
From Figures \ref{fig5}-\ref{fig7}, it can be seen that the noise increases as $\epsilon$ increases.

\begin{figure}[H]
\begin{minipage}{0.32\hsize}
\begin{center}
\includegraphics[width=4.5cm]{2.jpeg}
\captionsetup{labelformat=empty,labelsep=none}
\subcaption{$\epsilon=0.01$}
\end{center}
\end{minipage}
\begin{minipage}{0.32\hsize}
\begin{center}
\includegraphics[width=4.5cm]{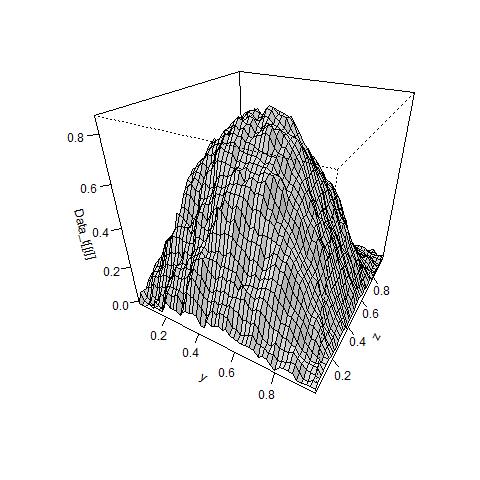}
\captionsetup{labelformat=empty,labelsep=none}
\subcaption{$\epsilon=0.25$}
\end{center}
\end{minipage}
\begin{minipage}{0.32\hsize}
\begin{center}
\includegraphics[width=4.5cm]{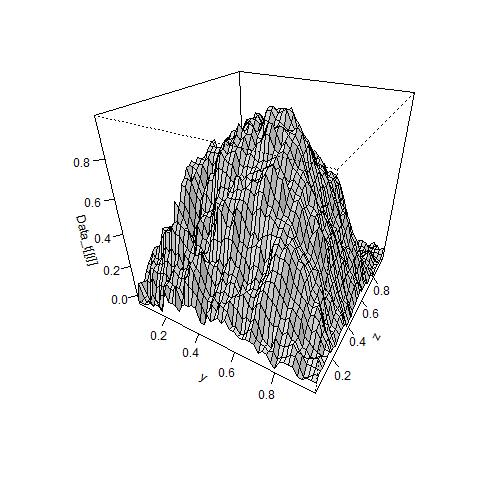}
\captionsetup{labelformat=empty,labelsep=none}
\subcaption{$\epsilon=0.5$}
\end{center}
\end{minipage}
\caption{Cross sections of sample path at $t=0.5$ with $\theta=(4,0.3,0.3,0.3)$ and $\lambda_{1,1} \approx 2.07$}
\label{fig5}
\end{figure}

\begin{figure}[H]
\begin{minipage}{0.32\hsize}
\begin{center}
\includegraphics[width=4.5cm]{4.jpeg}
\captionsetup{labelformat=empty,labelsep=none}
\subcaption{$\epsilon=0.01$}
\end{center}
\end{minipage}
\begin{minipage}{0.32\hsize}
\begin{center}
\includegraphics[width=4.5cm]{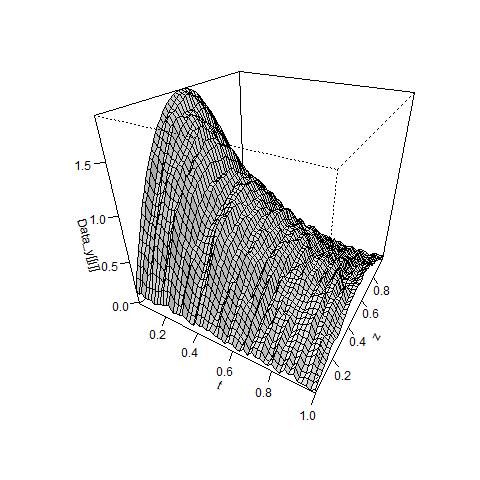}
\captionsetup{labelformat=empty,labelsep=none}
\subcaption{$\epsilon=0.25$}
\end{center}
\end{minipage}
\begin{minipage}{0.32\hsize}
\begin{center}
\includegraphics[width=4.5cm]{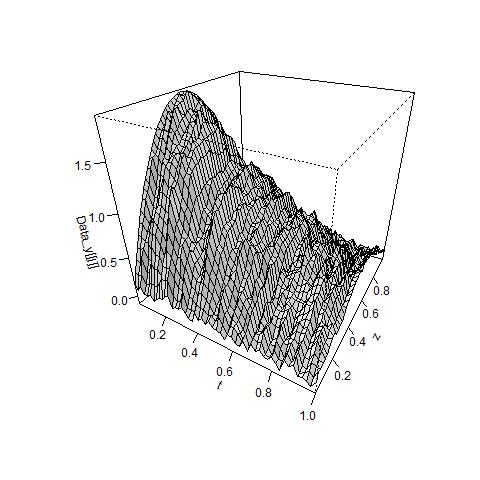}
\captionsetup{labelformat=empty,labelsep=none}
\subcaption{$\epsilon=0.5$}
\end{center}
\end{minipage}
\caption{Cross sections of sample path at $y=0.5$ with $\theta=(4,0.3,0.3,0.3)$ and $\lambda_{1,1} \approx 2.07$}
\label{fig6}
\end{figure}

\begin{figure}[H]
\begin{minipage}{0.32\hsize}
\begin{center}
\includegraphics[width=4.5cm]{5.jpeg}
\captionsetup{labelformat=empty,labelsep=none}
\subcaption{$\epsilon=0.01$}
\end{center}
\end{minipage}
\begin{minipage}{0.32\hsize}
\begin{center}
\includegraphics[width=4.5cm]{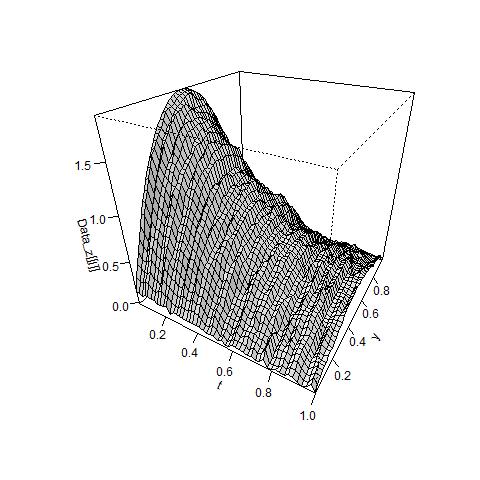}
\captionsetup{labelformat=empty,labelsep=none}
\subcaption{$\epsilon=0.25$}
\end{center}
\end{minipage}
\begin{minipage}{0.32\hsize}
\begin{center}
\includegraphics[width=4.5cm]{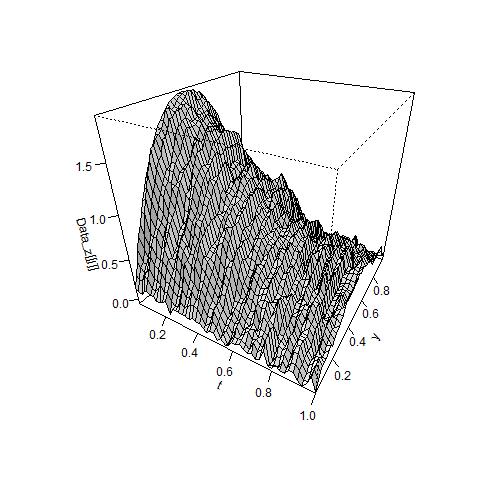}
\captionsetup{labelformat=empty,labelsep=none}
\subcaption{$\epsilon=0.5$}
\end{center}
\end{minipage}
\caption{Cross sections of sample path at $z=0.5$ with $\theta=(4,0.3,0.3,0.3)$ and  $\lambda_{1,1} \approx 2.07$}
\label{fig7}
\end{figure}

\end{document}